\documentclass{amsart}
    \usepackage{amsmath, amssymb, amsthm, graphicx}
    \usepackage{comment}
    \newtheorem{theorem}{Theorem}
    
    \newtheorem{prop}{Proposition}
    \newtheorem{lemma}{Lemma}
    \newtheorem{cor}{Corollary}
    
    \theoremstyle{remark}
    
    \theoremstyle{definition}
    \newtheorem{defn}{Definition}
    \usepackage{verbatim}
    \usepackage{hyperref}
    \usepackage{color}
    \usepackage{bbm}
    \usepackage{tikz}
    \usepackage{appendix}
    \DeclareMathOperator{\cp}{Cap}
    \DeclareMathOperator{\var}{Var}
    
    \usetikzlibrary{calc}
    \tikzset{vtx/.style={circle, fill, inner sep=1.5pt}}
    \tikzset{inlinevtx/.style={circle, fill, inner sep=1pt}}
    \tikzset{openvtx/.style={circle, draw, inner sep=1.5pt}}
    
    \title[Capacity of the range of random walk: Moderate deviations in dimensions $4$ and $5$]{Capacity of the range of random walk: \\ Moderate deviations in dimensions $4$ and $5$}
    \author{Arka Adhikari}
    \address{Arka Adhikari\hfill\break
    Department of Mathematics, University of Maryland-College Park, College Park, MD, USA}
    \email{arkaa@umd.edu}
    
    \author{Jiyun Park}
    \address{Jiyun Park\hfill\break
    Department of Mathematics, Stanford University, Stanford, CA, USA}
    \email{jiyunp@stanford.edu}
    

    \usepackage[ margin=1in]{geometry}
    \begin{document}
    
    \begin{abstract}
        We prove a moderate deviation principle for the capacity of the range of random walk in $\mathbb{Z}^5$. Depending on the scale of deviation, we get two different regimes. We observe Gaussian tails when the deviation scale is smaller than $n^{1/2} (\log n)^{3/4}$. Otherwise, we get non-Gaussian tails with a constant arising from a generalized Gagliardo-Nirenberg inequality. This is analogous to the behavior of the volume of the random walk range in $\mathbb{Z}^3$. Our methods can also be applied to the $d = 4$ case to prove the moderate deviation principle in almost the full range of interest. This extends the work of Okada and the first author \cite{AdhikariOkada2023}, where they showed moderate deviations up to a deviation scale of $\log \log n$ times the standard deviation.
    \end{abstract}
    
    \maketitle
    
    \section{Introduction} \label{sec:introduction}
    
    Given a finite set $A$ in $\mathbb{Z}^d$, let $\tau_A$ denote the first positive hitting time of $A$ by a simple random walk, hereafter denoted $(S_i)_{i \ge 0}$. When $d \ge 3$, the corresponding (Newtonian) capacity is defined by
    \[  
    \cp (A) := \sum_{x \in A} \mathbb{P}^x (\tau_A = \infty) = \lim_{|z| \to \infty} \frac{\mathbb{P}^z (\tau_A < \infty)}{G_D(z)}.
    \]
    Here, $G_D$ denotes the Green's function for the random walk, and $| \cdot |$ is the Euclidean distance. We note that $G_D (x) \asymp |x|_+^{-(d-2)}$, where $|x|_+ = \max (|x|, 1)$.
    
    In recent years, there have been much interest in the asymptotics of the capacity of the random walk range, $S[0, n] = \{ S_0 , S_1 , \dots , S_n \}$ \cite{AsselahSchapiraSousi2018, AsselahSchapiraSousi2019, Schapira2020, DemboOkada2024, AsselahSchapira2020, AdhikariOkada2023}. In addition to inherent interest, the capacity of $S[0, n]$ has many connections with intersections of random walks (e.g., \cite{Lawler2013, AsselahSchapira2021}). These results in turn play a key role in analyzing numerous self-interacting random walks. For instance, they are used in renormalization group methods in quantum field theory (e.g., \cite{FernandezFroehlichSokal1992}) and in random polymer models such as the self-avoiding walk (e.g., \cite{MadrasSlade1993}).
    
    \subsection{Main Results} \label{subset:results}
    
    Our main result is a moderate deviation principle for the capacity of the random walk range in $\mathbb{Z}^5$. This is the first example of such a theorem for $d = 5$; previous works only obtained the correct order, in a limited range of $b_n$ \cite{AsselahSchapira2020}. Depending on the size of  the deviation, we get two different regimes. When the deviation scale $b_n$ is relatively small, we see Gaussian tail behavior.
    
    \begin{theorem}
    \label{thm:gaussian}
    Suppose $1 \ll b_n \ll \sqrt{\log n}$ \footnote{$f \ll g$ means $f/g \to 0$ as $n \to \infty$. For a full explanation of our notation, see Section~\ref{subsec:notation}.}. Then, for any $\lambda > 0$,
    \[  
    \lim_{n \to \infty} \frac{1}{b_n} \log \mathbb{P} \left( \pm 
     \left\{ \cp (S[0, n]) - \mathbb{E}\cp ( S[0, n]) \right\} \ge \lambda \sqrt{n b_n \log n} \right) = -\frac{\lambda^2}{2\sigma^2}
    \]
    where $\sigma^2 n \log n$ is the variance of $\cp(S[0, n])$.
    \end{theorem}
    
    The dynamics change when $b_n$ gets larger. Since $\cp (S[0, n])$ is non-negative and and $\mathbb{E} \cp S[0, n] \asymp n$, we can see that $1 \ll b_n \ll n^{1/3}$ is the full range of interest. We prove a large deviation principle in almost this full range, in the sense that our theorem holds for all $\lambda > 0$ and $b_n = n^{1-\epsilon}$, where $0 < \epsilon < 1$. This is optimal since when $\epsilon = 0$, such a theorem would not exist for large $\lambda$, and the case $\epsilon = 1$ is the central limit theorem. More precisely, our results hold as long as $\sqrt{\log n} \ll b_n \ll n^{1/3}(\log n)^{-8/3}$ (the case $1 \ll b_n \ll \sqrt{\log n}$ was dealt with in Theorem~\ref{thm:gaussian}).
    
    \begin{theorem}
    \label{thm:non-gaussian}
    Suppose $\sqrt{\log n} \ll b_n \ll n^{1/3} (\log n)^{-8/3}$. Then, for any $\lambda > 0$,
    \[  
    \lim_{n \to \infty} \frac{1}{b_n} \log \mathbb{P} \left( 
     \cp (S[0, n]) - \mathbb{E}\cp ( S[0, n]) \le -\lambda \sqrt{n b_n^3} \right) = - I_5(\lambda),
    \]
    where $I_5(\lambda) = \frac{1}{2} \left( 5^{-5/2}  \tilde{\gamma}_S^{-2} \tilde{\kappa}(5,2)^{-4} \lambda \right)^{2/3}$ and
    \[
    \tilde{\gamma}_S = \mathbb{P}(0 \notin S^1[1, \infty), \{S^1[0, \infty) \cup S^2[0, \infty)\} \cap S^3[1, \infty) = \emptyset)
    \]
    for three independent random walks $S^1, S^2$, and $S^3$, and $ \tilde{\kappa}(d, 2)$ is the smallest constant such that the following generalized Gagliardo--Nirenberg inequality holds.
    \begin{equation}
    \label{eq:general-GN}
        \left[ \int_{(\mathbb{R}^d)^2} g^2(x) G(x-y) g^2(y) \mathrm{d}x \mathrm{d}y \right]^{1/4} \le \tilde{\kappa}(d,2) \left( \int_{\mathbb{R}^5} g^2(x) \mathrm{d}x \right)^{\frac{6-d}{8}} \left( \int_{\mathbb{R}^d} |\nabla g(x)|^2 \mathrm{d}x \right)^{\frac{d-2}{8}}.
    \end{equation}
    $G(x)$ is the Green's function for the Brownian motion on $\mathbb{R}^d$. For $d = 5$, this would be $G(x) = \frac{1}{4\pi^2}|x|^{-3}$.
    \end{theorem}
    
    Previously, Asselah and Schapira \cite{AsselahSchapira2020} showed that Theorem~\ref{thm:non-gaussian} holds up to constant factors in the range $n^{1/7} (\log n)^{2/3} \le b_n \le \epsilon n^{1/3}$. Our paper is the first to reach the Gaussian regime for $d = 5$, and hence also the first to observe the transition at $b_n \asymp \sqrt{\log n}$. Furthermore, we have determined the exact value of the constant determining the size of the moderate deviations.
    
    Along the way, we also show upper tail deviations of the \emph{cross term} $\chi(S[0, n], \tilde{S}[0, n])$ of two independent random walks, for all $\lambda > 0$ and $b_n = n^{1 - \epsilon}$, where $0 < \epsilon < 1$.(  Again, such a result is not reasonable for $\epsilon = 0$ when $\lambda$ is large, and $\epsilon = 1$ would fall into the weak convergence regime.)
    
       % When studying the capacity, a key quantity is the \emph{cross term} $\chi$, defined by
       The cross term can be explicitly expressed as follows:
    \[  
    \chi(A, B) = \sum_{x \in A} \sum_{y \in B} \mathbb{P}^x(\tau_A = \infty) G_D(x, y) \mathbb{P}^y(\tau_B = \infty).
    \]
    This term plays a role similar to that of the intersection of sets when analyzing the volume. However, the cross-term involves long-term correlations between random walks, making it much harder to analyze. Furthermore, compared to similar work done in lower dimensions \cite{AdhikariOkada2023}, our results require much more precise control over the moments of highly singular random variables. Because there are many sources of blow-ups, we have to account for their specific combinatorial structure. To better manage this, we introduce a novel graph representation, which allows us to perform induction over a variety of possible states. This new technique is what enabled us to improve the results of \cite{AdhikariOkada2023} to nearly the full range as stated in Theorem~\ref{thm:dim-4}. We believe this representation may have widespread applications beyond our setting; see Section~\ref{subsec:graph-intro} for potential applications and Section~\ref{subsec:moments-sketch} for how they are used in this paper.    
  
    \begin{theorem}
        \label{thm:cross-term-LDP}
        Suppose $1 \ll b_n \ll n^{1/3} (\log n)^{-8/3}$. Then, for any $\lambda > 0$,
        \[  
        \lim_{n \to \infty} \frac{1}{b_n} \log \mathbb{P} \left( \chi(S[0, n], \tilde{S}[0, n]) \ge \lambda \sqrt{n b_n^3} \right) = -2I_5(\lambda),
        \]
        where $S$, $\tilde{S}$ are two independent random walks and
        \[  
        \chi(A, B) = \sum_{x \in A} \sum_{y \in B} \mathbb{P}^x(\tau_A = \infty) G_D(x - y) \mathbb{P}^y(\tau_B = \infty).
        \]
    \end{theorem}
    
    Furthermore, the techniques developed in this paper can also be applied to the $d = 4$ case to obtain the moderate deviation principle for a wide range of $b_n$. This is an improvement over the work of Okada and the first author \cite{AdhikariOkada2023}, in which they showed the same result but only for $1 \ll b_n =O( \log \log n)$. Since $\mathbb{E} \cp (S[0, n]) \asymp n / \log n$, the full range of interest is $1 \ll b_n \ll \log n$. Of this, the following theorem covers the cases where $b_n = (\log n)^{1- \epsilon}$ for $0 < \epsilon < 1$.
    
    \begin{theorem}
        \label{thm:dim-4}
        Let $d = 4$. For any $b_n$ satisfying $1 \ll b_n \ll \frac{\log n}{\log \log n}$ and $\lambda > 0$, the following are true.
        \begin{gather}      
            \lim_{n \to \infty} \frac{1}{b_n} \log \mathbb{P} \left( \cp (S[0, n]) - \mathbb{E} \cp (S[0, n]) \le - \frac{\lambda nb_n}{(\log n)^2} \right) = -I_4 (\lambda) \\         
            \lim_{n \to \infty} \frac{1}{b_n} \log \mathbb{P} \left( \chi(S[0,n], \tilde{S}[0,n]) \ge  \frac{\lambda nb_n}{(\log n)^2}\right) = -2I_4 (\lambda) \\
            \lim_{n \to \infty} \frac{1}{b_n} \log \mathbb{P} \left( \chi(S[0,n], S[0,n]) -\mathbb{E}\chi(S[0,n], S[0,n]) \ge \frac{\lambda nb_n}{(\log n)^2} \right) = -I_4 (\lambda)
        \end{gather}
        These theorems correspond to Theorem~1, Theorem~2, and Corollary~2 of \cite{AdhikariOkada2023}, respectively.  Here, $I_4(\lambda) = 2\pi^{-4} \tilde{\kappa}(4,2)^{-4} \lambda$.
    \end{theorem}
    The constant term comes from the fact that, for a simple random walk in $d = 4$,
    \[
        \mathbb{P}(0 \notin S^1[1, n), \{S^1[0, n) \cup S^2[0, n)\} \cap S^3[1, \infty) = \emptyset) \sim \frac{\pi^2}{8} \log n.
    \]
    We note that our definition of $\chi$ differs from that of \cite{AdhikariOkada2023}, which is what leads to the change of constants from one paper to another. To be more precise, their definition of $\chi$ corresponds with $\chi_{\mathcal{C}}$ in this paper, and they used $\chi'$ to denote our definition of $\chi$.
    
    \subsection{Related Work} \label{subsec:prev-works}
    There is a long series of works devoted to understanding the capacity of the range of a random walk. A classical result of Jain and Orey \cite{JainOrey1968} proves a strong law of large numbers for $d \ge 3$: almost surely,
    \[  
    \lim_{n \to \infty} \frac{ \cp (S[0, n])}{n} = \gamma_d, \quad \gamma_d > 0 \text{ if and only if }  d \ge 5.
    \]
     More recently, Chang \cite{Chang2017} obtained the first order asymptotics for $d = 3$, in which $\cp(S[0, n])$ scales like $\sqrt{n}$. For $d = 4$, a non-Gaussian central limit theorem was shown by Asselah, Schapira, and Sousi \cite{AsselahSchapiraSousi2019}. In all higher dimensions, we have the usual central limit theorem, as shown by Schapira \cite{Schapira2020} for $d = 5$ and Asselah, Schapira, and Sousi \cite{AsselahSchapiraSousi2018} for $d \ge 6$. Dembo and Okada \cite{DemboOkada2024} proved the law of iterated logarithms for all $d \ge 3$. Okada and the first author \cite{AdhikariOkada2023} proved a moderate deviations principle in $d = 4$, as mentioned above. They also connected the constant to a generalized Gagliardo--Nirenberg inequality. In higher dimensions $d \ge 7$, Asselah and Schapira \cite{AsselahSchapira2020} proved a moderate deviation principle in the Gaussian regime, as well as finding the correct order in the non-Gaussian regime. For $d = 5$, however, they only obtain the correct order, and only in the non-Gaussian case.
    
    One remarkable property of the capacity is that it behaves very similarly to the \emph{volume} of $S[0, n]$ in dimension $d-2$. For instance, the strong law of large numbers for the volume of $S[0, n]$ is the following.
    \[  
    \lim_{n \to \infty} \frac{|S[0, n]|}{n}  = \alpha_d,\quad  \alpha_d > 0 \text{ if and only if }  d \ge 3.
    \]
    This is almost identical to the strong law of the capacity, except for a difference of $2$ in the dimensions. In fact, such a resemblance is present for \emph{all} of the results discussed in the previous paragraph. The only case where they differ by something other than a constant factor is in \cite{DemboOkada2024}, where the $\limsup$ for the capacity has an additional factor of $(\log \log \log n)^{-1}$. For more information on the asymptotics for the volume, see \cite{Chen2010} and the references therein. Our results strengthen this correspondence by showing that the capacity exhibits a similar transition from Gaussian tails to stretched exponential tails at the same deviation scale as the volume for $d = 3$. We also see a relation between the constant term and the generalized Gagliardo-Nirenberg inequality, similar to how the usual Gagliardo-Nirenberg inequality appears in the volume case.

    The cross term is also related to the Riesz potential of the Brownian motion $\int_0^1 \int_0^1 G(B_t - \tilde{B}_s) \mathrm{d}t \mathrm{d}s$, of which the large deviations were shown by Bass, Chen, and Rosen \cite{BassChenRosen2009} (stated here  as Theorem~\ref{thm:brownian-LDP}).  Indeed, we show that these two quantities behave the same way in the moderate deviation regime. However, when $d \ge 6$, this is no longer conjectured to be true. This fact, along with the exponents on the right-hand side of \eqref{eq:general-GN}, indicate that a fundamentally different approach has to be taken when $d \ge 6$.

    \subsection{Graphical Representation} \label{subsec:graph-intro}
    In order to resolve the difficulties mentioned in the previous section, we introduce a new graphical representation. For instance, the value $|x-y|_+^{-2}$ is represented by the graph $[\small
    \tikz[scale = 0.5, baseline=(current bounding box.center)]{
    \coordinate[inlinevtx] (x) at (0, 0);
    \coordinate[inlinevtx] (y) at (1.5, 0);
    \node at (x) [left] {$x$};
    \node at (y) [right] {$y$};
    \draw[thick] (x) -- (y) node [midway, above] { $2$};
    }]$. One example of a graph we study is the following, where $S$ and $\tilde{S}$ are independent random walks.
    \[ \mathbb{E} \left[ \prod_{k=1}^m |S_{i_k} - \tilde{S}_{j_{2k-1}}|_+^{-3} |S_{i_k} - \tilde{S}_{j_{2k}}|_+^{-3} \right] = \mathbb{E} \left[
    \begin{tikzpicture}[scale=0.7, baseline=(current bounding box.center)]
        \coordinate[vtx] (u0) at (1.5, 0);
        \node at (u0) [below] {$S_{i_1}$};
        \coordinate[vtx] (v0) at (1, 2);
        \node at (v0) [above] {$\tilde{S}_{j_1}$};
        \coordinate[vtx] (w0) at (2, 2);
        \node at (w0) [above] {$\tilde{S}_{j_2}$};
        \draw[thick] (v0) -- (u0) node [midway, left] {$3$};
        \draw[thick] (w0) -- (u0) node [midway, right] {$3$};
        \coordinate[vtx] (u1) at (3.5, 0);
        \node at (u1) [below] {$S_{i_2}$};
        \coordinate[vtx] (v1) at (3, 2);
        \node at (v1) [above] {$\tilde{S}_{j_3}$};
        \coordinate[vtx] (w1) at (4, 2);
        \node at (w1) [above] {$\tilde{S}_{j_4}$};
        \draw[thick] (v1) -- (u1) node [midway, left] {$3$};
        \draw[thick] (w1) -- (u1) node [midway, right] {$3$};
        \node at (5, 1) {$\dots$};
        \coordinate[vtx] (u) at (6.5, 0);
        \node at (u) [below] {$S_{i_{m}}$};
            \coordinate[vtx] (v) at (6, 2);
            \node at (v) [above] {$\tilde{S}_{j_{2m-1}}$};
            \coordinate[vtx] (w) at (7, 2);
            \node at (w) [above] {$\tilde{S}_{j_{2m}}$};
            \draw[thick] (v) -- (u) node [midway, left] {$3$};
            \draw[thick] (w) -- (u) node [midway, right] {$3$};
    \end{tikzpicture}\right]
    \]
    Such representations make it easier to handle complicated correlations between different parts of the random walks, and also to perform induction over a variety of possible states; see Section~\ref{subsec:moments-sketch} for more details. As such, this representation has the potential to be applicable to a large number of problems that require moment computations with nontrivial self-interactions. Indeed, Chen \cite{Chen2010} lists open questions in Conjecture 8.7.3 and 8.7.4 about the large deviation behavior of the self-intersections and volume of the random walk in dimensions $d \ge 5$ and $d=4$, respectively. As of yet, current methods were not able to obtain these conjectures in all regimes; in addition, the computational methods of Chen in \cite{Chen2010} were not yet able to adapt to these higher-dimensional cases. We believe that our graphical notation gives us a new method to improve computations of these moments in higher dimensions and analyze these asymptotics in a tractable fashion. We present two possible scenarios where we believe our representation may be useful. The first concerns branching random walks, and the second is related to the diagrammatic notation used in percolation theory.
    
    Beyond the context of intersections of simple random walks, we believe that this graphical representation is key to expressing and computing intersection events of  branching random walks involving multiple distinct points. In what follows, we will derive an explicit formula for a  relatively simple intersection event. We will observe that even this computation for the probability is associated with a complex expression that is better suited to our graphical model. Thus, our graphical representation becomes even more efficacious when considering intersection events that would appear in higher moment computations.
    
    Let us now turn to describing our simpler intersection event. Consider a branching random walk $B$, with branching probability $\rho$, run until time $n$ that branched only once at time $t$; we define $B^{1}_{[t,n]}$ to be one branch of $B$ after the branching time and $B^2_{[t,n]}$ to be the other branch. $B_{[0,t]}$ will denote the unique part of the walk before the branching. Let us suppose we want to compute the probability that this branching walk $B$ intersects the random walk $S$ at three points; with respect to $B$, we will have exactly one intersection on each of the parts $B_{[0,t]}, B^1_{[t,n]}$, and $B^2_{[t,n]}$. We let $t_1 <t$ be the time of the first intersection on $B_{[0,t]}$, $t_2$ and $t_3$ be the time of the intersections on $B^1$ and $B^2$, respectively, and $x$ be the value of the random walk $B$ at the branching point. Furthermore, let $S_{a}$ be the point in which $B_{[0,t]}$ intersects $S$, $S_{b}$ be the point in which $B^1_{[t,n]}$ intersects $S$, $S_c$ be the point in which $B^2_{[t,n]}$ intersects $S$. Conditionally on the random walk $S$, we can write the probability that the aforementioned intersection event happens as:
    \begin{equation}
    \begin{aligned}
    &\rho (1-\rho)^{n-1}\sum_{t_1,t,t_2,t_3,x} p_{t_1}(S_a) p_{t-t_1}(x- S_a) p_{t_2- t}(S_b-x) p_{t_3- t}(S_c-x) \\& \approx   \rho (1-\rho)^{n-1}\sum_x G(S_a) G(x-S_a) G(S_b-x) G(S_c-x),
    \end{aligned}
    \end{equation} where the approximation improves as $n$ gets larger.  Expressions such as those of the last line above cannot be computed without the great simplifications afforded by our novel graphical representation. This should compare to papers such as \cite{Angeletal2021} which analyzes the local time of a Branching random walk, but only at the origin, the paper \cite{AsselahSchapira24} that analyzes the local time in a ball around the origin, or \cite{Asselahetal24} on the branching capacity. We believe that our graphical method could obtain exact asymptotics for general expressions regarding intersections of branching random walks.
    
Furthermore, we believe that the tools we have developed here could lead to insights even in topics outside the domain of random walk intersections. For example, one can consider the diagrammatic notation of the lace expansion (see \cite{Slade2006} for a survey). In those settings, an unlabeled vertex $\tikz{\coordinate[vtx] (x) at (0, 0); }$ represents summation over all points, and edges corresponds to the \emph{two-point function} $T(x, y)$. For instance, the \emph{triangle diagram}, a widely studied object in percolation theory, is the following.
    \[
        \label{eq:triangle-diagram}
        \nabla := 
        \begin{tikzpicture}[scale=0.7, baseline=(current bounding box.center)]
            \coordinate[vtx] (0) at (0, 0);
            \coordinate[vtx] (x) at (-1, 1);
            \coordinate[vtx] (y) at (1, 1);
            \node at (0) [below] {$0$};
            \draw (0) -- (x);
            \draw (x) -- (y);
            \draw (0) -- (y);
        \end{tikzpicture}
        =
        \sum_{x, y \in \mathbb{Z}^d} T(0, x) T(x, y) T(y, 0)
    \]

    Compared to previous diagrammatic estimates, one novelty of our technique is that we consider \emph{random} vertices. This method can be used to simplify many diagrams beyond those we use in this paper. Namely, if we think of $T(x, y)$ as the (pre-normalized) transition probability from $x$ to $y$, we may remove an edge from this diagram in exchange for randomizing one of its endpoints. Thus,
    \[  
    \begin{tikzpicture}[scale=0.7, baseline=(current bounding box.center)]
            \coordinate[vtx] (0) at (0, 0);
            \coordinate[vtx] (x) at (-1, 1);
            \coordinate[vtx] (y) at (1, 1);
            \node at (0) [below] {$0$};
            \draw (0) -- (x);
            \draw (x) -- (y);
            \draw (0) -- (y);
        \end{tikzpicture} = \mathcal{X} \mathbb{E} \left[
        \begin{tikzpicture}[scale=0.7, baseline=(current bounding box.center)]
            \coordinate[vtx] (0) at (0, 0);
            \coordinate[vtx] (x) at (-1, 1);
            \coordinate[vtx] (y) at (1, 1);
            \node at (0) [below] {$0$};
            \node at (x) [above] {$T_1$};
            \draw (x) -- (y);
            \draw (0) -- (y);
        \end{tikzpicture} \right] = \mathcal{X}^2 \mathbb{E} \left[
        \begin{tikzpicture}[scale=0.7, baseline=(current bounding box.center)]
            \coordinate[vtx] (0) at (0, 0);
            \coordinate[vtx] (x) at (-1, 1);
            \coordinate[vtx] (y) at (1, 1);
            \node at (0) [below] {$0$};
            \node at (x) [above] {$T_1$};
            \node at (y) [above] {$T_2$};
            \draw (0) -- (y);
        \end{tikzpicture} \right] = \mathcal{X}^2 \mathbb{E} \left[
        \begin{tikzpicture}[scale=0.7, baseline=(current bounding box.center)]
            \coordinate[vtx] (0) at (0, 0);
            \coordinate[vtx] (y) at (2, 0);
            \node at (0) [left] {$0$};
            \node at (y) [right] {$T_2$};
            \draw (0) -- (y);
        \end{tikzpicture} \right],
    \]
    where $T_1$ has distribution $\mathcal{X}^{-1} T(0, \cdot)$ and $T_2$ has distribution $\mathcal{X}^{-1} T(T_1, \cdot)$. Here, the normalizing constant is the \emph{susceptibility}, $\mathcal{X} = \sum_{x \in \mathbb{Z}^d} T(0, x)$.
    For another example, the following inequality is one of the main results of a paper by Hutchcroft on mean-field percolation \cite[Proposition~1.8]{Hutchcroft2022}.
    \begin{equation}
    \label{eq:hutchcroft-inequality}
    \begin{tikzpicture}[scale=0.7, baseline=(current bounding box.center)]
        \coordinate[vtx] (a) at (0, 0);
        \coordinate[vtx] (b) at (-1, 0);
        \coordinate[vtx] (c) at (1, 0);
        \coordinate[vtx] (d) at (0, 1);
        \coordinate[vtx] (e) at (0, -1);
        \node at (a) [left] {$0$};
        \draw (a) -- (d);
        \draw (a) -- (e);
        \draw (b) -- (d);
        \draw (b) -- (e);
        \draw (c) -- (d);
        \draw (c) -- (e);
    \end{tikzpicture} \quad \le \quad
    \begin{tikzpicture}[scale=0.7, baseline=(current bounding box.center)]
        \coordinate[vtx] (a) at (0, 0);
        \coordinate[vtx] (b) at (-1, 0);
        \coordinate[vtx] (c) at (-1, 1);
        \coordinate[vtx] (d) at (0, 1);
        \coordinate[vtx] (e) at (1, 1);
        \node at (b) [left] {$0$};
        \draw (a) -- (b);
        \draw (b) -- (c);
        \draw (c) -- (d);
        \draw (d) -- (a);
        \draw (d) -- (e);
        \draw (a) -- (e);
    \end{tikzpicture} \quad \le \quad
    \begin{tikzpicture}[scale=0.7, baseline=(current bounding box.center)]
        \coordinate[vtx] (a) at (0, 0);
        \coordinate[vtx] (b) at (-1, 0);
        \coordinate[vtx] (c) at (-1, 1);
        \coordinate[vtx] (d) at (1, 0);
        \coordinate[vtx] (e) at (1, 1);
        \node at (b) [left] {$0$};
        \draw (a) -- (b);
        \draw (b) -- (c);
        \draw (c) -- (a);
        \draw (a) -- (d);
        \draw (d) -- (e);
        \draw (e) -- (a);
    \end{tikzpicture} 
    \end{equation}
    By choosing an appropriate path and proceeding as before, we can rewrite each diagram as follows.
    \begin{align*}
    \begin{tikzpicture}[scale=0.7, baseline=(current bounding box.center)]
        \coordinate[vtx] (a) at (0, 0);
        \coordinate[vtx] (b) at (-1, 0);
        \coordinate[vtx] (c) at (1, 0);
        \coordinate[vtx] (d) at (0, 1);
        \coordinate[vtx] (e) at (0, -1);
        \node at (a) [left] {$0$};
        \draw (a) -- (d);
        \draw (a) -- (e);
        \draw (b) -- (d);
        \draw (b) -- (e);
        \draw (c) -- (d);
        \draw (c) -- (e);
    \end{tikzpicture} &= \mathcal{X}^4 \mathbb{E} \left[ \begin{tikzpicture}[scale=0.7, baseline=(current bounding box.center)]
        \coordinate[vtx] (a) at (0, 0);
        \coordinate[vtx] (b) at (2, 0);
        \coordinate[vtx] (c) at (0, -1);
        \coordinate[vtx] (d) at (2, -1);
        \node at (a) [left] {$0$};
        \node at (b) [right] {$T_3$};
        \node at (c) [left] {$T_1$};
        \node at (d) [right] {$T_4$};
        \draw (a) -- (b);
        \draw (c) -- (d);
    \end{tikzpicture} \right] \\
    \begin{tikzpicture}[scale=0.7, baseline=(current bounding box.center)]
        \coordinate[vtx] (a) at (0, 0);
        \coordinate[vtx] (b) at (-1, 0);
        \coordinate[vtx] (c) at (-1, 1);
        \coordinate[vtx] (d) at (0, 1);
        \coordinate[vtx] (e) at (1, 1);
        \node at (b) [left] {$0$};
        \draw (a) -- (b);
        \draw (b) -- (c);
        \draw (c) -- (d);
        \draw (d) -- (a);
        \draw (d) -- (e);
        \draw (a) -- (e);
    \end{tikzpicture} &= \mathcal{X}^4 \mathbb{E} \left[ \begin{tikzpicture}[scale=0.7, baseline=(current bounding box.center)]
        \coordinate[vtx] (a) at (0, 0);
        \coordinate[vtx] (b) at (2, 0);
        \coordinate[vtx] (c) at (0, -1);
        \coordinate[vtx] (d) at (2, -1);
        \node at (a) [left] {$0$};
        \node at (b) [right] {$T_3$};
        \node at (c) [left] {$T_2$};
        \node at (d) [right] {$T_4$};
        \draw (a) -- (b);
        \draw (c) -- (d);
    \end{tikzpicture} \right] = \mathcal{X}^4 \mathbb{E} \left[ \begin{tikzpicture}[scale=0.7, baseline=(current bounding box.center)]
        \coordinate[vtx] (a) at (0, 0);
        \coordinate[vtx] (b) at (1.5, 0);
        \coordinate[vtx] (c) at (3, 0);
        \node at (a) [left] {$0$};
        \node at (b) [below] {$T_4$};
        \node at (c) [below] {$T_2$};
        \draw (a) -- (b) -- (c);
    \end{tikzpicture} \right] \\
    \begin{tikzpicture}[scale=0.7, baseline=(current bounding box.center)]
        \coordinate[vtx] (a) at (0, 0);
        \coordinate[vtx] (b) at (-1, 0);
        \coordinate[vtx] (c) at (-1, 1);
        \coordinate[vtx] (d) at (1, 0);
        \coordinate[vtx] (e) at (1, 1);
        \node at (b) [left] {$0$};
        \draw (a) -- (b);
        \draw (b) -- (c);
        \draw (c) -- (a);
        \draw (a) -- (d);
        \draw (d) -- (e);
        \draw (e) -- (a);
    \end{tikzpicture}  &= \mathcal{X}^4 \mathbb{E} \left[ \begin{tikzpicture}[scale=0.7, baseline=(current bounding box.center)]
        \coordinate[vtx] (a) at (0, 0);
        \coordinate[vtx] (b) at (1.5, 0);
        \coordinate[vtx] (c) at (3, 0);
        \node at (a) [left] {$0$};
        \node at (b) [below] {$T_2$};
        \node at (c) [below] {$T_4$};
        \draw (a) -- (b) -- (c);
    \end{tikzpicture} \right]  
    \end{align*}
    Thus, the inequality \eqref{eq:hutchcroft-inequality} is equivalent to the following.
    \[  
    \mathbb{E} \left[ \begin{tikzpicture}[scale=0.7, baseline=(current bounding box.center)]
        \coordinate[vtx] (a) at (0, 0);
        \coordinate[vtx] (b) at (2, 0);
        \coordinate[vtx] (c) at (0, -1);
        \coordinate[vtx] (d) at (2, -1);
        \node at (a) [left] {$0$};
        \node at (b) [right] {$T_3$};
        \node at (c) [left] {$T_1$};
        \node at (d) [right] {$T_4$};
        \draw (a) -- (b);
        \draw (c) -- (d);
    \end{tikzpicture} \right] \le \mathbb{E} \left[ \begin{tikzpicture}[scale=0.7, baseline=(current bounding box.center)]
        \coordinate[vtx] (a) at (0, 0);
        \coordinate[vtx] (b) at (2, 0);
        \coordinate[vtx] (c) at (0, -1);
        \coordinate[vtx] (d) at (2, -1);
        \node at (a) [left] {$0$};
        \node at (b) [right] {$T_3$};
        \node at (c) [left] {$T_2$};
        \node at (d) [right] {$T_4$};
        \draw (a) -- (b);
        \draw (c) -- (d);
    \end{tikzpicture} \right] = \mathbb{E} \left[ \begin{tikzpicture}[scale=0.7, baseline=(current bounding box.center)]
        \coordinate[vtx] (a) at (0, 0);
        \coordinate[vtx] (b) at (1.5, 0);
        \coordinate[vtx] (c) at (3, 0);
        \node at (a) [left] {$0$};
        \node at (b) [below] {$T_4$};
        \node at (c) [below] {$T_2$};
        \draw (a) -- (b) -- (c);
    \end{tikzpicture} \right] \le  \mathbb{E} \left[ \begin{tikzpicture}[scale=0.7, baseline=(current bounding box.center)]
        \coordinate[vtx] (a) at (0, 0);
        \coordinate[vtx] (b) at (1.5, 0);
        \coordinate[vtx] (c) at (3, 0);
        \node at (a) [left] {$0$};
        \node at (b) [below] {$T_2$};
        \node at (c) [below] {$T_4$};
        \draw (a) -- (b) -- (c);
    \end{tikzpicture} \right]  
    \]
    Combined with the fact that $T$ is a positive definite kernel, our representation gives a clear intuition for why this inequality holds. For instance, the first inequality replaces $T_1$ with $T_2$. On average, $T_2$ is closer to $T_4$ than $T_1$ is, so it is believable that the inequality holds true. The second inequality is even easier to prove. Namely, conditioned on $T_2$ and $T_4$, the inequality is a direct consequence of the positive definiteness of $T$. We believe that this idea will have many similar applications in areas where lace expansion and diagrammatic sums are used, such as percolation or self-avoiding walks (e.g., \cite{HaraHofstadSlade2003, Sakai2022, BolthausenHofstadKozma2018}).
    
    \subsection{Notation} \label{subsec:notation}
    We use $\mathcal{S}_n^{(l, j)}$ to denote the $j$-th piece of the random walk $S[0, n]$ that has been divided up into $2^l$ parts. We will also use $[n]^{(l, j)}$ to denote the $j$-th piece of the interval $[n] = \{ 0, \dots , n \}$. That is, 
    $\mathcal{S}_n^{(l, j)} = S[[n]^{(l, j)}] = S[2^{-l}n (j-1), 2^{-l}n j]$. In particular, $S[0, n] = \mathcal{S}_n^{(0, 1)} =: \mathcal{S}_n$. For convenience, we assume that $n$ is divisible by a large power of $2$, but this does not affect the proof in any significant way. $\tilde{S}$ will refer to an independent copy of $S$, with range $\tilde{\mathcal{S}}_n = \tilde{S}[0, n]$ and so on. $\mathbb{P}^x$ (resp. $\mathbb{E}^x$) be the probability (resp. expectation) of the simple random walk starting at $x$. $\tau_A$ is the first positive hitting time of the set $A$: $\tau_A = \min \{ i > 0 : S_i \in A \}$.
    
    We use $C$ and $c$ to denote universal constants that may change from line to line. $C_{\epsilon}$ denotes a constant which depends on $\epsilon$, and also may change from line to line. For positive $f$ and $g$ depending on $n$, We write $f = O(g)$, $f \lesssim g$, and $g \gtrsim f$ to mean $f \le Cg$. $f \asymp g$ means $f \lesssim g \lesssim f$. $f = o(g)$, $f \ll g$, and $g \gg f$ all mean $f/g \to 0$ as $n \to \infty$.
    
    $G_D (x, y) = G_D(x - y)$ is the discrete Green's function for the simple random on $\mathbb{Z}^5$, with convolutional square root $\tilde{G}_D$. $G(x,y) = G(x-y)$ is the continuous Green's function of the Brownian motion in $\mathbb{R}^5$, and $\tilde{G}$ is its convolutional square root. $p_t$ is the transition probability of the Brownian motion . $G^{\epsilon}( = p_{\epsilon} \ast G)$ and $\tilde{G}^{\epsilon}$ denote the convolution of $G$ and $\tilde{G}$ with $p_{\epsilon}$, respectively.
    
    \section{Proof Outline} \label{sec:sketch}
    
    \subsection{Splitting the Walk} \label{subsec:split}
    One common technique when analyzing $\mathcal{S}_n = S[0,n]$ is to think of it as a union of its two halves, $\mathcal{S}_n^{(1, 1)} = S[0,n/2]$ and $\mathcal{S}_n^{(1, 2)} = S[n/2, n]$. By the reversibility of random walks, we may view these parts as $\emph{two independent}$ random walks starting from $S_{n/2}$. Since $\mathbb{Z}^d$ is transitive, we have
    \begin{align*}
        \cp(\mathcal{S}_n) &= \cp(\mathcal{S}_n^{(1, 1)}) + \cp(\mathcal{S}_n^{(1, 2)}) - \chi_{\mathcal{C}}(\mathcal{S}_n^{(1, 1)}, \mathcal{S}_n^{(1, 2)}) \\
        &\overset{d}{=} \cp(\mathcal{S}_n^{(1, 1)}) + \cp(\tilde{\mathcal{S}}_n^{(1, 1)}) - \chi_{\mathcal{C}}(\mathcal{S}_n^{(1, 1)}, \tilde{\mathcal{S}}_n^{(1, 1)}),
    \end{align*}
    where $S$ and $\tilde{S}$ are two independent random walks and
    \[  
    \chi_{\mathcal{C}}(A, B) := \cp(A) + \cp(B) - \cp(A \cup B).
    \]
    Thus, we get a recursive structure for $\cp(\mathcal{S}_n)$, which more or less reduces our problem to studying $\chi_{\mathcal{C}}(\mathcal{S}_n, \tilde{\mathcal{S}}_n)$. As it turns out, $\chi_{\mathcal{C}}$ is very close to the following quantity.
    \begin{equation}
    \label{eq:cross-term}
        2\chi(A, B) := 2\sum_{x \in A} \sum_{y \in B} \mathbb{P}^x ( \tau_A = \infty) G_D(x, y) \mathbb{P}^y (\tau_B = \infty)
    \end{equation}
    In fact, their difference is negligible up to the large deviation regime. This is shown by proving
    \begin{equation}
    \label{eq:error-bound}
        0 \le \epsilon(\mathcal{S}_n, \tilde{\mathcal{S}}_n) := 2\chi(\mathcal{S}_n, \tilde{\mathcal{S}}_n) - \chi_{\mathcal{C}}(\mathcal{S}_n, \tilde{\mathcal{S}}_n) \le \sum_{i, j_1, j_2 = 1}^n G_D(S_i - \tilde{S}_{j_1}) G_D (S_i - \tilde{S}_{j_2})
    \end{equation}
    and then bounding the moments of the right-hand side. The inequality itself is straightforward via a last-passage path decomposition (see Section~\ref{sec:main} and also \cite{AsselahSchapira2020}). Bounding the moments, however, turns out to be much more challenging (see  the discussion in Section~\ref{subsec:moments-sketch}). This moment computation is the main novelty of our paper, and also how we improve the results of Okada and the first author \cite{AdhikariOkada2023}. The move from $\chi_{\mathcal{C}}$ to $\chi$ is done in order to reduce the dependence between $\mathcal{S}_n$ and $\tilde{\mathcal{S}}_n$; see Section~\ref{subsec:regimes-sketch} for more explanation or Proposition~\ref{prop:lower-bound} for where it is used in the proof.
    
    Once the error term is dealt with, we can turn our attention to understanding $\chi(\mathcal{S}_n, \tilde{\mathcal{S}}_n)$. The first result is a moment estimate,
    \begin{equation}
    \label{eq:cross-term-moment-1}
        \mathbb{E} \chi(\mathcal{S}_n, \tilde{\mathcal{S}}_n)^m  \le C^m (m!)^{3/2} n^{m/2}.
    \end{equation}
    By the above and an induction argument, we also get an exponential moment estimate for $\cp(\mathcal{S}_n)$. That is, for every $\theta > 0$,
    \begin{equation}
    \label{eq:cap-moment}
        \sup_n \mathbb{E} \exp \left[ \frac{\theta}{\sqrt[3]{n \log n}} | \cp(\mathcal{S}_n) - \mathbb{E} \cp (\mathcal{S}_n) |^{2/3} \right] < \infty.
    \end{equation}
    This is the correct scaling to consider since $\var(\cp(\mathcal{S}_n)) \asymp n \log n$. Moreover, we know from \cite{Schapira2020} that $\cp(\mathcal{S}_n)$ follows a central limit theorem.
    
    \subsection{The two regimes} \label{subsec:regimes-sketch}
    Here we describe, on a heuristic level, why the phase transition occurs and our strategy for proving Theorems~\ref{thm:gaussian} and \ref{thm:non-gaussian}. From now on, we will use the decomposition
    \begin{equation}
    \label{eq:decomposition}
    \begin{aligned}
        \cp(\mathcal{S}_n) &= \cp(\mathcal{S}_n^{(1, 1)}) + \cp(\mathcal{S}_n^{(1, 2)}) - \chi_{\mathcal{C}}(\mathcal{S}_n^{(1, 1)}, \mathcal{S}_n^{(1, 2)}) \\
        &=\cp(\mathcal{S}_n^{(1, 1)}) + \cp(\mathcal{S}_n^{(1, 2)}) - 2\chi(\mathcal{S}_n^{(1, 1)}, \mathcal{S}_n^{(1, 2)}) + \epsilon(\mathcal{S}_n^{(1, 1)}, \mathcal{S}_n^{(1, 2)}).
    \end{aligned}
    \end{equation}
    Again, the key idea is that $\mathcal{S}_n^{(1, 1)}$ and $\mathcal{S}_n^{(1, 2)}$ may be viewed as two independent random walks of length $n/2$, and that $\epsilon(\mathcal{S}_n^{(1, 1)} , \mathcal{S}_n^{(1,2)})$ is negligible. We may repeat this decomposition for $\mathcal{S}_n^{(1, 1)}$ and $\mathcal{S}_n^{(1, 2)}$. After doing so $L$ times (where $L$ may grow with $n$), we get the following.
    
    \begin{equation}
    \label{eq:decomposition-L}
        \cp(\mathcal{S}_n) = \sum_{j=1}^{2^L} \cp(\mathcal{S}_n^{(L, j)}) -\sum_{l=1}^L  \Lambda^{\mathcal{C}}_l = \sum_{j=1}^{2^L} \cp(\mathcal{S}_n^{(L, j)}) -\sum_{l=1}^L  \Lambda_l + \epsilon_L,
    \end{equation}
    where $\Lambda_l^{\mathcal{C}}$, $\Lambda_l$, and $\epsilon_L$ stand for
    \begin{align*}
        \Lambda_l^{\mathcal{C}} &= \sum_{j=1}^{2^{l-1}} \chi_{\mathcal{C}}(\mathcal{S}_n^{(l, 2j-1)}, \mathcal{S}_n^{(l, 2j)}) \\
        \Lambda_l &= 2\sum_{j=1}^{2^{l-1}} \chi(\mathcal{S}_n^{(l, 2j-1)}, \mathcal{S}_n^{(l, 2j)}) \\
        \epsilon_L &= \sum_{l=1}^L \sum_{j=1}^{2^{l-1}} \epsilon(\mathcal{S}_n^{(l, 2j-1)}, \mathcal{S}_n^{(l, 2j)}).
    \end{align*}  
    As mentioned previously, $\epsilon_L$ is negligible up to the large deviation regime. When $1 \ll b_n \ll \sqrt{\log n}$, \eqref{eq:cross-term-moment-1} implies that the deviations of $\Lambda_l$ are much smaller than those of $\cp(\mathcal{S}_n^{(L, j)})$. Since these form an i.i.d. sum, we can expect Gaussian behavior in this regime. This is indeed what we see in Theorem~\ref{thm:gaussian}. The proof is straightforward, relying on sufficient tail bounds on $\cp(\mathcal{S}_n^{(L, j)})$ as given by \eqref{eq:cap-moment}.
    
    On the other hand, when $\sqrt{\log n} \ll b_n$, it is the cross-term that dominates. Hence, computing the lower tail deviations of $\cp(\mathcal{S}_n)$ is closely related to computing the upper tail deviations of $\chi(\mathcal{S}_n, \tilde{\mathcal{S}}_n)$. The first step is showing weak convergence to its continuous counterpart (Lemma~\ref{lem:weak-conv}):
    \[  
        \frac{1}{\sqrt{n}}\chi(\mathcal{S}_n, \tilde{\mathcal{S}}_n) \xrightarrow{d} 5^{5/2} \tilde{\gamma}_S^2 \int_0^1 \int_0^1 G(B_t - \tilde{B}_s) \mathrm{d}t \mathrm{d}s,
    \]
    where $G(x) = \frac{1}{4\pi^2}|x|^{-3}$ is the continuous Green's function for the Brownian motion in $\mathbb{R}^5$. The large deviations for the right-hand side was given by Bass, Chen, and Rosen \cite{BassChenRosen2009} and stated in Theorem~\ref{thm:brownian-LDP} of this paper.
    
    We show that these two behave the same way even in the large deviations regime. This is done by showing the upper and lower bounds separately. The upper bound uses subadditivity. The lower bound is an application of the Feynman--Kac formula adapted to the discrete setting. Both of these techniques were used in previously in other papers \cite{AdhikariOkada2023, Chen2010}. In particular, the Feynman--Kac formula is why we move from $\chi_{\mathcal{C}}$ to $\chi$. Because each of the probability terms in $\chi(A, B)$ only depend on one of $A$ and $B$, we can separate it into two halves:
    \begin{align*}
    \chi(A, B) &= \sum_{x \in A} \sum_{y \in B} \mathbb{P}^x (\tau_A = \infty) G_D (x, y) \mathbb{P}^y (\tau_B = \infty) \\
    &= \sum_{a \in \mathbb{Z}^5} \left( \sum_{x \in A} \mathbb{P}^x (\tau_A = \infty) \tilde{G}_D (x, a) \right) \left( \sum_{y \in B} \mathbb{P}^y (\tau_B = \infty) \tilde{G}_D (y, a) \right),
    \end{align*}
    where $\tilde{G}_D$ is the convolutional square root of $G_D$, i.e., $G_D = \tilde{G}_D \ast \tilde{G}_D$. However, we still have the problem that $\mathbb{P}^{S_i}(\tau_{\mathcal{S}_n} =\infty)$ is not Markovian. To resolve this issue, we also introduce $\chi_{b_l}$ (defined in \eqref{eq:defchibn}), which a more localized version of $\chi$. Showing that $\chi$ and $\chi_{b_l}$ behave the same in the large deviation regime is also a key challenge, which we resolve via a moment computation similar to (but even more intricate than) that of $\epsilon_L$.
    
    \subsection{Moment estimates via graph reduction} \label{subsec:moments-sketch}
    In order to justify the claims of the previous subsections, we show an upper bound of the moments of $\chi(\mathcal{S}_n, \tilde{\mathcal{S}}_n)$ and $\epsilon(\mathcal{S}_n, \tilde{\mathcal{S}}_n)$. We first focus on $\chi(\mathcal{S}_n, \tilde{\mathcal{S}}_n)$, as it is more simple and helps motivate the innovations necessary for the latter term. Due to \eqref{eq:cross-term}, it suffices to bound
    \[
    \mathbb{E} \left[\sum_{i=1}^n \sum_{j=1}^n |S_i - \tilde{S}_j|_+^{-3}\right]^m = \mathbb{E} \left[\sum_{\substack{1 \le i_1 , \dots , i_m \le n \\ 1 \le j_1 , \dots , j_m \le n}}
     \prod_{k=1}^m |S_{i_k} - \tilde{S}_{j_k}|_+^{-3} \right].    
    \]
    Computations of this type were done in Lemma 4.1 of \cite{DemboOkada2024} for the random walk when $d = 4$ and follow analogous proofs for the Brownian motion \cite{LeGall1994}. The main strategy is to progressively exchange the randomness of $S_i$ for a deterministic factor of $|i|_+^{1/2}$. Indeed, for $m = 1$, we can use $\mathbb{E}|S_i - x|_+^{-3} \le C |i|_+^{-3/4} |x|_+^{-3/2}$ to get
     \begin{equation}
     \label{eq:ineq-example}
     \mathbb{E} |S_i - \tilde{S}_j|_+^{-3} \le C|i|_+^{-3/4} \mathbb{E}|\tilde{S}_j|_+^{-3/2} \le C |i|_+^{-3/4} |j|_+^{-3/4}.
     \end{equation}
     by first conditioning on $\tilde{S}_j$. Summing over all $1 \le i, j \le n$ concludes the proof. For higher orders $m$, the additional step is to condition over all but the last increment. Assuming $i_1 \le \dots \le i_m$, we may view $S_{i_m}$ as the $(i_m - i_{m-1})$-th position of a random walk starting at $S_{i_{m-1}}$. Using similar arguments as the one demonstrated and inducting over $m$, we get the desired result.
    
     Controlling the error term is much more involved. In view of \eqref{eq:error-bound}, we want to bound the moments of $\sum_{i=1}^n \sum_{j_1, j_2=1}^n |S_i - \tilde{S}_{j_1}|_+^{-3} |S_i - \tilde{S}_{j_2}|_+^{-3}$. This boils down to computing
     \begin{equation}
         \label{eq:error-bound-moment}
    \mathbb{E} \left[ \prod_{k=1}^m |S_{i_k} - \tilde{S}_{j_{2k-1}}|_+^{-3} |S_{i_k} - \tilde{S}_{j_{2k}}|_+^{-3} \right]
     \end{equation}
    for $1 \le i_k, j_k \le n$. This is harder than \eqref{eq:cross-term-moment} for multiple reasons. Most importantly, the expression is much more singular. From the view of $S_{i_k}$, it is of order $6$. This quickly leads to a problem, since $\mathbb{E}|S_i|_+^{-p} \asymp |i|_+^{-p/2}$ only holds up to $p < 5$. Furthermore, even if this exchange were possible, we still have the problem that $\sum_{i=1}^n |i|_+^{-p} \asymp n^{1-p}$ is only true for $p < 1$. If we want the best bounds possible, it is crucial that we do not have such singularities.
    
    Thus, it is natural to shift our focus to $\tilde{S}$. Here we see our second difficulty: the expression is not symmetric with respect to $j_1 , \dots , j_{2m}$. Since $\tilde{S}_{j_{2k-1}}$ and $\tilde{S}_{j_{2k}}$ are linked via $S_{i_k}$, we may no longer assume $j_1 \le \dots \le j_{2m}$ without loss of generality. Instead, we have to consider all possible relative orders. To ensure that we don't introduce new singularities, we need to proceed differently depending on how the $j_k$'s are ordered. In order to better manage this, we use a novel graphical notation to represent the relationship between points. For instance, the value $|x-y|_+^{-2}$ is represented by the graph $[\small
    \tikz[scale = 0.5, baseline=(current bounding box.center)]{
    \coordinate[inlinevtx] (x) at (0, 0);
    \coordinate[inlinevtx] (y) at (1.5, 0);
    \node at (x) [left] {$x$};
    \node at (y) [right] {$y$};
    \draw[thick] (x) -- (y) node [midway, above] { $2$};
    }]$. Using this notation, \eqref{eq:error-bound-moment} becomes the following graph.
    \[ \mathbb{E} \left[ \prod_{k=1}^m |S_{i_k} - \tilde{S}_{j_{2k-1}}|_+^{-3} |S_{i_k} - \tilde{S}_{j_{2k}}|_+^{-3} \right] = \mathbb{E} \left[
    \begin{tikzpicture}[scale=0.7, baseline=(current bounding box.center)]
        \coordinate[vtx] (u0) at (1.5, 0);
        \node at (u0) [below] {$S_{i_1}$};
        \coordinate[vtx] (v0) at (1, 2);
        \node at (v0) [above] {$\tilde{S}_{j_1}$};
        \coordinate[vtx] (w0) at (2, 2);
        \node at (w0) [above] {$\tilde{S}_{j_2}$};
        \draw[thick] (v0) -- (u0) node [midway, left] {$3$};
        \draw[thick] (w0) -- (u0) node [midway, right] {$3$};
        \coordinate[vtx] (u1) at (3.5, 0);
        \node at (u1) [below] {$S_{i_2}$};
        \coordinate[vtx] (v1) at (3, 2);
        \node at (v1) [above] {$\tilde{S}_{j_3}$};
        \coordinate[vtx] (w1) at (4, 2);
        \node at (w1) [above] {$\tilde{S}_{j_4}$};
        \draw[thick] (v1) -- (u1) node [midway, left] {$3$};
        \draw[thick] (w1) -- (u1) node [midway, right] {$3$};
        \node at (5, 1) {$\dots$};
        \coordinate[vtx] (u) at (6.5, 0);
        \node at (u) [below] {$S_{i_{m}}$};
            \coordinate[vtx] (v) at (6, 2);
            \node at (v) [above] {$\tilde{S}_{j_{2m-1}}$};
            \coordinate[vtx] (w) at (7, 2);
            \node at (w) [above] {$\tilde{S}_{j_{2m}}$};
            \draw[thick] (v) -- (u) node [midway, left] {$3$};
            \draw[thick] (w) -- (u) node [midway, right] {$3$};
    \end{tikzpicture}\right]
    \]
    Furthermore, inequalities such as \eqref{eq:ineq-example} can be interpreted as manipulations on graphs:
    \[
    \mathbb{E} \left[
    \tikz[scale = 0.7, baseline=(current bounding box.center)]{
    \coordinate[vtx] (i) at (2, 1) node at (i) [right] {$S_i$};
    \coordinate[vtx] (j) at (0, 0.5) node at (j) [left] {$\tilde{S}_j$};
    \coordinate[vtx] (0) at (2, 0) node at (0) [right] {$0$};
    
    \draw[thick] (i) -- (j) node [midway, above] {\small $3$};
    }
    \right] \le C |i|_+^{-3/4} \mathbb{E} \left[
    \tikz[scale = 0.7, baseline=-3]{
    \coordinate[vtx] (j) at (0,0) node at (j) [left] {$\tilde{S}_j$};
    \coordinate[vtx] (0) at (2,0) node at (0) [right] {$0$};
    
    \draw[thick] (0) -- (j) node [midway, above] {\small $3/2$};
    }
    \right] \le C |i|_+^{-3/4} |j|_+^{-3/4}.
    \]
    This means that we may use such inequalities to sequentially reduce a given graph. By doing so in a careful manner, we can prove our desired result. Note that there are $3m$ vertices in this graph, while sum of edge weights is $6m$. Hence, the optimal strategy is to remove two edges for every vertex. This will introduce a deterministic term of $\mathbb{E} |S_i|_+^{-2} \asymp |i|_+^{-1}$ at every step. After summing over all indices, this leads to a bound of order $(\log n)^{3m}$.
     
    The precise algorithm for achieving this result is described in Section~\ref{sec:moments}. The main challenge is finding a sequence of reductions that works inductively for all possible orderings of $j_1 , \dots , j_{2m}$ while not introducing loops or high-degree vertices; this is crucial if we want to obtain an optimal bound.
    
    \subsection{Outline} \label{subsec:outline}
    In Section~\ref{sec:moments}, we prove the key moment bounds for $\chi(\mathcal{S}_n, \tilde{\mathcal{S}}_n)$ and $\epsilon(\mathcal{S}_n, \tilde{\mathcal{S}}_n)$. This is the most novel part of our paper. Then, in Section~\ref{sec:main}, we prove the claims made in Section~\ref{subsec:split} and prove Theorem~\ref{thm:gaussian}. This includes justifying the decomposition of \eqref{eq:decomposition}, explaining the negligibility of $\epsilon(\mathcal{S}_n, \tilde{\mathcal{S}}_n$), and giving exponential moment bounds for $\cp(\mathcal{S}_n)$. These provide the necessary background for the proof of Theorems~\ref{thm:non-gaussian} and \ref{thm:cross-term-LDP}, which are given in Section~\ref{sec:cross-term}. Finally in Section~\ref{sec:4d}, we explain how our methods can be adapted to the $d = 4$ setting, leading to the proof of Theorem~\ref{thm:dim-4}.
    
    \section{Moment inequalities via graph representation}\label{sec:moments}
    
    In this section, we prove the following moment bounds.
    
    \begin{prop}
        \label{prop:cross-term-moment}
        Let $S$ and $\tilde{S}$ be independent random walks on $\mathbb{Z}^5$ starting from the origin. For any positive integers $n$ and $m$,
        \begin{equation}
        \label{eq:cross-term-moment}
        \mathbb{E} \left[ \sum_{1 \le i_1 , \dots , i_m \le n} \sum_{1 \le j_1 , \dots , j_{m} \le n} \prod_{k = 1}^{m} |S_{i_k} - \tilde{S}_{j_{k}}|_+^3 \right] \le C^m (m!)^{3/2} n^{m/2}.
        \end{equation}
    \end{prop}
    
    \begin{prop}
        \label{prop:error-term-moment}
        Let $S$ and $\tilde{S}$ be independent random walks on $\mathbb{Z}^5$ starting from the origin. For any positive integers $n$ and $m$,
        \begin{equation}
        \label{eq:error-term-moment}
        \mathbb{E} \left[ \sum_{1 \le i_1 , \dots , i_m \le n} \sum_{1 \le j_1 , \dots , j_{2m} \le n} \prod_{k = 1}^{m} |S_{i_k} - \tilde{S}_{j_{2k-1}}|_+^3 |S_{i_k} - \tilde{S}_{j_{2k}}|_+^3\right] \le C^m (m!)^3 (\log n)^{3m}.
        \end{equation}
    \end{prop}
    
    These inequalities are used in Section~\ref{subsec:LDP-estimates} (Lemmas~\ref{lem:cross-term-moment} and \ref{lem:error-LDP}) to compute exponential moments, which in turn give rise to bounds on large deviation probabilities. In addition, we also prove Lemma~\ref{lem:aux-error-moment} which is used to prove Proposition~\ref{prop:chiischibn}. Our starting point is the following lemma.
    
    \begin{lemma}
        \label{lem:rw-moment}
        Let $S$ be a random walk in $\mathbb{Z}^5$ starting from the origin, and $x, y, z \in \mathbb{Z}^5$. Then, the following hold.
        \begin{align}
            \mathbb{E} |S_i - x|_+^{-4} &\le C \min (|i|_+^{-2}, |x|_+^{-4}) \label{eq:rw-moment}\\
            \mathbb{E} |S_i - x|_+^{-3} &\le C \min (|i|^{-3/2}, |x|_+^{-3})
        \end{align}
    \end{lemma}
    
    \begin{proof}
    By H\"{o}lder's inequality, it suffices to show \eqref{eq:rw-moment}. First consider the corresponding computation for the Brownian motion.
    Namely, consider the integral,
    \begin{equation}
    \frac{1}{i^{5/2}} \int_{\mathbb{R}^5} \frac{1}{|y-x|^4} \exp[-|y|^2/i] \text{d}y
    \end{equation}
    
    We divide the integral above into two regions 1: $|x-y| \le |x|/2$ and 2: $|x-y| > |x|/2$. Note that in the second region, we have that $|y| \le |x-y| +|x| < 3 |x-y|$ and in the first region, we have $|y| \ge |x|- |x-y| \ge |x|/2$
    
    Now, consider the integral over the first region,
    \begin{equation}
    \begin{aligned}
    \frac{1}{i^{5/2}}\int_{|x-y| \le |x|/2} \frac{1}{|x-y|^4} \exp[-|y|^2/i] \text{d}y &\le \frac{1}{i^{5/2}} \int_{|x-y| \le|x|/2} \frac{1}{|x-y|^4} \exp[-|x|^2/4i] \text{d}y \\
    & \le \frac{\exp[-|x|^2/4i]}{i^{5/2}} \int_{|y| \le  |x|/2} \frac{1}{|y|^4} \text{d}y\\
    &= \text{O}\left(\frac{\exp[-|x|^2/4i] |x|}{i^{5/2}}\right)
    \end{aligned}
    \end{equation}
    
    In the first line, we used the fact that $|y| \ge |x|/2$ in the first region. In the second line, we changed variable from $x-y$ to $y$ and the third line is a direct integration.
    
    Notice also that there is a constant such that,$
    \frac{\exp[-\frac{|x|^2}{4i}] |x|}{i^{5/2}} \le  C\frac{1}{i^2},$
    since this is equivalent to saying,
    $
    \exp\left[-\frac{|x|^2}{4i}\right] \left(\frac{|x|^2}{i}\right)^{1/2}$
    is bounded for all values for $\frac{|x|^2}{i}$. This is clearly true since the term $\exp[- \frac{|x|^2}{4i}]$ will be small enough to compensate for the increase in $\left(\frac{|x|^2}{i}\right)^{1/2}$ as $\frac{|x|^2}{i}$ goes to infinity. 
    
    Similarly, we have that,$
    \frac{\exp\left[-\frac{|x|^2}{4i}\right] |x|}{i^{5/2}}  \le C \frac{1}{|x|^4},
    $
    since this is equivalent to saying that,
    $
    \exp\left[-\frac{|x|^2}{4i}\right]  \left(\frac{|x|^2}{i}\right)^{5/2}
    $
    is bounded for all values of $\frac{|x|^2}{i}$. We can use the same logic as we have done previously.
    
    Now, consider the integral over the second region. By one computation, we have that,
    
    \begin{equation}
    \begin{aligned}
    \frac{1}{i^{5/2}} \int_{|x-y| >|x|/2} \frac{1}{|x-y|^4} \exp[-|y|^2/i] \text{d}y &\le \frac{1}{i^{5/2}}   \int_{|x-y| >|x|/2} \frac{16}{|x|^4} \exp[-|y|^2/i] \text{d}y\\
    &\le  \frac{16}{|x|^4} \frac{1}{i^{5/2}}\int_{\mathbb{R}^5}  \exp[-|y|^2/i] \text{d}y = O(|x|^{-4})
    \end{aligned}
    \end{equation}
    
    By another computation, we have that,
    \begin{equation}
    \begin{aligned}
    \frac{1}{i^{5/2}} \int_{|x-y| >|x|/2} \frac{1}{|x-y|^4} \exp[-|y|^2/i] \text{d}y & \le \frac{1}{i^{5/2}}   \int_{|x-y| >|x|/2} \frac{81}{|y|^4} \exp[-|y|^2/i] \text{d}y\\
    & \le \frac{1}{i^{5/2}} \int_{\mathbb{R}^5} \frac{81}{|y|^4} \exp[-|y|^2/i] \text{d}y\\
    & = \frac{81}{i^2}\int_{\mathbb{R}^5} \frac{1}{|y|^4} \exp[-|y|^2] \text{d}y = O(i^{-2}).
    \end{aligned}
    \end{equation}
    
    In the first line, we used the fact that in region 2 we have, $3|x-y|> |y| $. In the last line, we used the rescaling $y \to \sqrt{i}y$. 
    
    This establishes the result for the probability distribution given by the Brownian motion. Now, if we consider the random walk, note that we have by Theorem 2.3.11 in \cite{LawlerLimic2010} that for $|y|\ll i^{2/3}$ that,
    \begin{equation} \label{eq:LCLT}
    p_i(y) = \overline{p}_{i}(y) \exp \left\{ O \left(\frac{1}{\sqrt{i}}  +\frac{|y|^3}{i^2}\right)\right\},
    \end{equation}
    where $p_i$ is the probability distribution for the random walk, and $\overline{p}_i$ corresponds to the density of the Brownian motion with appropriate covariance structure.
    
    Thus, all the computations for the Brownian motion will match those of the random walk as long as we consider the region $|y| \ll i^{2/3}$. Furthermore, exponential moments  for the random walk, show that $\mathbb{P}(|S_i| \ge i^{1/2 + \epsilon}) \le \exp[-i^{\epsilon}]$.  Therefore, we have,
    \begin{equation}
    \mathbb{E}[|S_i - y|^{-4}] \le \mathbb{P}(|S_i| \ge i^{1/2 + \epsilon}) + \mathbb{E}[|S_i - y|^{-4} \mathbf{1}[|S_i| \le i^{1/2 + \epsilon} ]]
    \end{equation}
    For the latter integral, we can perform the same computations done for the Brownian motion by appealing to equation \eqref{eq:LCLT}. $\mathbb{P}(|S_i| \ge i^{1/2 + \epsilon})$  is exponentially small and should be smaller than $i^{-2}$ or $|y|^{-4}$, unless $|y|$ is exponentially large. In the case that $|y| $ is exponentially large, one could simply use the fact that $|S_i| \le i$ deterministically, and assert $\frac{1}{|S_i - y|^4} = O(|y|^{-4})$ deterministically.
    \end{proof}
    
    \subsection{Representation by graphs}
    
    As explained in Section~\ref{subsec:moments-sketch}, we will be using graphs to represent our expressions.
    
    \begin{defn}
        Let $G$ be a finite edge-weighted graph with vertex set $V \subseteq \mathbb{Z}^d$. We use brackets $[G]$ to denote
        \[  
        [G] := \prod_{x, y\in V} |x - y|_+^{-w(x, y)},
        \]
        where $w(x, y)$ is the weight of the edge connecting $x$ and $y$. In particular, if there is no edge between $x$ and $y$, then $w(x, y) = 0$.
    \end{defn}
    
    In particular, our two main graphs of interest are those arising from \eqref{eq:cross-term-moment} and \eqref{eq:error-term-moment}. Black edges always have edge weight $3$.
    \begin{equation}
    \label{eq:cross-term-graph}
       \mathbb{E} \left[ \prod_{k = 1}^{m} |S_{i_k} - \tilde{S}_{j_{k}}|_+^3 \right] = \mathbb{E} \left[ 
        \begin{tikzpicture}[scale=0.7, baseline=(current bounding box.center)]
            \coordinate[vtx] (u) at (0, 0);
            \node at (u) [below] {$S_{i_1}$};
            \coordinate[vtx] (v) at (0, 2);
            \node at (v) [above] {$\tilde{S}_{j_1}$};
            \draw[thick] (v) -- (u);
            \coordinate[vtx] (u) at (1.5, 0);
            \node at (u) [below] {$S_{i_2}$};
            \coordinate[vtx] (v) at (1.5, 2);
            \node at (v) [above] {$\tilde{S}_{j_2}$};
            \draw[thick] (v) -- (u);
            \node at (3, 1) {$\dots$};
            \coordinate[vtx] (u) at (4.5, 0);
            \node at (u) [below] {$S_{i_{m-2}}$};
            \coordinate[vtx] (v) at (4.5, 2);
            \node at (v) [above] {$\tilde{S}_{j_{m-2}}$};
            \draw[thick] (v) -- (u);
            \coordinate[vtx] (u) at (6, 0);
            \node at (u) [below] {$S_{i_{m-1}}$};
            \coordinate[vtx] (v) at (6, 2);
            \node at (v) [above] {$\tilde{S}_{j_{m-1}}$};
            \draw[thick] (v) -- (u);
            \coordinate[vtx] (u) at (7.5, 0);
            \node at (u) [below] {$S_{i_{m}}$};
            \coordinate[vtx] (v) at (7.5, 2);
            \node at (v) [above] {$\tilde{S}_{j_{m}}$};
            \draw[thick] (v) -- (u);
    \end{tikzpicture}\right]
    \end{equation}
    \begin{equation}
        \label{eq:error-term-graph}
        \mathbb{E} \left[ \prod_{k=1}^m |S_{i_k} - \tilde{S}_{j_{2k-1}}|^{-3} |S_{i_k} - \tilde{S}_{j_{2k}}|^{-3} \right] = \mathbb{E} \left[
    \begin{tikzpicture}[scale=0.7, baseline=(current bounding box.center)]
        \coordinate[vtx] (u0) at (1.5, 0);
        \node at (u0) [below] {$S_{i_1}$};
        \coordinate[vtx] (v0) at (1, 2);
        \node at (v0) [above] {$\tilde{S}_{j_1}$};
        \coordinate[vtx] (w0) at (2, 2);
        \node at (w0) [above] {$\tilde{S}_{j_2}$};
        \draw[thick] (v0) -- (u0);
        \draw[thick] (w0) -- (u0);
        \coordinate[vtx] (u1) at (3.5, 0);
        \node at (u1) [below] {$S_{i_2}$};
        \coordinate[vtx] (v1) at (3, 2);
        \node at (v1) [above] {$\tilde{S}_{j_3}$};
        \coordinate[vtx] (w1) at (4, 2);
        \node at (w1) [above] {$\tilde{S}_{j_4}$};
        \draw[thick] (v1) -- (u1);
        \draw[thick] (w1) -- (u1);
        \node at (5, 1) {$\dots$};
        \coordinate[vtx] (u) at (6.5, 0);
        \node at (u) [below] {$S_{i_{m}}$};
            \coordinate[vtx] (v) at (6, 2);
            \node at (v) [above] {$\tilde{S}_{j_{2m-1}}$};
            \coordinate[vtx] (w) at (7, 2);
            \node at (w) [above] {$\tilde{S}_{j_{2m}}$};
            \draw[thick] (v) -- (u);
            \draw[thick] (w) -- (u);
    \end{tikzpicture}\right]
    \end{equation}
    We begin with the moment bounds for the cross term. As mentioned in Section~\ref{subsec:moments-sketch}, this is essentially identical to the proof of the $d = 4$ case done by Dembo and Okada (\cite{DemboOkada2024}, Lemma 4.1). However, we repeat it here using our new notation, both for completeness and to serve as a simpler example of how our graph reduction is done.
    
    We begin with the following lemma, which allow us to gradually reduce our graph $G$. A neighborhood refers to a point and all its adjacent vertices.
    
    \begin{lemma}
    \label{lem:cross-term-graph}
    Let $S$ be a random walk on $\mathbb{Z}^5$ starting at $S_0$. Assume all vertices except $S_i$ are deterministic, and that $G$ and $G'$ agree outside the neighborhood of $S_0$ and $S_i$ which are as in the figure below. Then, $\mathbb{E}[G] \le C |i|_+^{-3/4} [G']$.
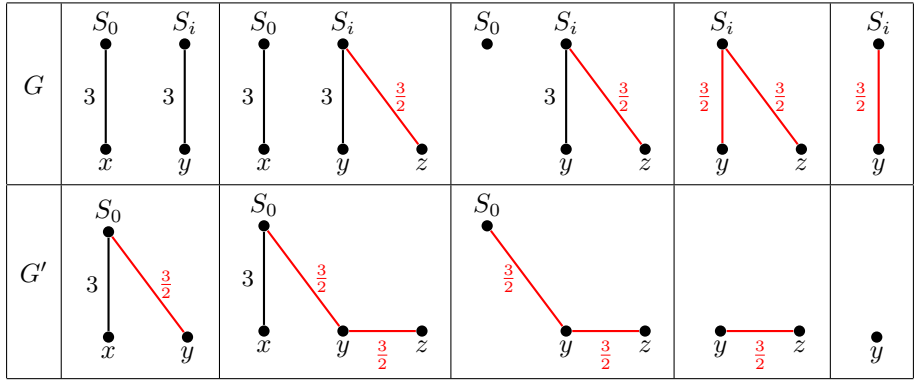
    \[
    \begin{array}{|c|c|c|c|c|c|}
    \hline
       G  & \begin{tikzpicture}[scale=0.7, baseline=(current bounding box.center)]
            \coordinate[vtx] (u1) at (0, 0);
            \node at (u1) [below] {$x$};
            \coordinate[vtx] (v1) at (0, 2);
            \node at (v1) [above] {$S_0$};
            \draw[thick] (v1) -- (u1) node [left, midway] {\small $3$};
            
            \coordinate[vtx] (u2) at (1.5, 0);
            \node at (u2) [below] {$y$};
            \coordinate[vtx] (v2) at (1.5, 2);
            \node at (v2) [above] {$S_i$};
            \draw[thick] (v2) -- (u2) node [left, midway] {\small $3$};
        \end{tikzpicture} & \begin{tikzpicture}[scale=0.7, baseline=(current bounding box.center)]
            \coordinate[vtx] (u1) at (0, 0);
            \node at (u1) [below] {$x$};
            \coordinate[vtx] (v1) at (0, 2);
            \node at (v1) [above] {$S_0$};
            \draw[thick] (v1) -- (u1) node [left, midway] {\small $3$};
            
            \coordinate[vtx] (u2) at (1.5, 0);
            \node at (u2) [below] {$y$};
            \coordinate[vtx] (v2) at (1.5, 2);
            \node at (v2) [above] {$S_i$};
            \draw[thick] (v2) -- (u2) node [left, midway] {\small $3$};
    
            \coordinate[vtx] (u3) at (3, 0);
            \node at (u3) [below] {$z$};
            \draw[thick, red] (u3) -- (v2) node [right, midway] {\small $\frac{3}{2}$};
        \end{tikzpicture} & \begin{tikzpicture}[scale=0.7, baseline=(current bounding box.center)]
            \coordinate[vtx] (v1) at (0, 2);
            \node at (v1) [above] {$S_0$};
            
            \coordinate[vtx] (u2) at (1.5, 0);
            \node at (u2) [below] {$y$};
            \coordinate[vtx] (v2) at (1.5, 2);
            \node at (v2) [above] {$S_i$};
            \draw[thick] (v2) -- (u2) node [left, midway] {\small $3$};
    
            \coordinate[vtx] (u3) at (3, 0);
            \node at (u3) [below] {$z$};
            \draw[thick, red] (u3) -- (v2) node [right, midway] {\small $\frac{3}{2}$};
        \end{tikzpicture} & \begin{tikzpicture}[scale=0.7, baseline=(current bounding box.center)]
            \coordinate[vtx] (u2) at (1.5, 0);
            \node at (u2) [below] {$y$};
            \coordinate[vtx] (v2) at (1.5, 2);
            \node at (v2) [above] {$S_i$};
            \draw[thick, red] (v2) -- (u2) node [left, midway] {\small $\frac{3}{2}$};
    
            \coordinate[vtx] (u3) at (3, 0);
            \node at (u3) [below] {$z$};
            \draw[thick, red] (u3) -- (v2) node [right, midway] {\small $\frac{3}{2}$};
        \end{tikzpicture} & \begin{tikzpicture}[scale=0.7, baseline=(current bounding box.center)]        
            \coordinate[vtx] (u2) at (1.5, 0);
            \node at (u2) [below] {$y$};
            \coordinate[vtx] (v2) at (1.5, 2);
            \node at (v2) [above] {$S_i$};
            \draw[thick, red] (v2) -- (u2) node [left, midway] {\small $\frac{3}{2}$};
        \end{tikzpicture} \\
        \hline
       G' & \begin{tikzpicture}[scale=0.7, baseline=(current bounding box.center)]
            \coordinate[vtx] (u1) at (0, 0);
            \node at (u1) [below] {$x$};
            \coordinate[vtx] (v1) at (0, 2);
            \node at (v1) [above] {$S_0$};
            \draw[thick] (v1) -- (u1) node [left, midway] {\small $3$};
            
            \coordinate[vtx] (u2) at (1.5, 0);
            \node at (u2) [below] {$y$};
            \draw[thick, color=red] (v1) -- (u2) node [right, midway] {\small $\frac{3}{2}$};
        \end{tikzpicture} & \begin{tikzpicture}[scale=0.7, baseline=(current bounding box.center)]
            \coordinate[vtx] (u1) at (0, 0);
            \node at (u1) [below] {$x$};
            \coordinate[vtx] (v1) at (0, 2);
            \node at (v1) [above] {$S_0$};
            \draw[thick] (v1) -- (u1) node [left, midway] {\small $3$};
            
            \coordinate[vtx] (u2) at (1.5, 0);
            \node at (u2) [below] {$y$};
            \draw[thick, red] (v1) -- (u2) node [right, midway] {\small $\frac{3}{2}$};
    
            \coordinate[vtx] (u3) at (3, 0);
            \node at (u3) [below] {$z$};
            \draw[thick, red] (u3) -- (u2) node [below, midway] {\small $\frac{3}{2}$};
        \end{tikzpicture} & \begin{tikzpicture}[scale=0.7, baseline=(current bounding box.center)]
            \coordinate[vtx] (v1) at (0, 2);
            \node at (v1) [above] {$S_0$};
            
            \coordinate[vtx] (u2) at (1.5, 0);
            \node at (u2) [below] {$y$};
            \draw[thick, red] (v1) -- (u2) node [left, midway] {\small $\frac{3}{2}$};
    
            \coordinate[vtx] (u3) at (3, 0);
            \node at (u3) [below] {$z$};
            \draw[thick, red] (u3) -- (u2) node [below, midway] {\small $\frac{3}{2}$};
        \end{tikzpicture} & \begin{tikzpicture}[scale=0.7, baseline=(current bounding box.center)]
            \coordinate[vtx] (u2) at (1.5, 0);
            \node at (u2) [below] {$y$};
            \node[white] at (v2) [above] {$S_i$};
            \draw[thick, white] (v2) -- (u2) node [left, midway] {\small $3$};
    
            \coordinate[vtx] (u3) at (3, 0);
            \node at (u3) [below] {$z$};
            \draw[thick, red] (u3) -- (u2) node [below, midway] {\small $\frac{3}{2}$};
        \end{tikzpicture} & \begin{tikzpicture}[scale=0.7, baseline=(current bounding box.center)]        
            \coordinate[vtx] (u2) at (1.5, 0);
            \node at (u2) [below] {$y$};
            \node[white] at (v2) [above] {$S_i$};
            \draw[thick, white] (v2) -- (u2) node [left, midway] {\small $3$};
        \end{tikzpicture} \\
        \hline
    \end{array}
    \]
    \end{lemma}
    
    \begin{proof}
        Since all proofs are similar, we will only show the third reduction. This equates to proving
        \[  
        \mathbb{E}|S_i - y|_+^{-3} |S_i - z|_+^{-3/2} \le C |i|_+^{-3/4} |y|_+^{-3/2} |y-z|_+^{-3/2}.
        \]
        By the triangle inequality and Holder's inequality, $|y-z|_+^{3/2} \le C(|S_i - y|_+^{3/2} + |S_i - z|_+^{3/2})$. Therefore,
        \begin{align*}
            |y-z|_+^{3/2} \mathbb{E}|S_i - y|_+^{-3} |S_i - z|_+^{-3/2} &\le C\left( \mathbb{E}|S_i - y|_+^{-3/2} |S_i-z|_+^{-3/2} + \mathbb{E} |S_i - y|_+^{-3} \right) \\
            &\le C \left(\mathbb{E}|S_i-y|_+^{-3} \right)^{1/2} \left\{\left(\mathbb{E}|S_i-y|_+^{-3} \right)^{1/2} + \left(\mathbb{E}|S_i-z|_+^{-3} \right)^{1/2} \right\} \\
            &\le C |y|_+^{-3/2} |i|_+^{-3/4}
        \end{align*}
        By Lemma~\ref{lem:rw-moment}.
    \end{proof}
    
    For most applications of this lemma, the vertices will be random points $S_{i_1}, \dots S_{i_k}$. However, we will condition on the set $S[0, i_{k-1}]$, so that we may view $S_{i_k}$ as the only random variable. From this perspective, $S_{i_k}$ is the $(i_k - i_{k-1})$-th step of a random walk starting at $S_{i_{k-1}}$ (assuming $i_k \ge i_{k-1}$). That is, we may replace $i$ and $S_0$ in Lemma~\ref{lem:cross-term-graph} with $i_k - i_{k-1}$ and $S_{i_{k-1}}$, respectively. By repeating this process, we may gradually bound $\mathbb{E}[G]$ by a deterministic value. In order to do this, it is important to know the local structure around $S_0$ after the operation. This is why we have included $x$ in the graphs, even though it does not play any role in the inequalities. Note that all outcomes for $S_0$ in $G'$ fall under one of the cases considered in $G$.
    
    \begin{proof}[Proof of Proposition~\ref{prop:cross-term-moment}]
        Throughout this proof, a black edge will always have weight $3$ and a red edge will always have weight $3/2$. Using our graphical notation, we can write
        \[
        \mathbb{E} \left[ \prod_{k = 1}^{m} |S_{i_k} - \tilde{S}_{j_{k}}|_+^3 \right] = \mathbb{E} \left[ 
        \begin{tikzpicture}[scale=0.7, baseline=(current bounding box.center)]
            \coordinate[vtx] (u) at (0, 0);
            \node at (u) [below] {$S_{i_1}$};
            \coordinate[vtx] (v) at (0, 2);
            \node at (v) [above] {$\tilde{S}_{j_1}$};
            \draw[thick] (v) -- (u);
            
            \coordinate[vtx] (u) at (1.5, 0);
            \node at (u) [below] {$S_{i_2}$};
            \coordinate[vtx] (v) at (1.5, 2);
            \node at (v) [above] {$\tilde{S}_{j_2}$};
            \draw[thick] (v) -- (u);
            
            \node at (3, 1) {$\dots$};
            
            \coordinate[vtx] (u) at (4.5, 0);
            \node at (u) [below] {$S_{i_{m-2}}$};
            \coordinate[vtx] (v) at (4.5, 2);
            \node at (v) [above] {$\tilde{S}_{j_{m-2}}$};
            \draw[thick] (v) -- (u);
            
            \coordinate[vtx] (u) at (6, 0);
            \node at (u) [below] {$S_{i_{m-1}}$};
            \coordinate[vtx] (v) at (6, 2);
            \node at (v) [above] {$\tilde{S}_{j_{m-1}}$};
            \draw[thick] (v) -- (u);
    
            \coordinate[vtx] (u) at (7.5, 0);
            \node at (u) [below] {$S_{i_{m}}$};
            \coordinate[vtx] (v) at (7.5, 2);
            \node at (v) [above] {$\tilde{S}_{j_{m}}$};
            \draw[thick] (v) -- (u);
    \end{tikzpicture}\right].
        \]
        By symmetry, we may assume $0 = j_0 \le j_1 \le \dots \le j_m$. By iteratively conditioning on $S[0, n] \cup \tilde{S}[0, j_{k-1}]$ for $k = m, m-1, \dots , 1$, Lemma~\ref{lem:cross-term-graph} gives us a unique procedure at every step. More explicitly, we get the following inequality.
        \begin{align*}
            \mathbb{E} \left[ \begin{tikzpicture}[scale=0.7, baseline=(current bounding box.center)]
            \coordinate[vtx] (u1) at (0, 0);
            \node at (u1) [below] {$S_{i_1}$};
            \coordinate[vtx] (v1) at (0, 1);
            \node at (v1) [above] {$\tilde{S}_{j_1}$};
            \draw[thick] (v1) -- (u1);
            \node at (1, 0.5) {$\dots$};        
            \coordinate[vtx] (u2) at (2, 0);
            \coordinate[vtx] (v2) at (2, 1);
            \draw[thick] (v2) -- (u2);        
            \coordinate[vtx] (u3) at (3, 0);
            \node at (u3) [below] {$S_{i_{m-1}}$};
            \coordinate[vtx] (v3) at (3, 1);
            \node at (v3) [above] {$\tilde{S}_{j_{m-1}}$};
            \draw[thick] (v3) -- (u3);
            \coordinate[vtx] (u4) at (4, 0);
            \node at (u4) [below] {$S_{i_{m}}$};
            \coordinate[vtx] (v4) at (4, 1);
            \node at (v4) [above] {$\tilde{S}_{j_{m}}$};
            \draw[thick] (v4) -- (u4);
        \end{tikzpicture} \right] &\le \mathbb{E} \left[ \begin{tikzpicture}[scale=0.7, baseline=(current bounding box.center)]
            \coordinate[vtx] (u1) at (0, 0);
            \node at (u1) [below] {$S_{i_1}$};
            \coordinate[vtx] (v1) at (0, 1);
            \node at (v1) [above] {$\tilde{S}_{j_1}$};
            \draw[thick] (v1) -- (u1);        
            \node at (1, 0.5) {$\dots$};        
            \coordinate[vtx] (u2) at (2, 0);
            \coordinate[vtx] (v2) at (2, 1);
            \draw[thick] (v2) -- (u2);        
            \coordinate[vtx] (u3) at (3, 0);
            \node at (u3) [below] {$S_{i_{m-1}}$};
            \coordinate[vtx] (v3) at (3, 1);
            \node at (v3) [above] {$\tilde{S}_{j_{m-1}}$};
            \draw[thick] (v3) -- (u3);
            \coordinate[vtx] (u4) at (4, 0);
            \node at (u4) [below] {$S_{i_{m}}$};
            \draw[thick, red] (v3) -- (u4);
    \end{tikzpicture} \right] C |j_m - j_{m-1}|_+^{-3/4} \\
    &\le \mathbb{E} \left[ \begin{tikzpicture}[scale=0.7, baseline=(current bounding box.center)]
            \coordinate[vtx] (u1) at (0, 0);
            \node at (u1) [below] {$S_{i_1}$};
            \coordinate[vtx] (v1) at (0, 1);
            \node at (v1) [above] {$\tilde{S}_{j_1}$};
            \draw[thick] (v1) -- (u1);        
            \node at (1, 0.5) {$\dots$};        
            \coordinate[vtx] (u2) at (2, 0);
            \coordinate[vtx] (v2) at (2, 1);
            \node at (v2) [above] {$\tilde{S}_{j_{m-2}}$};
            \draw[thick] (v2) -- (u2);        
            \coordinate[vtx] (u3) at (3, 0);
            \node at (u3) [below] {$S_{i_{m-1}}$};
            \draw[thick, red] (v2) -- (u3);
            \coordinate[vtx] (u4) at (4, 0);
            \node at (u4) [below] {$S_{i_{m}}$};
            \draw[thick, red] (u3) -- (u4);
    \end{tikzpicture} \right] C^2 |j_m - j_{m-1}|_+^{-3/4} |j_{m-1} - j_{m-2}|_+^{-3/4} \\
    &\vdots \\
    &\le \mathbb{E} \left[ \begin{tikzpicture}[scale=0.7, baseline=(current bounding box.center)]
            \coordinate[vtx] (u1) at (0, 0);
            \node at (u1) [below] {$S_{i_1}$};
            \coordinate[vtx] (v1) at (0, 1);
            \node at (v1) [above] {$\tilde{S}_{j_1}$};
            \draw[thick] (v1) -- (u1);        
            \coordinate[vtx] (u2) at (1, 0);
            \node at (u2) [below] {$S_{i_2}$};
            \draw[thick, red] (v1) -- (u2);       
            \coordinate[vtx] (u3) at (3, 0);
            \node at (u3) [below] {$S_{i_{m-1}}$};
            \draw[thick, red] (u3) -- (u2) node [below, midway, black] {$\dots$};
            \coordinate[vtx] (u4) at (4, 0);
            \node at (u4) [below] {$S_{i_{m}}$};
            \draw[thick, red] (u3) -- (u4);
    \end{tikzpicture} \right] C^{m-1} \prod_{k=2}^m |j_k - j_{k-1}|_+^{-3/4} \\
    &\le \mathbb{E} \left[ \begin{tikzpicture}[scale=0.7, baseline=(current bounding box.center)]
            \coordinate[vtx] (u0) at (-1, 0);
            \node at (u0) [below] {$0$};
            \coordinate[vtx] (u1) at (0, 0);
            \node at (u1) [below] {$S_{i_1}$};  
            \draw[thick, red] (u1) -- (u0); 
            \coordinate[vtx] (u2) at (1, 0);
            \node at (u2) [below] {$S_{i_2}$};
            \draw[thick, red] (u1) -- (u2);     
            \coordinate[vtx] (u3) at (3, 0);
            \node at (u3) [below] {$S_{i_{m-1}}$};
            \draw[thick, red] (u3) -- (u2) node [below, midway, black] {$\dots$};
            \coordinate[vtx] (u4) at (4, 0);
            \node at (u4) [below] {$S_{i_{m}}$};
            \draw[thick, red] (u3) -- (u4);
    \end{tikzpicture} \right] C^m \prod_{k=1}^m |j_k - j_{k-1}|_+^{-3/4}.
        \end{align*}
    This gives us a path on $\{ 0, S_{i_1}, \dots , S_{i_{m}} \}$. There is some permutation $\sigma$ of $\{ 1, \dots , m \}$ such that $i_{\sigma(1)} \le i_{\sigma(2)} \le \dots \le i_{\sigma(m)}$. Conditioning on $S[0, i_{\sigma(m-1)}]$ and using one of the last two inequalities of Lemma~\ref{lem:cross-term-graph}, we may reduce this to a path on $\{ 0, S_{i_{\sigma(1)}}, \dots , S_{i_{\sigma(m-1)}} \}$. That is, we get one of the following inequalities depending on whether $\sigma(m) = m$.
    \begin{align*}
    \mathbb{E} \left[ \begin{tikzpicture}[scale=0.7, baseline=(current bounding box.center)]
            \coordinate[vtx] (u0) at (0.5, 0);
            \node at (u0) [below] {$0$};
            \coordinate[vtx] (u2) at (2, 0);
            \node at (u2) [above] {$S_{i_{\sigma(m) - 1}}$};
            \draw[thick, red] (u0) -- (u2) node [below, midway, black] {$\dots$};
            \coordinate[vtx] (u3) at (3, 0);
            \node at (u3) [below] {$S_{i_{\sigma(m)}}$};
            \draw[thick, red] (u2) -- (u3);     
            \coordinate[vtx] (u4) at (4, 0);
            \node at (u4) [above] {$S_{i_{\sigma(m) +1}}$};
            \draw[thick, red] (u3) -- (u4);
            \coordinate[vtx] (u6) at (5.5, 0);
            \node at (u6) [below] {$S_{i_{m}}$};
            \draw[thick, red] (u4) -- (u6) node [below, midway, black] {$\dots$};
    \end{tikzpicture} \right] &\le \mathbb{E} \left[ \begin{tikzpicture}[scale=0.7, baseline=(current bounding box.center)]
             \coordinate[vtx] (u0) at (0.5, 0);
            \node at (u0) [below] {$0$};
            \coordinate[vtx] (u2) at (2, 0);
            \node at (u2) [above] {$S_{i_{\sigma(m) - 1}}$};
            \draw[thick, red] (u0) -- (u2) node [below, midway, black] {$\dots$};
            \draw[thick, red] (u2) -- (u4);     
            \coordinate[vtx] (u4) at (4, 0);
            \node at (u4) [above] {$S_{i_{\sigma(m) +1}}$};
            \draw[thick, red] (u3) -- (u4);
            \coordinate[vtx] (u6) at (5.5, 0);
            \node at (u6) [below] {$S_{i_{m}}$};
            \draw[thick, red] (u4) -- (u6) node [below, midway, black] {$\dots$};
    \end{tikzpicture} \right] C |i_{\sigma(m)} - i_{\sigma(m-1)}|_+^{-3/4} \\
    \mathbb{E} \left[ \begin{tikzpicture}[scale=0.7, baseline=(current bounding box.center)]
            \coordinate[vtx] (u0) at (-1, 0);
            \node at (u0) [below] {$0$};
            \coordinate[vtx] (u1) at (0, 0);
            \node at (u1) [below] {$S_{i_1}$};  
            \draw[thick, red] (u1) -- (u0); 
            \coordinate[vtx] (u2) at (1, 0);
            \node at (u2) [below] {$S_{i_2}$};
            \draw[thick, red] (u1) -- (u2);     
            \coordinate[vtx] (u3) at (3, 0);
            \node at (u3) [below] {$S_{i_{m-1}}$};
            \draw[thick, red] (u3) -- (u2) node [below, midway, black] {$\dots$};
            \coordinate[vtx] (u4) at (4, 0);
            \node at (u4) [below] {$S_{i_{m}}$};
            \draw[thick, red] (u3) -- (u4);
    \end{tikzpicture} \right] &\le \mathbb{E} \left[ \begin{tikzpicture}[scale=0.7, baseline=(current bounding box.center)]
            \coordinate[vtx] (u0) at (-1, 0);
            \node at (u0) [below] {$0$};
            \coordinate[vtx] (u1) at (0, 0);
            \node at (u1) [below] {$S_{i_1}$};  
            \draw[thick, red] (u1) -- (u0); 
            \coordinate[vtx] (u2) at (1, 0);
            \node at (u2) [below] {$S_{i_2}$};
            \draw[thick, red] (u1) -- (u2);     
            \coordinate[vtx] (u3) at (3, 0);
            \node at (u3) [below] {$S_{i_{m-2}}$};
            \draw[thick, red] (u3) -- (u2) node [below, midway, black] {$\dots$};
            \coordinate[vtx] (u4) at (4, 0);
            \node at (u4) [below] {$S_{i_{m-1}}$};
            \draw[thick, red] (u3) -- (u4);
    \end{tikzpicture} \right] C |i_{\sigma(m)} - i_{\sigma(m-1)}|_+^{-3/4}
    \end{align*}
    After repeating this process $m$ times and combining with the above, we get
    \begin{equation}
        \label{eq:cross-term-1}
        \mathbb{E} \left[ \prod_{k = 1}^{m} |S_{i_k} - \tilde{S}_{j_{k}}|_+^3 \right] \le C^m \prod_{k=1}^m |j_k - j_{k-1}|_+^{-3/4} \prod_{k=1}^m |i_{\sigma(k)} - i_{\sigma(k-1)}|_+^{-3/4}.
    \end{equation}
    Further, by following methods similar to \cite{DemboOkada2024}, Lemma 4.1, we have
    \begin{align*}
        \sum_{1 \le j_1 \le \dots \le j_m \le n} \prod_{k=1}^m |j_k - j_{k-1}|_+^{-3/4} &\le C n^{m/4} \sum_{1 \le j_1 \le \dots \le j_m \le n} \prod_{k=1}^m \left[ \frac{1}{n} \left| \frac{j_k}{n} - \frac{j_{k-1}}{n} \right|^{-3/4} \right] \\
        &\le Cn^{m/4} \int_{0 \le x_1 \le \dots \le x_m \le 1} |x_k - x_{k-1}|^{-3/4} \mathrm{d} x_1 \dots \mathrm{d} x_m \\
        &\le C^m (m!)^{-1/4} n^{m/4}.
    \end{align*}
    The last inequality comes from \cite{Kono1977}, Lemma. Now, summing over all $(m!)^2$ orderings of $j_1, \dots , j_m$ and $i_1, \dots , i_m$ completes the proof.
    \begin{align*}
        \sum_{1 \le i_1,\ldots i_m,j_1,\ldots, j_m \le n}\mathbb{E} \left[ \prod_{k = 1}^{m} |S_{i_k} - \tilde{S}_{j_{k}}|_+^3 \right] &\le C^m (m!)^2 \left( \sum_{1 \le j_1 \le \dots \le j_m \le n}\prod_{k=1}^m |j_k - j_{k-1}|_+^{-3/4} \right)^2 \\
        &\le C^m (m!)^{3/2} n^{m/2}.
    \end{align*}
    \end{proof}
    
    \subsection{Moments of \texorpdfstring{$\epsilon(\mathcal{S}_n, \tilde{\mathcal{S}}_n)$}{E(S[0,n], S'[0,n])}.}
    
    We begin the proof of Proposition~\ref{prop:error-term-moment}. Unlike Proposition~\ref{prop:cross-term-moment}, we will remove $2$ edges for every vertex. A white vertex $\tikz{\coordinate[openvtx] (x)}$ denotes a vertex colored either red or blue. In Lemma~\ref{lem:error-term-graph} itself, the vertex colors are irrelevant; their role will be explained in the proof of Proposition~\ref{prop:error-term-moment}.
    
    \begin{lemma}
    \label{lem:error-term-graph}
    Let $S$ be a random walk on $\mathbb{Z}^5$ starting at $S_0$. Assume all vertices except $S_i$ are deterministic, and that $G$ and $G'$ agree outside the neighborhood of $S_0$ and $S_i$ which are as in the figure below. Then, $\mathbb{E}[G] \le C |i|_+^{-1} [G']$. Needless to say, the roles of red and blue may be reversed.
    \begin{equation}
    \label{eq:error-term-graph-1}
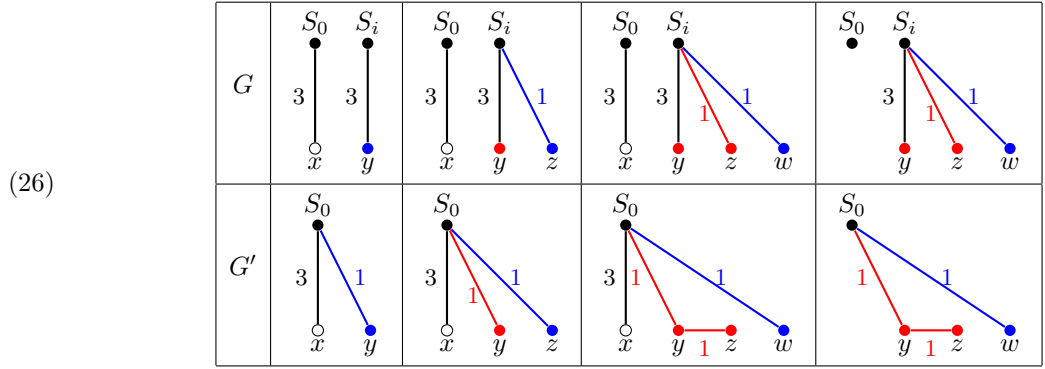
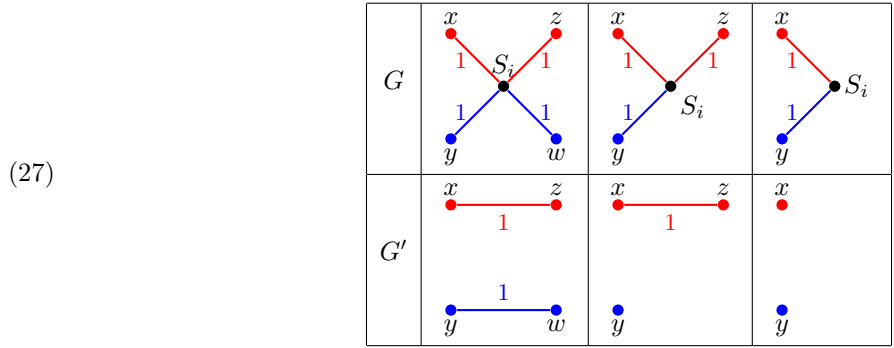
    \begin{array}{|c|c|c|c|c|c|c}
    \hline
       G  & \begin{tikzpicture}[scale=0.7, baseline=(current bounding box.center)]
            \coordinate[openvtx] (u1) at (0, 0);
            \node at (u1) [below] {$x$};
            \coordinate[vtx] (v1) at (0, 2);
            \node at (v1) [above] {$S_0$};
            \draw[thick] (v1) -- (u1) node [left, midway] {\small $3$};
            
            \coordinate[vtx, blue] (u2) at (1, 0);
            \node at (u2) [below] {$y$};
            \coordinate[vtx] (v2) at (1, 2);
            \node at (v2) [above] {$S_i$};
            \draw[thick] (v2) -- (u2) node [left, midway] {\small $3$};
        \end{tikzpicture} & \begin{tikzpicture}[scale=0.7, baseline=(current bounding box.center)]
            \coordinate[openvtx] (u1) at (0, 0);
            \node at (u1) [below] {$x$};
            \coordinate[vtx] (v1) at (0, 2);
            \node at (v1) [above] {$S_0$};
            \draw[thick] (v1) -- (u1) node [left, midway] {\small $3$};
            
            \coordinate[vtx, red] (u2) at (1, 0);
            \node at (u2) [below] {$y$};
            \coordinate[vtx] (v2) at (1, 2);
            \node at (v2) [above] {$S_i$};
            \draw[thick] (v2) -- (u2) node [left, midway] {\small $3$};
    
            \coordinate[vtx, blue] (u3) at (2, 0);
            \node at (u3) [below] {$z$};
            \draw[thick, blue] (u3) -- (v2) node [right, midway] {\small $1$};
        \end{tikzpicture} & \begin{tikzpicture}[scale=0.7, baseline=(current bounding box.center)]
            \coordinate[vtx] (v1) at (0, 2);
            \node at (v1) [above] {$S_0$};
            \coordinate[openvtx] (u1) at (0, 0);
            \node at (u1) [below] {$x$};
            \draw[thick] (v1) -- (u1) node [left, midway] {\small $3$};
            
            \coordinate[vtx, red] (u2) at (1, 0);
            \node at (u2) [below] {$y$};
            \coordinate[vtx] (v2) at (1, 2);
            \node at (v2) [above] {$S_i$};
            \draw[thick] (v2) -- (u2) node [left, midway] {\small $3$};
    
            \coordinate[vtx, red] (u3) at (2, 0);
            \node at (u3) [below] {$z$};
            \draw[thick, red] (u3) -- (v2) node [below, midway] {\small $1$};
    
            \coordinate[vtx, blue] (u4) at (3, 0);
            \node at (u4) [below] {$w$};
            \draw[thick, blue] (u4) -- (v2) node [right, midway] {\small $1$};
        \end{tikzpicture} & \begin{tikzpicture}[scale=0.7, baseline=(current bounding box.center)]
            \coordinate[vtx] (v1) at (0, 2);
            \node at (v1) [above] {$S_0$};
            
            \coordinate[vtx, red] (u2) at (1, 0);
            \node at (u2) [below] {$y$};
            \coordinate[vtx] (v2) at (1, 2);
            \node at (v2) [above] {$S_i$};
            \draw[thick] (v2) -- (u2) node [left, midway] {\small $3$};
    
            \coordinate[vtx, red] (u3) at (2, 0);
            \node at (u3) [below] {$z$};
            \draw[thick, red] (u3) -- (v2) node [below, midway] {\small $1$};
    
            \coordinate[vtx, blue] (u4) at (3, 0);
            \node at (u4) [below] {$w$};
            \draw[thick, blue] (u4) -- (v2) node [right, midway] {\small $1$};
        \end{tikzpicture} \\
        \hline
       G' & \begin{tikzpicture}[scale=0.7, baseline=(current bounding box.center)]
            \coordinate[openvtx] (u1) at (0, 0);
            \node at (u1) [below] {$x$};
            \coordinate[vtx] (v1) at (0, 2);
            \node at (v1) [above] {$S_0$};
            \draw[thick] (v1) -- (u1) node [left, midway] {\small $3$};
            
            \coordinate[vtx, blue] (u2) at (1, 0);
            \node at (u2) [below] {$y$};
            \draw[thick, blue] (v1) -- (u2) node [right, midway] {\small $1$};
        \end{tikzpicture} & \begin{tikzpicture}[scale=0.7, baseline=(current bounding box.center)]
            \coordinate[openvtx] (u1) at (0, 0);
            \node at (u1) [below] {$x$};
            \coordinate[vtx] (v1) at (0, 2);
            \node at (v1) [above] {$S_0$};
            \draw[thick] (v1) -- (u1) node [left, midway] {\small $3$};
            
            \coordinate[vtx, red] (u2) at (1, 0);
            \node at (u2) [below] {$y$};
            \draw[thick, red] (v1) -- (u2) node [below, midway] {\small $1$};
    
            \coordinate[vtx, blue] (u3) at (2, 0);
            \node at (u3) [below] {$z$};
            \draw[thick, blue] (u3) -- (v1) node [right, midway] {\small $1$};
        \end{tikzpicture} & \begin{tikzpicture}[scale=0.7, baseline=(current bounding box.center)]
            \coordinate[openvtx] (u1) at (0, 0);
            \node at (u1) [below] {$x$};
            \coordinate[vtx] (v1) at (0, 2);
            \node at (v1) [above] {$S_0$};
            \draw[thick] (v1) -- (u1) node [left, midway] {\small $3$};
            
            \coordinate[vtx, red] (u2) at (1, 0);
            \node at (u2) [below] {$y$};
            \draw[thick, red] (v1) -- (u2) node [left, midway] {\small $1$};
    
            \coordinate[vtx, red] (u3) at (2, 0);
            \node at (u3) [below] {$z$};
            \draw[thick, red] (u3) -- (u2) node [below, midway] {\small $1$};
            
            \coordinate[vtx, blue] (u4) at (3, 0);
            \node at (u4) [below] {$w$};
            \draw[thick, blue] (u4) -- (v1) node [right, midway] {\small $1$};
        \end{tikzpicture} & \begin{tikzpicture}[scale=0.7, baseline=(current bounding box.center)]
            \coordinate[vtx] (v1) at (0, 2);
            \node at (v1) [above] {$S_0$};
            
            \coordinate[vtx, red] (u2) at (1, 0);
            \node at (u2) [below] {$y$};
            \draw[thick, red] (v1) -- (u2) node [left, midway] {\small $1$};
    
            \coordinate[vtx, red] (u3) at (2, 0);
            \node at (u3) [below] {$z$};
            \draw[thick, red] (u3) -- (u2) node [below, midway] {\small $1$};
            
            \coordinate[vtx, blue] (u4) at (3, 0);
            \node at (u4) [below] {$w$};
            \draw[thick, blue] (u4) -- (v1) node [right, midway] {\small $1$};
        \end{tikzpicture} \\
        \hline
    \end{array}
    \end{equation}
    \begin{equation}
        \label{eq:error-term-graph-2}
        \begin{array}{|c|c|c|c|}
        \hline
       G  & \begin{tikzpicture}[scale=0.7, baseline=(current bounding box.center)]
            \coordinate[vtx, red] (u1) at (0, 2);
            \coordinate[vtx, red] (u2) at (2, 2);
            \coordinate[vtx, blue] (u3) at (0, 0);
            \coordinate[vtx, blue] (u4) at (2, 0);
            \node at (u1) [above] {$x$};
            \node at (u2) [above] {$z$};
            \node at (u3) [below] {$y$};
            \node at (u4) [below] {$w$};
            \coordinate[vtx] (v) at (1, 1);
            \node at (v) [above] {$S_i$};
            
            \draw[thick, red] (v) -- (u1) node [left, midway] {\small $1$};
            \draw[thick, red] (v) -- (u2) node [right, midway] {\small $1$};
            \draw[thick, blue] (v) -- (u3) node [left, midway] {\small $1$};
            \draw[thick, blue] (v) -- (u4) node [right, midway] {\small $1$};
        \end{tikzpicture} & \begin{tikzpicture}[scale=0.7, baseline=(current bounding box.center)]
            \coordinate[vtx, red] (u1) at (0, 2);
            \coordinate[vtx, red] (u2) at (2, 2);
            \coordinate[vtx, blue] (u3) at (0, 0);
            \node at (u1) [above] {$x$};
            \node at (u2) [above] {$z$};
            \node at (u3) [below] {$y$};
            \coordinate[vtx] (v) at (1, 1);
            \node at (v) [below right] {$S_i$};
            
            \draw[thick, red] (v) -- (u1) node [left, midway] {\small $1$};
            \draw[thick, red] (v) -- (u2) node [right, midway] {\small $1$};
            \draw[thick, blue] (v) -- (u3) node [left, midway] {\small $1$};
        \end{tikzpicture} & \begin{tikzpicture}[scale=0.7, baseline=(current bounding box.center)]
            \coordinate[vtx, red] (u1) at (0, 2);
            \coordinate[vtx, blue] (u3) at (0, 0);
            \node at (u1) [above] {$x$};
            \node at (u3) [below] {$y$};
            \coordinate[vtx] (v) at (1, 1);
            \node at (v) [right] {$S_i$};
            
            \draw[thick, red] (v) -- (u1) node [left, midway] {\small $1$};
            \draw[thick, blue] (v) -- (u3) node [left, midway] {\small $1$};
        \end{tikzpicture}  \\
        \hline
       G' & \begin{tikzpicture}[scale=0.7, baseline=(current bounding box.center)]
            \coordinate[vtx, red] (u1) at (0, 2);
            \coordinate[vtx, red] (u2) at (2, 2);
            \coordinate[vtx, blue] (u3) at (0, 0);
            \coordinate[vtx, blue] (u4) at (2, 0);
            \node at (u1) [above] {$x$};
            \node at (u2) [above] {$z$};
            \node at (u3) [below] {$y$};
            \node at (u4) [below] {$w$};
            
            \draw[thick, red] (u2) -- (u1) node [below, midway] {\small $1$};
            \draw[thick, blue] (u4) -- (u3) node [above, midway] {\small $1$};
        \end{tikzpicture} & \begin{tikzpicture}[scale=0.7, baseline=(current bounding box.center)]
            \coordinate[vtx, red] (u1) at (0, 2);
            \coordinate[vtx, red] (u2) at (2, 2);
            \coordinate[vtx, blue] (u3) at (0, 0);
            \node at (u1) [above] {$x$};
            \node at (u2) [above] {$z$};
            \node at (u3) [below] {$y$};
            
            \draw[thick, red] (u2) -- (u1) node [below, midway] {\small $1$};
        \end{tikzpicture} & \begin{tikzpicture}[scale=0.7, baseline=(current bounding box.center)]
            \coordinate[vtx, red] (u1) at (0, 2);
            \coordinate[vtx, blue] (u3) at (0, 0);
            \node at (u1) [above] {$x$};
            \node at (u3) [below] {$y$};
            
            \coordinate[vtx, white] (v) at (1, 1);
            \node[white] at (v) [right] {$S_i$};
        \end{tikzpicture} \\
        \hline
    \end{array}
    \end{equation}
    \end{lemma}
    
    \begin{proof}
        The proof is similar to Lemma~\ref{lem:cross-term-graph}. Since all proofs are similar, we will only explain the third operation in \eqref{eq:error-term-graph-1}. This equates to showing
        \[
        \mathbb{E} |S_i - y|_+^{-3} |S_i - z|^{-1} |S_i - w|_+^{-1} \le C |i|_+^{-1} |y|_+^{-1} |y-z|_+^{-1} |w|_+^{-1}.
        \]
        First note that $|y-z|_+ \le |S_i -y|_+ + |S_i - z|_+$ by the triangle inequality. Thus,
        \begin{align*}
            |y-z|_+ &\mathbb{E} |S_i - y|_+^{-3} |S_i - z|^{-1} |S_i - w|_+^{-1} \\
            &\le \mathbb{E}|S_i - y|_+^{-2} |S_i - z|^{-1} |S_i - w|_+^{-1} + \mathbb{E}|S_i - y|_+^{-3} |S_i - w|_+^{-1} \\
            &\le \left(\mathbb{E}|S_i-y|_+^{-4} \right)^{1/4}\left(\mathbb{E}|S_i - w|_+^{-4} \right)^{1/4} \left\{\left(\mathbb{E}|S_i-y|_+^{-4} \right)^{1/4}\left(\mathbb{E}|S_i - z|_+^{-4} \right)^{1/4} +  \left(\mathbb{E} |S_i - y|_+^{-4} \right)^{1/2} \right\} \\
            &\le C |y|_+^{-1} |w|_+^{-1} |i|_+^{-1}
        \end{align*}
        by Lemma~\ref{lem:rw-moment} and H\"{o}lder's inequality.
    \end{proof}
    
    \begin{proof}[Proof of Proposition~\ref{prop:error-term-moment}]
        Similarly to Proposition~\ref{prop:cross-term-moment}, we will progressively reduce this graph through a series of operations. Black edges have weight $3$ and red or blue edges have weight $1$. The key challenge here is that the graph is not symmetric with respect to $j_1 , \dots , j_{2m}$. In order to better manage this, we will draw the graph as
        \[
        \left[
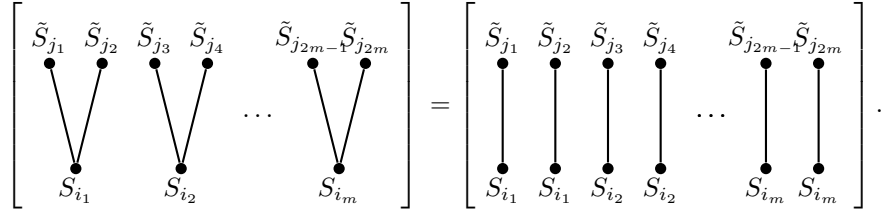
    \begin{tikzpicture}[scale=0.7, baseline=(current bounding box.center)]
        \coordinate[vtx] (u0) at (1.5, 0);
        \node at (u0) [below] {$S_{i_{1}}$};
        \coordinate[vtx] (v0) at (1, 2);
        \node at (v0) [above] {$\tilde{S}_{j_1}$};
        \coordinate[vtx] (w0) at (2, 2);
        \node at (w0) [above] {$\tilde{S}_{j_2}$};
        \draw[thick] (v0) -- (u0);
        \draw[thick] (w0) -- (u0);
        \coordinate[vtx] (u1) at (3.5, 0);
        \node at (u1) [below] {$S_{i_2}$};
        \coordinate[vtx] (v1) at (3, 2);
        \node at (v1) [above] {$\tilde{S}_{j_3}$};
        \coordinate[vtx] (w1) at (4, 2);
        \node at (w1) [above] {$\tilde{S}_{j_4}$};
        \draw[thick] (v1) -- (u1);
        \draw[thick] (w1) -- (u1);
        \node at (5, 1) {$\dots$};
        \coordinate[vtx] (u) at (6.5, 0);
        \node at (u) [below] {$S_{i_{m}}$};
            \coordinate[vtx] (v) at (6, 2);
            \node at (v) [above] {$\tilde{S}_{j_{2m-1}}$};
            \coordinate[vtx] (w) at (7, 2);
            \node at (w) [above] {$\tilde{S}_{j_{2m}}$};
            \draw[thick] (v) -- (u);
            \draw[thick] (w) -- (u);
    \end{tikzpicture}\right] \\
    = \left[
    \begin{tikzpicture}[scale=0.7, baseline=(current bounding box.center)]
        \coordinate[vtx] (v0) at (1, 2);
        \node at (v0) [above] {$\tilde{S}_{j_1}$};
        \coordinate[vtx] (v1) at (2, 2);
        \node at (v1) [above] {$\tilde{S}_{j_2}$};
        \coordinate[vtx] (v2) at (3, 2);
        \node at (v2) [above] {$\tilde{S}_{j_3}$};
        \coordinate[vtx] (v3) at (4, 2);
        \node at (v3) [above] {$\tilde{S}_{j_4}$};
        \coordinate[vtx] (u0) at (1, 0);
        \node at (u0) [below] {$S_{i_1}$};
        \coordinate[vtx] (u1) at (2, 0);
        \node at (u1) [below] {$S_{i_1}$};
        \coordinate[vtx] (u2) at (3, 0);
        \node at (u2) [below] {$S_{i_2}$};
        \coordinate[vtx] (u3) at (4, 0);
        \node at (u3) [below] {$S_{i_2}$};
        \node at (5, 1) {$\dots$};
        \coordinate[vtx] (v4) at (6, 2);
        \node at (v4) [above] {$\tilde{S}_{j_{2m-1}}$};
        \coordinate[vtx] (v5) at (7, 2);
        \node at (v5) [above] {$\tilde{S}_{j_{2m}}$};
        \node at (5, 1) {$\dots$};
        \coordinate[vtx] (u4) at (6, 0);
        \node at (u4) [below] {$S_{i_m}$};
        \coordinate[vtx] (u5) at (7, 0);
        \node at (u5) [below] {$S_{i_m}$};
        \draw[thick] (u0) -- (v0);
        \draw[thick] (u1) -- (v1);
        \draw[thick] (u2) -- (v2);
        \draw[thick] (u3) -- (v3);
        \draw[thick] (u4) -- (v4);
        \draw[thick] (u5) -- (v5);
    \end{tikzpicture}\right].
    \]
        
    In this way, we may assume $j_1 \le j_2 \dots \le j_{2m}$ in exchange for reordering $S_{i_1}, \dots, S_{i_m}$. That is, if suffices to consider graphs of the form
    
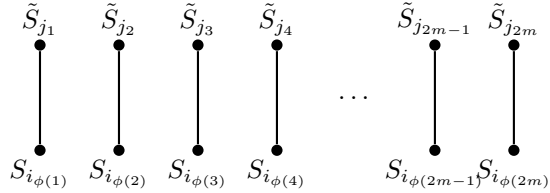
\[
    \begin{tikzpicture}[scale=0.7, baseline=(current bounding box.center)]
        \coordinate[vtx] (v0) at (0, 2);
        \node at (v0) [above] {$\tilde{S}_{j_1}$};
        \coordinate[vtx] (v1) at (1.5, 2);
        \node at (v1) [above] {$\tilde{S}_{j_2}$};
        \coordinate[vtx] (u0) at (0, 0);
        \node at (u0) [below] {$S_{i_{\phi(1)}}$};
        \coordinate[vtx] (u1) at (1.5, 0);
        \node at (u1) [below] {$S_{i_{\phi(2)}}$};
        \coordinate[vtx] (v2) at (3, 2);
        \node at (v2) [above] {$\tilde{S}_{j_3}$};
        \coordinate[vtx] (v3) at (4.5, 2);
        \node at (v3) [above] {$\tilde{S}_{j_4}$};
        \coordinate[vtx] (u2) at (3, 0);
        \node at (u2) [below] {$S_{i_{\phi(3)}}$};
        \coordinate[vtx] (u3) at (4.5, 0);
        \node at (u3) [below] {$S_{i_{\phi(4)}}$};
        \node at (6, 1) {$\dots$};
        \coordinate[vtx] (v4) at (7.5, 2);
        \node at (v4) [above] {$\tilde{S}_{j_{2m-1}}$};
        \coordinate[vtx] (v5) at (9, 2);
        \node at (v5) [above] {$\tilde{S}_{j_{2m}}$};
        \coordinate[vtx] (u4) at (7.5, 0);
        \node at (u4) [below] {$S_{i_{\phi(2m-1)}}$};
        \coordinate[vtx] (u5) at (9, 0);
        \node at (u5) [below] {$S_{i_{\phi(2m)}}$};
        \draw[thick] (u0) -- (v0);
        \draw[thick] (u1) -- (v1);
        \draw[thick] (u2) -- (v2);
        \draw[thick] (u3) -- (v3);
        \draw[thick] (u4) -- (v4);
        \draw[thick] (u5) -- (v5);
    \end{tikzpicture}
    \]
    for arbitrary two-to-one maps $\phi : \{ 1, 2, \dots , 2m \} \to \{ 1, \dots, m \}$. Since we are considering all such maps, we can also assume $i_1 \le \dots \le i_m$ without loss of generality. 
    
    Now, we provide an algorithm that works for all such graphs. Consider the following coloring on the multi-set $\{S_{i_{\phi(1)}}, \dots , S_{i_{\phi(2m)}}\}$. First, color $S_{i_{\phi(2m)}}$ blue and $S_{i_{\phi(2m-1)}}$ red. Now, going from right to left, if $S_{i_{\phi(k)}}$ has not already appeared, give it an arbitrary color. If it has appeared, color it the opposite of what it was colored previously. Since $S_{i_{\phi(2m)}}$ and $S_{i_{\phi(2m-1)}}$ have different colors, \eqref{eq:error-term-graph-1} gives a unique way of reducing this graph. Further, this process gives two paths, one red and one blue, that are joined at $0$. Since each vertex $S_{i_k}$ appears once as red and once as blue, each monochromatic path is a permutation of $\{ S_{i_1} , \dots , S_{i_m} \}$. In summary, we get the following inequality after reduction.
    \begin{equation}
    \label{eq:error-term-1}
    \mathbb{E} \left[ \begin{tikzpicture}[scale=0.7, baseline=(current bounding box.center)]
        \coordinate[vtx] (v1) at (1, 2);
        \node at (v1) [above] {$\tilde{S}_{j_1}$};
        \coordinate[vtx] (v2) at (2, 2);
        \node at (v2) [above] {$\tilde{S}_{j_2}$};
        \coordinate[vtx] (v4) at (4, 2);
        \node at (v4) [above] {$\tilde{S}_{j_{2m}}$};
        \coordinate[openvtx] (u1) at (1, 0);
        \node at (u1) [below] {$S_{i_{\phi(1)}}$};
        \coordinate[openvtx] (u2) at (2, 0);
        \node at (u2) [below] {$S_{i_{\phi(2)}}$};
        \coordinate[openvtx] (u4) at (4, 0);
        \node at (u4) [below] {$S_{i_{\phi(2m)}}$};
        \draw[thick] (v1) -- (u1);
        \draw[thick] (v2) -- (u2);
        \draw[thick] (v4) -- (u4);
        \node at (3, 1) {$\dots$};
    \end{tikzpicture}
    \right] \le \mathbb{E} \left[
    \begin{tikzpicture}[scale=0.7, baseline=(current bounding box.center)]
        \coordinate[vtx] (w) at (0, 1);
        \node at (w) [left] {$0$};
        \coordinate[vtx, red] (v1) at (1, 2);
        \node at (v1) [above] {$S_{i_{\sigma(1)}}$};
        \coordinate[vtx, red] (v2) at (2.5, 2);
        \node at (v2) [above] {$S_{i_{\sigma(2)}}$};
        \coordinate[vtx, red] (v4) at (5, 2);
        \node at (v4) [above] {$S_{i_{\sigma(m)}}$};
        \coordinate[vtx, blue] (u1) at (1, 0);
        \node at (u1) [below] {$S_{i_{\tau(1)}}$};
        \coordinate[vtx, blue] (u2) at (2.5, 0);
        \node at (u2) [below] {$S_{i_{\tau(2)}}$};
        \coordinate[vtx, blue] (u4) at (5, 0);
        \node at (u4) [below] {$S_{i_{\tau(m)}}$};
        \draw[thick, red] (w) -- (v1) -- (v2) -- (v4);
        \draw[thick, blue] (w) -- (u1) -- (u2) -- (u4);
        \node at (3.75, 2.5) {$\dots$};
        \node at (3.75, -0.5) {$\dots$};
    \end{tikzpicture}
    \right] C^{2m} \prod_{k=1}^{2m} |j_k - j_{k-1}|_+^{-1}.
    \end{equation}
    Here, $\sigma$ and $\tau$ are permutations of $\{ 1, \dots , m \}$ such that $S_{i_1} , \dots , S_{i_m}$ appear in the same order on both sides.That is, $\{\sigma(i)\}_{i=1}^m$ and $\{ \tau(i)\}_{i=1}^m$ are disjoint subsequences of $\{ \phi(i)\}_{i=1}^{2m}$. Now we can use \eqref{eq:error-term-graph-2} to reduce this graph further. This gives
    \[  
    \mathbb{E} \left[
    \begin{tikzpicture}[scale=0.7, baseline=(current bounding box.center)]
        \coordinate[vtx] (w) at (0, 1);
        \node at (w) [left] {$0$};
        \coordinate[vtx, red] (v1) at (1, 2);
        \node at (v1) [above] {$S_{i_{\sigma(1)}}$};
        \coordinate[vtx, red] (v2) at (2.5, 2);
        \node at (v2) [above] {$S_{i_{\sigma(2)}}$};
        \coordinate[vtx, red] (v4) at (5, 2);
        \node at (v4) [above] {$S_{i_{\sigma(m)}}$};
        \coordinate[vtx, blue] (u1) at (1, 0);
        \node at (u1) [below] {$S_{i_{\tau(1)}}$};
        \coordinate[vtx, blue] (u2) at (2.5, 0);
        \node at (u2) [below] {$S_{i_{\tau(2)}}$};
        \coordinate[vtx, blue] (u4) at (5, 0);
        \node at (u4) [below] {$S_{i_{\tau(m)}}$};
        \draw[thick, red] (w) -- (v1) -- (v2) -- (v4);
        \draw[thick, blue] (w) -- (u1) -- (u2) -- (u4);
        \node at (3.75, 2.5) {$\dots$};
        \node at (3.75, -0.5) {$\dots$};
    \end{tikzpicture}
    \right] \le C^m \prod_{k=1}^{m} |i_k - i_{k-1}|_+^{-1}.
    \]
    One key point here is that a vertex never neighbors itself, since they are on opposite sides of $0$. Combined with \eqref{eq:error-term-1} and summing over $i_1 \le \dots \le i_m$, $j_1 \le \dots \le j_{2m}$ and their $m! \times (2m)!$ possible reorderings, we have
    \begin{align*}
    \mathbb{E} \left[ \sum_{1 \le i_1 , \dots , i_m \le n} \right. & \left. \sum_{1 \le j_1 , \dots , j_{2m} \le n} \prod_{k = 1}^{m} |S_{i_k} - \tilde{S}_{j_{2k-1}}|_+^3 |S_{i_k} - \tilde{S}_{j_{2k}}|_+^3 \right] \\
    &\le C^{m} (m!)(2m)! \left( \prod_{k=1}^{2m} \sum_{j_k = j_{k-1}}^n |j_k - j_{k-1}|_+^{-1} \right) \left( \prod_{k=1}^{m} \sum_{i_k = i_{k-1}}^n |i_k - i_{k-1}|_+^{-1} \right) \\
    &\le C^m (m!)(2m)! \left( \sum_{j=1}^n |j|_+^{-1} \right)^{3m} \\
    &\le C^m (m!)^{3} (\log n)^{3m}.
    \end{align*}
    \end{proof}
    
    \begin{lemma}
    \label{lem:aux-error-moment}
    Recall that $[n]^{(l, p)} = [2^{-l} n(p-1), 2^{-l}n p]$ denotes the $p$-th piece of the interval $\{ 0, 1, \dots ,n \}$. For any $l$, let
    \begin{equation}
    \label{eq:aux-error-def}
    \begin{aligned}
    X_l := \sum_{p = 1}^{2^{l-1}} \sum_{i =1}^n \sum_{j_1 \in [0,n]^{(l, 2p-1)}} \sum_{j_2 \in [0,n]^{(l, 2p)}} G_D (S_i - \tilde{S}_{j_1}) G_D(\tilde{S}_{j_1} - \tilde{S}_{j_2}) \\
    Y_l := \sum_{p = 1}^{2^{l-1}} \sum_{i =1}^n \sum_{j_1 \in [0,n]^{(l, 2p-1)}} \sum_{j_2 \in [0,n]^{(l, 2p)}} G_D (S_i - \tilde{S}_{j_2}) G_D(\tilde{S}_{j_1} - \tilde{S}_{j_2}) \\
    Z_l := \sum_{p = 1}^{2^{l-1}}  \sum_{i_1 \in [0,n]^{(l, 2p-1)}} \sum_{i_2 \in [0,n]^{(l, 2p)}}\sum_{j=1}^n G_D (S_{i_1} - \tilde{S}_j) G_D(S_{i_1} - S_{i_2}) \\
    W_l := \sum_{p = 1}^{2^{l-1}}  \sum_{i_1 \in [0,n]^{(l, 2p-1)}} \sum_{i_2 \in [0,n]^{(l, 2p)}}\sum_{j=1}^n  G_D (S_{i_2} - \tilde{S}_j) G_D(S_{i_1} - S_{i_2}) \\
    \end{aligned}
    \end{equation}
    Then,
    \begin{equation}
    \label{eq:aux-error-moment}
    \mathbb{E}[X_l^m], \mathbb{E}[Y_l^m], \mathbb{E}[Z_l^m], \mathbb{E}[W_l^m] \le  C^m 2^l e^{2^l / m} (m!)^3 (\log n)^{3m}.
    \end{equation}
    \end{lemma}
    
    \begin{proof}
        We only show this estimate for $X_l$, as the proof for the rest are almost identical. First assume $l = 1$. In this case, it is clear that
        \[  
        \mathbb{E}[X_1^m] \le \sum_{0 \le i_1 , \dots , i_m \le n} \sum_{0 \le j_1 , \dots , j_m \le n/2} \sum_{n/2 \le j_{m+1} , \dots , j_{2m} \le n} C^m \mathbb{E} \left[ \prod_{k=1}^m |S_{i_k} - \tilde{S}_{j_k}|_+^{-3} |\tilde{S}_{j_k} - \tilde{S}_{j_{m+k}}|_+^{-3} \right].
        \]
        Similarly to before, we may use the graph inequalities $\mathbb{E}[G] \le C|i|_+^{-1} \mathbb{E}[G']$, where
        \[
    \begin{array}{|c|c|c|c|c|c|c}
    \hline
       G  & \begin{tikzpicture}[scale=0.7, baseline=(current bounding box.center)]
            \coordinate[vtx] (u1) at (0, 0);
            \node at (u1) [below] {$x$};
            \coordinate[vtx] (v1) at (0, 2);
            \node at (v1) [above] {$S_0$};
            \draw[thick] (v1) -- (u1) node [left, midway] {\small $3$};
            
            \coordinate[vtx] (u2) at (1, 0);
            \node at (u2) [below] {$y$};
            \coordinate[vtx] (v2) at (1, 2);
            \node at (v2) [above] {$S_i$};
            \draw[thick] (v2) -- (u2) node [left, midway] {\small $3$};
        \end{tikzpicture} & \begin{tikzpicture}[scale=0.7, baseline=(current bounding box.center)]
            \coordinate[vtx] (v1) at (0, 2);
            \node at (v1) [above] {$S_0$};
            \coordinate[vtx] (u1) at (0, 0);
            \node at (u1) [below] {$x$};
            \draw[thick] (v1) -- (u1) node [left, midway] {\small $3$};
            
            \coordinate[vtx] (u2) at (1, 0);
            \node at (u2) [below] {$y$};
            \coordinate[vtx] (v2) at (1, 2);
            \node at (v2) [above] {$S_i$};
            \draw[thick] (v2) -- (u2) node [left, midway] {\small $3$};
    
            \coordinate[vtx] (u3) at (2, 0);
            \node at (u3) [below] {$z$};
            \draw[thick, red] (u3) -- (v2) node [below, midway] {\small $1$};
        \end{tikzpicture}\\
        \hline
       G' & \begin{tikzpicture}[scale=0.7, baseline=(current bounding box.center)]
            \coordinate[vtx] (u1) at (0, 0);
            \node at (u1) [below] {$x$};
            \coordinate[vtx] (v1) at (0, 2);
            \node at (v1) [above] {$S_0$};
            \draw[thick] (v1) -- (u1) node [left, midway] {\small $3$};
            
            \coordinate[vtx] (u2) at (1, 0);
            \node at (u2) [below] {$y$};
            \draw[thick] (v1) -- (u2) node [right, midway] {\small $1$};
        \end{tikzpicture} & \begin{tikzpicture}[scale=0.7, baseline=(current bounding box.center)]
            \coordinate[vtx] (u1) at (0, 0);
            \node at (u1) [below] {$x$};
            \coordinate[vtx] (v1) at (0, 2);
            \node at (v1) [above] {$S_0$};
            \draw[thick] (v1) -- (u1) node [left, midway] {\small $3$};
            
            \coordinate[vtx] (u2) at (1, 0);
            \node at (u2) [below] {$y$};
            \draw[thick, red] (v1) -- (u2) node [left, midway] {\small $1$};
    
            \coordinate[vtx] (u3) at (2, 0);
            \node at (u3) [below] {$z$};
            \draw[thick, red] (u3) -- (u2) node [below, midway] {\small $1$};
        \end{tikzpicture}\\
        \hline
    \end{array}
    \]
    and condition on $\tilde{S}$ to get 
    \[  
        \mathbb{E}\left[
    \begin{tikzpicture}[scale=0.7, baseline=(current bounding box.center)]
        \coordinate[vtx] (u1) at (3.5, 0);
        \node at (u1) [below] {$\tilde{S}_{j_{1}}$};
        \coordinate[vtx] (v1) at (3, 2);
        \node at (v1) [above] {$S_{i_1}$};
        \coordinate[vtx] (w1) at (4, 2);
        \node at (w1) [above] {$\tilde{S}_{j_{m+1}}$};
        \draw[thick] (v1) -- (u1);
        \draw[thick] (w1) -- (u1);
        \node at (5, 1) {$\dots$};
        \coordinate[vtx] (u) at (6.5, 0);
        \node at (u) [below] {$\tilde{S}_{j_{m}}$};
            \coordinate[vtx] (v) at (6, 2);
            \node at (v) [above] {$S_{i_{m}}$};
            \coordinate[vtx] (w) at (7, 2);
            \node at (w) [above] {$\tilde{S}_{j_{2m}}$};
            \draw[thick] (v) -- (u);
            \draw[thick] (w) -- (u);
    \end{tikzpicture}\right]  \le \mathbb{E} \left[ 
    \begin{tikzpicture}[scale=0.7, baseline=(current bounding box.center)]
        \coordinate[vtx] (u1) at (3, 0);
        \coordinate[vtx] (u0) at (2, 0);
        \node at (u0) [below] {$0$};
        \draw[thick, red] (u0) -- (u1);
        \coordinate[vtx] (u2) at (4, 0);
        \node at (u1) [below] {$\tilde{S}_{j_{1}}$};
        \node at (u2) [below] {$\tilde{S}_{j_{2}}$};
        \coordinate[vtx] (v1) at (3, 2);
        \node at (v1) [above] {$\tilde{S}_{j_{m+1}}$};
        \coordinate[vtx] (w1) at (4, 2);
        \node at (w1) [above] {$\tilde{S}_{j_{m+2}}$};
        \draw[thick] (v1) -- (u1);
        \draw[thick] (w1) -- (u2);
        \draw[thick, red] (u1) -- (u2);
        \node at (5, 1) {$\dots$};
        \coordinate[vtx] (u) at (6, 0);
        \node at (u) [below] {$\tilde{S}_{j_{m}}$};
            \coordinate[vtx] (w) at (6, 2);
            \node at (w) [above] {$\tilde{S}_{j_{2m}}$};
            \draw[thick] (w) -- (u);
            \draw[thick, red] (u2) -- (u);
    \end{tikzpicture}
    \right] C^m \prod_{k=1}^m |i_k - i_{k-1}|_+^{-1}.
    \]
    Here we've assumed that $i_1 \le \dots \le i_m$ without loss of generality Now since $j_1 , \dots, j_m \le n/2 \le j_{m+1} \le \dots \le j_{2m}$, we can condition on $\tilde{S}[0, n/2]$ and repeat a similar process (except with blue lines). Assume $j_{\tau(1)} \le j_{\tau(2)} \le \dots \le j_{\tau(2m)}$ for some permutation $\tau$ (there are $(m!)^2$ possible such permutations). This gives
    \begin{multline*}
    \mathbb{E} \left[ 
    \begin{tikzpicture}[scale=0.7, baseline=(current bounding box.center)]
        \coordinate[vtx] (u1) at (3, 0);
        \coordinate[vtx] (u0) at (2, 0);
        \node at (u0) [below] {$0$};
        \draw[thick, red] (u0) -- (u1);
        \coordinate[vtx] (u2) at (4, 0);
        \node at (u1) [below] {$\tilde{S}_{j_{1}}$};
        \node at (u2) [below] {$\tilde{S}_{j_{2}}$};
        \coordinate[vtx] (v1) at (3, 2);
        \node at (v1) [above] {$\tilde{S}_{j_{m+1}}$};
        \coordinate[vtx] (w1) at (4, 2);
        \node at (w1) [above] {$\tilde{S}_{j_{m+2}}$};
        \draw[thick] (v1) -- (u1);
        \draw[thick] (w1) -- (u2);
        \draw[thick, red] (u1) -- (u2);
        \node at (5, 1) {$\dots$};
        \coordinate[vtx] (u) at (6, 0);
        \node at (u) [below] {$\tilde{S}_{j_{m}}$};
            \coordinate[vtx] (w) at (6, 2);
            \node at (w) [above] {$\tilde{S}_{j_{2m}}$};
            \draw[thick] (w) -- (u);
            \draw[thick, red] (u2) -- (u);
    \end{tikzpicture}
    \right] \\
    \le \mathbb{E} \left[
    \begin{tikzpicture}[scale=0.7, baseline=(current bounding box.center)]
        \coordinate[vtx] (w1) at (0, 1);
        \node at (w1) [left] {$0$};
        \coordinate[vtx] (w2) at (6, 1);
        \node at (w2) [left] {$\tilde{S}_{n/2}$};
        \coordinate[vtx, blue] (v1) at (1, 2);
        \node at (v1) [above] {$\tilde{S}_{j_{\sigma(1)}}$};
        \coordinate[vtx, blue] (v2) at (2.5, 2);
        \node at (v2) [above] {$\tilde{S}_{j_{\sigma(2)}}$};
        \coordinate[vtx, blue] (v4) at (5, 2);
        \node at (v4) [above] {$\tilde{S}_{j_{\sigma(m)}}$};
        \coordinate[vtx, red] (u1) at (1, 0);
        \node at (u1) [below] {$\tilde{S}_{j_1}$};
        \coordinate[vtx, red] (u2) at (2.5, 0);
        \node at (u2) [below] {$\tilde{S}_{j_2}$};
        \coordinate[vtx, red] (u4) at (5, 0);
        \node at (u4) [below] {$\tilde{S}_{j_{m}}$};
        \draw[thick, blue] (v1) -- (v2) -- (v4) -- (w2);
        \draw[thick, red] (w1) -- (u1) -- (u2) -- (u4);
        \node at (3.75, 2.5) {$\dots$};
        \node at (3.75, -0.5) {$\dots$};
    \end{tikzpicture}
    \right] C^{2m} |j_{\tau(m+1)} - n/2|_+^{-1} \prod_{k=m+2}^{2m} |j_{\tau(k)} - j_{\tau(k-1)}|_+^{-1}.
    \end{multline*}
    for some permutation $\sigma$ of $\{ 1, \dots , m \}$. Now we use \eqref{eq:error-term-graph-2} to obtain
    \[  
    \mathbb{E} \left[
    \begin{tikzpicture}[scale=0.7, baseline=(current bounding box.center)]
        \coordinate[vtx] (w1) at (0, 1);
        \node at (w1) [left] {$0$};
        \coordinate[vtx] (w2) at (6, 1);
        \node at (w2) [left] {$\tilde{S}_{n/2}$};
        \coordinate[vtx, blue] (v1) at (1, 2);
        \node at (v1) [above] {$\tilde{S}_{j_{\sigma(1)}}$};
        \coordinate[vtx, blue] (v2) at (2.5, 2);
        \node at (v2) [above] {$\tilde{S}_{j_{\sigma(2)}}$};
        \coordinate[vtx, blue] (v4) at (5, 2);
        \node at (v4) [above] {$\tilde{S}_{j_{\sigma(m)}}$};
        \coordinate[vtx, red] (u1) at (1, 0);
        \node at (u1) [below] {$\tilde{S}_{j_1}$};
        \coordinate[vtx, red] (u2) at (2.5, 0);
        \node at (u2) [below] {$\tilde{S}_{j_2}$};
        \coordinate[vtx, red] (u4) at (5, 0);
        \node at (u4) [below] {$\tilde{S}_{j_{m}}$};
        \draw[thick, blue] (v1) -- (v2) -- (v4) -- (w2);
        \draw[thick, red] (w1) -- (u1) -- (u2) -- (u4);
        \node at (3.75, 2.5) {$\dots$};
        \node at (3.75, -0.5) {$\dots$};
    \end{tikzpicture}
    \right] \le C^m |j_{\tau(1)}|_+^{-1/2} \prod_{k=2}^{m} |j_{\tau(k)} - j_{\tau(k-1)}|_+^{-1} |n/2 - j_{\tau(m)}|_+^{-1/2}
    \]
    
    The only additional reduction we need is when $\tilde{S}_j$ is a leaf, in which case we use $\mathbb{E}|S_i -x|_+^{-1} \le C|i|_+^{-1/2}$. Combining all above arguments gives \eqref{eq:aux-error-moment}.
    
    For general $l$, the proof is similar. The only change is that we have to repeat the above steps for every subinterval $[0,n]^{(l, p)}$ instead of just for $[0, n/2]$ and $[n/2, n]$. For instance, suppose you have $p_{2^{l-1}}$ points in $\mathcal{S}_n^{(l, 2^{l})}$. If you first condition on $\tilde{S}[0, n2^{-l}(2^l - 1)]$, then $p_{2^{l-1}}$ of the black edges will get replaced by a blue path going through the red endpoints. Then, if you integrate over $\mathcal{S}_n^{(l, 2^{l}-1)}$, the points you are integrating over will be exactly the points with red/blue (and no black) edges. Thus, once you integrate over these points, you are left with the red--black graph, except with $2p_{2^{l-1}}$ less vertices. As we repeat for all $p = 2^{l-1}, \dots, 1$, we have to consider all possible positions of $j_1 , \dots , j_m \in [0, n]^{(l, 1)}, \dots , [0, n]^{(l, 2^{l} - 1)}$. Accounting for all possible relative orders of $j_1 , \dots , j_{2m}$, this gives a combinatorial factor of
    \[  
    \sum_{m_1 + \dots + m_{2^{l-1}} = m} \frac{m!}{m_1! \dots m_{2^{l-1}}!} (m_1!)^2 \dots (m_{2^{l-1}}!)^2 = m! \sum_{m_1 + \dots + m_{2^{l-1}} = m}  (m_1!) \dots (m_{2^{l-1}}!).
    \]
    Following the proof of ~\cite{BenderRichmond1984}, this is bounded by $2^{l} (m!)^2 e^{2^l / m}$. In particular, when $2^l \ll m$, this converges to $2^l m!$. This can be explained by the fact that the sum is concentrated at the $2^l$ vertices of the simplex $m_1 + \dots + m_{2^l} = m$.
    
    \end{proof}
    
    \section{Preliminaries} \label{sec:main}
    
    \subsection{Decompositions} \label{subsec:decomposition}
    
    \begin{lemma}
    \label{lem:error-bound}
        Let $\epsilon(A, B) = 2 \chi(A, B) - \chi_{\mathcal{C}}(A, B)$. Then,      
        \[
        0 \le \epsilon(A, B) \le \sum_{x \in A} \sum_{y_1 , y_2 \in B} G_D(x-y_1)G_D(x-y_2) + \sum_{x_1, x_2 \in A} \sum_{y \in B} G_D(x_1-y)G_D(x_2 - y) + \cp (A \cap B).            
        \]
    \end{lemma}
    
    \begin{proof}
    By Proposition 1.6 of \cite{AsselahSchapiraSousi2019}, we know that
    \[  
    \chi_{\mathcal{C}}(A, B) = \bar{\chi}(A, B) + \bar{\chi}(B, A) - \bar{\epsilon}(A, B),
    \]
    where
    \[  
    \bar{\chi}(A, B) = \sum_{x \in A}\sum_{y \in B} \mathbb{P}^x(\tau_{A \cup B} = \infty)G_D(x-y)\mathbb{P}^y(\tau_B = \infty)
    \]
    and $0 \le \bar{\epsilon}(A, B) \le \cp(A \cap B)$. Further, since
    \begin{equation}
    \label{eq:last-pass}
    0 \le \mathbb{P}^x(\tau_A = \infty) - \mathbb{P}^x(\tau_{A \cup B} = \infty) \le \mathbb{P}^x(\tau_B < \infty) \le \sum_{y \in B} G_D(x-y),
    \end{equation}
    we see that
    \[  
    0 \le \chi(A, B) - \bar{\chi}(A, B) \le \sum_{x \in A}\sum_{y_1, y_2 \in B} G_D(x - y_1)G_D(x - y_2).
    \]
    Therefore, writing
    \[  
    0 \le 2\chi(A, B) - \chi_{\mathcal{C}}(A, B) = (\chi(A, B) - \bar{\chi}(A, B)) + (\chi(B, A) - \bar{\chi}(B, A)) + \bar{\epsilon}(A, B)
    \]
    completes the proof.
    \end{proof}
    
    \begin{lemma}
        \label{lem:aux}
        \begin{multline*}
        0 \le \chi(A, C) + \chi(B, C) -\chi(A \cup B , C) \\ \le C \left(\sum_{x \in A}\sum_{y \in B}\sum_{z \in C} G_D(x-y)G_D(y-z) + \sum_{x \in A}\sum_{y \in B}\sum_{z \in C} G_D(x-y)G_D(x-z) \right)
        \end{multline*}
    \end{lemma}
    
    \begin{proof}
        The proof is similar to Lemma~\ref{lem:error-bound}. Indeed, note that
        \begin{align*}
            \chi(A, C) &+ \chi(B, C) -\chi(A \cup B , C) \\
            &= \sum_{z \in C} \left( \sum_{x \in A} \{\mathbb{P}(\tau_A = \infty)  - \mathbb{P}(\tau_{A \cup B} = \infty) \} G_D(x - z)\right) \mathbb{P}(\tau_C = \infty) \\
            &\quad + \sum_{z \in C} \left( \sum_{y \in B}\{ \mathbb{P}(\tau_B = \infty)  - \mathbb{P}(\tau_{A \cup B} = \infty) \} G_D(x - z)\right) \mathbb{P}(\tau_C = \infty) \\
            &\quad + \sum_{x \in A \cap B} \sum_{z \in C} \mathbb{P} (\tau_{A \cup B = \infty}) G_D(x-z) \mathbb{P}(\tau_C = \infty).
        \end{align*}
        Apply \eqref{eq:last-pass} to the first two summands and note that since $G_D ( 0) > 0$,
        \[
        \sum_{x \in A \cap B} \sum_{z \in C} G_D (x - z) = \sum_{x \in A} \sum_{y \in B}\sum_{z \in C} \mathbf{1}_{\{ x = y \}} G_D (x - z) \le C \sum_{x \in A} \sum_{y \in B} \sum_{z \in C} G_D(x-y)G_D(x-z).
        \]
        
    \end{proof}
    
    \subsection{A Priori Estimates} \label{subsec:LDP-estimates}
    
    \begin{lemma}
    \label{lem:cross-term-moment}
    The following are true.
    \begin{enumerate}
        \item For any $m \ge 1$,
         \begin{equation}
         \label{eq:chi-moment}
        \mathbb{E} \left[ \chi (\mathcal{S}_n, \tilde{\mathcal{S}}_n)^m \right] \le C^m (m!)^{3/2} n^{m/2}.
        \end{equation}
        \item There exists some $\theta > 0$ such that
        \begin{equation} 
        \label{eq:chi-exp-moment}
        \sup_n \mathbb{E} \exp \left\{\theta  n^{-1/3} \chi(\mathcal{S}_n, \tilde{\mathcal{S}}_n)^{2/3} \right\} < \infty.
        \end{equation}
        \item For any $1 \ll b_n \ll n^{1/3}$,
        \begin{equation}  
        \label{eq:chi-C-LDP}
        \limsup_{n \to \infty} \frac{1}{b_n} \log \mathbb{P} \left(\chi(\mathcal{S}_n, \tilde{\mathcal{S}}_n) \ge \lambda \sqrt{nb_n^3} \right) \le -C \lambda^{2/3}.
        \end{equation}
    \end{enumerate}
    Furthermore, all of the above are also true when $\chi(\mathcal{S}_n, \tilde{\mathcal{S}}_n)$ is replaced with $\chi_{\mathcal{C}}(\mathcal{S}_n, \tilde{\mathcal{S}}_n)$.
    \end{lemma}
    
    \begin{proof}
        Since $\chi_{\mathcal{C}}(\mathcal{S}_n, \tilde{\mathcal{S}}_n) \le 2\chi(\mathcal{S}_n, \tilde{\mathcal{S}}_n)$, it suffices to show the above claims for $\chi(\mathcal{S}_n, \tilde{\mathcal{S}}_n)$. \eqref{eq:chi-moment} is a direct corollary of Proposition~\ref{prop:cross-term-moment}. Then, \eqref{eq:chi-exp-moment} follows from Taylor expansion and H\"{o}lder's inequality, since
        \begin{align*}
            \mathbb{E} \exp \left\{ \theta n^{-1/3} \chi(\mathcal{S}_n, \tilde{\mathcal{S}}_n)^{2/3} \right\} &= \sum_{m=0}^{\infty} \frac{\theta^m}{m!}n^{-m/3} \mathbb{E}[\chi(\mathcal{S}_n, \tilde{\mathcal{S}}_n)^{2m/3}] \\
            &\le \sum_{m=0}^{\infty} \frac{\theta^m}{m!}n^{-m/3} \mathbb{E}[\chi(\mathcal{S}_n, \tilde{\mathcal{S}}_n)^m]^{2/3} \\
            &\le \sum_{m=0}^{\infty} (C\theta)^m.
        \end{align*}
        Lastly, \eqref{eq:chi-C-LDP} follows from Markov's inequality:
        \begin{align*}
            \limsup_{n \to \infty} \frac{1}{b_n} \log \mathbb{P} \left(\chi(\mathcal{S}_n, \tilde{\mathcal{S}}_n) \ge \lambda \sqrt{nb_n^3} \right) &\le \limsup_{n \to \infty} \frac{1}{b_n} \log \left\{ \frac{\mathbb{E} \exp\left\{ \theta n^{-1/3} \chi(\mathcal{S}_n, \tilde{\mathcal{S}}_n)^{2/3} \right\} }{ \exp ( \theta \lambda^{2/3} b_n ) }\right\} \\
            &\le \limsup_{n \to \infty} \frac{C - \theta \lambda^{2/3} b_n}{b_n} \\
            &= -\theta \lambda^{2/3}.
        \end{align*}
    \end{proof}
    
    \begin{lemma}
    \label{lem:error-LDP}
    For $1 \ll b_n \ll n^{1/3}(\log n)^{-2}$ and $\lambda > 0$,
    \[  
        \lim_{n \to \infty} \frac{1}{b_n} \log \mathbb{P} \left( \epsilon(\mathcal{S}_n, \tilde{\mathcal{S}_n}) \ge \lambda \sqrt{n b_n^3} \right) = - \infty.
    \]
    \end{lemma}
    
    \begin{proof}
        This lemma follows easily from Lemma~\ref{lem:error-bound} and Proposition~\ref{prop:error-term-moment}. More explicitly, by Lemma~\ref{lem:error-bound}, it suffices to show the following.
        \begin{gather}
            \lim_{n \to \infty} \frac{1}{b_n} \log \mathbb{P} \left( \sum_{i = 1}^n \sum_{j_1 , j_2 =1}^n G_D(S_i - \tilde{S}_{j_1})G_D(S_i - \tilde{S}_{j_2}) \ge \lambda \sqrt{n b_n^3} \right) = - \infty \\
            \lim_{n \to \infty} \frac{1}{b_n} \log \mathbb{P} \left( \sum_{i_1, i_2}^n \sum_{j=1}^n G_D(S_{i_1} - \tilde{S}_j)G_D(S_{i_2} - \tilde{S}_j) \ge \lambda \sqrt{n b_n^3} \right) = - \infty \\
            \label{eq:int-bound} \lim_{n \to \infty} \frac{1}{b_n} \log \mathbb{P} \left( \cp(\mathcal{S}_n \cap \tilde{\mathcal{S}}_n ) \ge \lambda \sqrt{nb_n^3} \right) = - \infty.
        \end{gather}
        We can use a Chernoff bound to get the first two inequalities. More specifically, we use Markov's inequality on 
        \[
        \exp \left\{ \frac{\theta\sum_{i = 1}^n \sum_{j_1 , j_2 =1}^n G_D(S_i - \tilde{S}_{j_1})G_D(S_i - \tilde{S}_{j_2})}{(\log n)^3}\right\}^{1/3}.
        \]
        By Proposition~\ref{prop:error-term-moment}, this is guaranteed to converge for some small $\theta > 0$. Thus,
        \begin{align*}
            \lim_{n \to \infty} \frac{1}{b_n} \log &\mathbb{P} \left( \sum_{i = 1}^n \sum_{j_1 , j_2 =1}^n G_D(S_i - \tilde{S}_{j_1})G_D(S_i - \tilde{S}_{j_2}) \ge \lambda \sqrt{n b_n^3} \right) \\
            &\le \lim_{n \to \infty} \frac{1}{b_n} \log \frac{C}{\exp \{\lambda^{1/3}n^{1/6} b_n^{1/2}/(\log n) \}} \\
            &=\lim_{n \to \infty} \frac{\log C - \lambda^{1/3} n^{1/6}b_n^{1/2}}{b_n \log n} \\
            &= - \infty
        \end{align*}
        whenever $b_n \ll n^{1/3} (\log n)^{-2}$. Equation~\eqref{eq:int-bound} is an easy consequence of intersection probabilities (\cite{AsselahSchapira2023}, Theorem 1.1):
        \[
        \lim_{n \to \infty}  \frac{1}{b_n} \log \mathbb{P} \left( \cp(\mathcal{S}_n \cap \tilde{\mathcal{S}}_n ) \ge \lambda \sqrt{nb_n^3} \right) \le \lim_{n \to \infty}  \frac{1}{b_n} \log \mathbb{P} \left(|\mathcal{S}_n \cap \tilde{\mathcal{S}}_n| \ge \lambda \sqrt{nb_n^3} \right) = - \infty.
        \]
    \end{proof}
    
    The next two lemmas give a priori estimates on the deviations of $\cp(\mathcal{S}_n)$.
    
    \begin{lemma}
        \label{lem:cap-moment}
    \[  
    \sup_n \mathbb{E} \exp \left\{ \frac{\theta}{\sqrt[3]{n \log n}} | \cp(\mathcal{S}_n) - \mathbb{E} \cp (\mathcal{S}_n) |^{2/3} \right\} < \infty \quad  \text{for every } \theta > 0. \]    
    \end{lemma}
    
    \begin{proof}
    The proof is almost identical to that of \cite[Theorem~6.4.1]{Chen2010}, so we only provide a brief sketch here. The key steps are the following.
        \begin{enumerate}
            \item Show that $\mathbb{E} |\cp (\mathcal{S}_n) - \mathbb{E} \cp(\mathcal{S}_n)|^m = O_m( (n \log n)^{m/2})$.
            \item Show that $\mathbb{E} |\cp (\mathcal{S}_n) - \mathbb{E} \cp(\mathcal{S}_n)|^m \le C^m (m!)^{3/2} (n \log n)^{m/2}$. This implies Lemma~\ref{lem:cap-moment} for some (but not all) $\theta$.
            \item Extend to all $\theta > 0$. This is done by splitting the walk into $l$ parts and using the independence between parts.
        \end{enumerate}
        Steps $(1)$ and $(2)$ use induction on $m$. The base cases $m = 1, 2$ for $(1)$ are a corollary of the central limit theorem shown in \cite{AsselahSchapiraSousi2018}.
    \end{proof}
    
    \begin{lemma}
        \label{lem:cap-C-LDP}
        For any $\sqrt{\log n} \ll b_n \ll n^{1/3}$,
        \[
        \limsup_{n \to \infty} \frac{1}{b_n} \log \mathbb{P} \left\{ | \cp(\mathcal{S}_n) - \mathbb{E} \cp(\mathcal{S}_n)| \ge \lambda \sqrt{n b_n^3} \right\} \le - C \lambda^{2/3}.
        \]
    \end{lemma}
    
    \begin{proof}
        This proof is similar to that of \cite{Chen2010}, Lemma 8.4.1. By the decomposition \eqref{eq:decomposition-L} with $L = \log_2 (\log n)$ (i.e., $2^L = \log n$), it suffices to show the following.
        \begin{align}
        \label{eq:gamma-sum-LDP}
            \limsup_{n \to \infty} \frac{1}{b_n} \log \mathbb{P} \left\{ \left| \sum_{j=1}^{2^L} \Gamma_n^{(L, j)} \right| \ge \lambda \sqrt{nb_n^3} \right\} \le - C\lambda^{2/3} \\
        \label{eq:cap-logn-cross-LDP}
            \limsup_{n \to \infty} \frac{1}{b_n} \log \mathbb{P} \left\{ \left|\sum_{l=1}^{L} \Lambda_l^{\mathcal{C}} \right| \ge \lambda \sqrt{nb_n^3} \right\} \le - C\lambda^{2/3},
        \end{align}
        where $\Gamma_{n}^{(L, j)} = \cp(\mathcal{S}_n^{(L, j)}) - \mathbb{E} \cp(\mathcal{S}_n^{(L, j)})$.
        By Lemma~\ref{lem:cap-moment},
        \begin{equation}
            \label{eq:gamma-exp-moment}
            \sup_{n, j} \mathbb{E} \exp \left\{ \frac{\theta}{\sqrt[3]{n}} |\Gamma_{n}^{(L, j)}|^{2/3} \right\} < \infty.
        \end{equation}
        To show \eqref{eq:gamma-sum-LDP}, write $\Gamma_n^{(L, j)}$ as a sum of the following two terms.
        \begin{gather}
        \label{eq:gamma-truncate}
        \hat{\Gamma}_n^{(L, j)} = \Gamma_n^{(L, j)} \mathbf 1 \left\{ | \Gamma_n^{(L, j)}| \le \lambda\sqrt{nb_n^3} \right\} - \mathbb{E} \Gamma_n^{(L, j)} \mathbf 1 \left\{ | \Gamma_n^{(L, j)}| \le \lambda\sqrt{nb_n^3} \right\} \\
        \tilde{\Gamma}_n^{(L, j)} = \Gamma_n^{(L, j)} \mathbf 1 \left\{ | \Gamma_n^{(L, j)}| > \lambda\sqrt{nb_n^3} \right\} - \mathbb{E} \Gamma_n^{(L, j)} \mathbf 1 \left\{ | \Gamma_n^{(L, j)}| > \lambda\sqrt{nb_n^3} \right\}.
        \end{gather}
        Then,
        \begin{align*}
            \mathbb{P} \left\{ \left| \sum_{j=1}^{2^L} \tilde{\Gamma}_n^{(L, j)} \right| \ge \lambda \sqrt{nb_n^3} \right\} &\le \sum_{j=1}^{2^L} \mathbb{P} \left\{ \left| \Gamma_n^{(L, j)} \right| \ge \lambda \sqrt{nb_n^3} \right\}  \\
            &\le 2^L \sup_{n, j} \mathbb{E} \exp \left\{ \frac{\theta}{\sqrt[3]{n}} |\Gamma_{n}^{(L, j)}|^{2/3} \right\} \exp(-\theta \lambda^{2/3} b_n)
        \end{align*}
        by Markov's inequality. Using \eqref{eq:gamma-exp-moment} and taking logarithms, we see that
        \begin{equation}
            \label{eq:gamma-tilde-LDP}
        \limsup_{n \to \infty} \frac{1}{b_n} \log \mathbb{P} \left\{ \left|\sum_{j=1}^{2^L} \tilde{\Gamma}_n^{(L, j)} \right| \ge \lambda \sqrt{nb_n^3} \right\} \le - C\lambda^{2/3}.
        \end{equation}
        Further, since $|\hat{\Gamma}_n^{(L, j)}| \le \lambda \sqrt{nb_n^3}$,
        \begin{align*}  
        \mathbb{E} \exp \left\{ \frac{\theta}{\lambda^{1/3} \sqrt{nb_n}}\left|\sum_{j=1}^{2^L} \hat{\Gamma}_n^{(L, j)}\right| \right\} &= \prod_{j=1}^{2^L} \mathbb{E} \exp\left\{ \frac{\theta}{\lambda^{1/3} \sqrt{nb_n}}\left| \hat{\Gamma}_n^{(L, j)}\right| \right\} \\
        &\le \prod_{j=1}^{2^L} \mathbb{E} \exp
        \left\{ \frac{\theta}{\lambda^{1/3} \sqrt{nb_n}}\left| \hat{\Gamma}_n^{(L, j)}\right|^{2/3} \left(\lambda \sqrt{n b_n^3}\right)^{1/3} \right\} \\
        &= \prod_{j=1}^{2^L} \mathbb{E} \exp \left\{ \frac{\theta}{\sqrt[3]{n}} |\hat{\Gamma}_n^{(L, j)}|^{2/3}\right\} \\
        &\le \left(  \sup_{n, j} \mathbb{E} \exp \left\{ \frac{\theta}{\sqrt[3]{n}} |\Gamma_{n}^{(L, j)}|^{2/3} \right\} \right)^{2^L}.
        \end{align*}
        Again, we can use Markov's inequality and \eqref{eq:gamma-exp-moment} to deduce 
        \begin{equation*}
            \limsup_{n \to \infty} \frac{1}{b_n} \log \mathbb{P} \left\{ \left|\sum_{j=1}^{2^L} \hat{\Gamma}_n^{(L, j)} \right| \ge \lambda \sqrt{nb_n^3} \right\} \le - C\lambda^{2/3}.
        \end{equation*}
        Combined with \eqref{eq:gamma-tilde-LDP}, this shows \eqref{eq:gamma-sum-LDP}.
    
        The proof for \eqref{eq:cap-logn-cross-LDP} is similar. Recalling that $\Lambda_l^{\mathcal{C}} = \sum_{j=1}^{2^{l-1}} \chi_{\mathcal{C}}(\mathcal{S}_n^{(l, 2j-1)}, \mathcal{S}_n^{(l, 2j)})$,
        \[  
        \mathbb{P} \left\{ \left|\sum_{l=1}^{L} \Lambda_l^{\mathcal{C}} \right| \ge \lambda \sqrt{nb_n^3} \right\} \le \sum_{l=1}^L \mathbb{P} \left\{ \sum_{j=1}^{2^l} \left| \chi_{\mathcal{C}}(\mathcal{S}_n^{(l, 2j-1)}, \mathcal{S}_n^{(l, 2j)})\right| \ge(\sqrt{2}-1) \lambda \sqrt{n 2^{-l} b_n^3}\right\}.
        \]
        Now truncate $\chi_{\mathcal{C}}(\mathcal{S}_n^{(l, j)}, \tilde{\mathcal{S}}_n^{(l, j)})$ in a similar fashion as \eqref{eq:gamma-truncate}, except with threshold $\lambda \sqrt{n2^{-l} b_n^3}$. Then, the same argument as above can be repeated (with \eqref{eq:chi-exp-moment} taking the place of \eqref{eq:gamma-exp-moment}) to deduce \eqref{eq:cap-logn-cross-LDP}.
    \end{proof}
    \subsection{The Gaussian Regime} \label{subsec:gaussian}
    
    In this section, we prove Theorem~\ref{thm:gaussian}. This proof is very similar to that of Lemma~\ref{lem:cap-C-LDP}, except with a different choice of $L$.
    \begin{proof}[Proof of Theorem~\ref{thm:gaussian}]
        Take $L = \gamma^{-1} \log_2 b_n$, where $0 <\gamma < 1/2$. We first show that for any $M > 0$, there exists some $\lambda$ such that
        \begin{equation}
        \label{eq:gaussian-cross-bound}
        \lim_{n \to \infty} \frac{1}{b_n} \log \mathbb{P}\left\{ \left| \sum_{l = 1}^L \Lambda_l^{\mathcal{C}} \right| \ge \lambda \sqrt{n b_n^3} \right\} \le - M.
        \end{equation}
        To this end, write
        \[  
        \mathbb{P}\left\{ \left| \sum_{l = 1}^L \Lambda_l^{\mathcal{C}} \right| \ge \lambda \sqrt{n b_n^3} \right\} \le \sum_{l=1}^L \mathbb{P}\left\{ \left| \sum_{j=1}^{2^{l-1}} \chi_{\mathcal{C}}(\mathcal{S}_n^{(l, 2j-1)}, \mathcal{S}_n^{(l, 2j)}) \right| \ge \lambda 2^{\gamma} (1 - 2^{-\gamma}) 2^{-\gamma l}\sqrt{n b_n^3} \right\}.
        \]
        Now let $C > 0$ be some constant and truncate $\bar{\chi}_n^{(l, j)} = \chi_{\mathcal{C}}(\mathcal{S}_n^{(l, 2j-1)}, \mathcal{S}_n^{(l, 2j)}) - \mathbb{E} \chi_{\mathcal{C}}(\mathcal{S}_n^{(l, 2j-1)}, \mathcal{S}_n^{(l, 2j)})$ into the sum of the following.
        \begin{gather*}
            \hat{\chi}_n^{(l, j)} = \bar{\chi}_n^{(l, j)} \mathbf{1} \left\{|\bar{\chi}_n^{(l, j)}| \le C \sqrt{n 2^{-l} b_n^3} \right\} -  \mathbb{E} \bar{\chi}_n^{(l, j)} \mathbf{1} \left\{|\bar{\chi}_n^{(l, j)}| \le C \sqrt{n 2^{-l} b_n^3} \right\} \\
            \tilde{\chi}_n^{(l, j)} = \bar{\chi}_n^{(l, j)} \mathbf{1} \left\{|\bar{\chi}_n^{(l, j)}| > C \sqrt{n 2^{-l} b_n^3} \right\} -  \mathbb{E} \bar{\chi}_n^{(l, j)} \mathbf{1} \left\{|\bar{\chi}_n^{(l, j)}| > C \sqrt{n 2^{-l} b_n^3} \right\}
        \end{gather*}
        By choosing $C$ to be sufficiently large, we can argue as in the proof of Lemma~\ref{lem:cap-C-LDP} to show \eqref{eq:gaussian-cross-bound} as long as $0 < \gamma < 1/2$. Since $\sqrt{n b_n^3} \ll \sqrt{n b_n \log n}$ when $b_n \ll \sqrt{\log n}$, this implies
        \[  
        \lim_{n \to \infty} \frac{1}{b_n} \log \mathbb{P}\left\{ \left| \sum_{l = 1}^L \Lambda_l^{\mathcal{C}} \right| \ge \epsilon \sqrt{n b_n \log n} \right\} = - \infty.
        \]
        for any $\epsilon > 0$. Now we truncate $\Gamma_n^{(L, j)}$ (as defined in the proof of Lemma~\ref{lem:cap-C-LDP}) with the threshold
        \begin{gather*}
        \hat{\Gamma}_n^{(L, j)} = \Gamma_n^{(L, j)} \mathbf 1 \left\{ | \Gamma_n^{(L, j)}| \le \sqrt{2^{-L} nb_n^3 \log n} \right\} - \mathbb{E} \Gamma_n^{(L, j)} \mathbf 1 \left\{ | \Gamma_n^{(L, j)}| \le \sqrt{2^{-L} nb_n^3 \log n} \right\}, \\
        \tilde{\Gamma}_n^{(L, j)} = \Gamma_n^{(L, j)} \mathbf 1 \left\{ | \Gamma_n^{(L, j)}| > \sqrt{2^{-L} nb_n^3 \log n} \right\} - \mathbb{E} \Gamma_n^{(L, j)} \mathbf 1 \left\{ | \Gamma_n^{(L, j)}| > \sqrt{2^{-L} nb_n^3 \log n} \right\} .
        \end{gather*}
        Similarly to before, we can deduce that for any $M > 0$, there exists some $\lambda > 0$ such that
        \[  
        \lim_{n \to \infty} \frac{1}{b_n} \log \mathbb{P} \left( \sum_{j=1}^{2^L} \tilde{\Gamma}_n^{(L, j)}\ge \lambda \sqrt{2^{-L} nb_n^3 
        \log n} \right) = - M.
        \]
        If we take $0 < \gamma < 1/2$, then $\sqrt{2^{-L} nb_n^3 
        \log n} \ll \sqrt{n b_n \log n}$ and so
        \[  
        \lim_{n \to \infty} \frac{1}{b_n} \log \mathbb{P} \left( \sum_{j=1}^{2^L} \tilde{\Gamma}_n^{(L, j)}\ge \epsilon \sqrt{n b_n \log n } \right) = - \infty
        \] for any $\epsilon > 0$. Therefore, it only remains to  show
        \[
            \frac{1}{b_n} \lim_{n \to \infty} \log \mathbb{E} \exp \left\{ \theta \sqrt{\frac{b_n}{n \log n} }\sum_{j=1}^{2^L} \hat{\Gamma}_n^{(L, j)} \right\} = \frac{\theta^2 \sigma^2}{2}
        \]
        for every real number $\theta$, since this implies Theorem~\ref{thm:gaussian} by the G\"{a}rtner--Ellis theorem. By independence and Taylor expansion,
       \begin{align*}
           \mathbb{E} \exp \left\{ \theta \sqrt{\frac{b_n}{n \log n} }\sum_{j=1}^{2^L} \hat{\Gamma}_n^{(L, j)} \right\} &= \prod_{j=1}^{2^L} \mathbb{E} \exp \left\{ \theta \sqrt{\frac{b_n}{n \log n} }\hat{\Gamma}_n^{(L, j)} \right\} \\
           &=   \exp \left\{(1 + o(1)) 2^L \frac{\theta^2 b_n}{2n \log n} \mathbb{E} |\hat{\Gamma}_n^{(L, 1)}|^2 \right\}.
       \end{align*}
       Since $\mathbb{E}|\hat{\Gamma}_n^{(L, 1)}|^2 \sim \mathbb{E} |\Gamma_n^{(L, 1)}|^2 \sim \sigma^2 2^{-L} n \log n$, the proof is complete.
    \end{proof}
    
    \section{Deviations of the cross term} \label{sec:cross-term}
    
    \subsection{Preliminaries}
    In this section, we prove Theorem~\ref{thm:cross-term-LDP}. Our main tool is the following version of the G\"{a}rtner--Ellis theorem.
    \begin{theorem}[\cite{Chen2010}, Theorem 1.2.6]
    \label{thm:gartner-ellis}
        Assume that for each $\theta > 0$, the limit
        \[  
        \Psi(\theta) := \lim_{n \to \infty} \frac{1}{b_n} \log \sum_{m=0}^{\infty} \frac{\theta^m}{m!} b_n^m \left( \mathbb{E} Y_n^m \right)^{1/p}
        \]
        exists as an extended real number. Assume that the function $\Psi(|\theta|)$ is essentially smooth. For each $\lambda > 0$,
        \[  
        \lim_{n \to \infty} \frac{1}{b_n} \log \mathbb{P} (Y_n \ge \lambda) = -I_{\Psi}(\lambda),
        \]
        where the rate function $I_{\Psi}$ is defined by
        \[
        I_{\Psi}(\lambda) = p \sup_{\theta > 0} \{ \theta \lambda^{1/p} - \Psi(\theta) \}. \quad \lambda > 0.
        \]
    \end{theorem}
    
    This reduces Theorem~\ref{thm:cross-term-LDP} to showing
    \begin{equation}
    \label{eq:cross-exp-moment}
    \lim_{n \to \infty} \frac{1}{b_n} \log \sum_{m = 0}^{\infty} \frac{\theta^m}{m!} \left( \frac{b_n}{n} \right)^{m/4} (\mathbb{E}  \chi(\mathcal{S}_n, \tilde{\mathcal{S}}_n)^m )^{1/2} = \frac{27 \cdot 5^5}{32} \theta^4 \tilde{\gamma}_S^4 \tilde{\kappa}^8(5,2).
    \end{equation}
    
    Indeed, we can check that $\theta \lambda^{1/2} - C \theta^4$ obtains its minimum at $\theta = (\lambda^{1/2} / 4C)^{1/3}$, and so
    \begin{align*}
        \sup_{\theta > 0} \left\{ \theta \lambda^{1/2} - C \theta^4 \right\} &= \left(\frac{\lambda^{1/2}}{4C} \right)^{1/3} \lambda^{1/2} - C \left(\frac{\lambda^{1/2}}{4C} \right)^{4/3} \\
        &= 3 \cdot 4^{-4/3} \lambda^{2/3} C^{-1/3}.
    \end{align*}
    Plugging in the value from \eqref{eq:cross-exp-moment} for $C$ gives
    \[  
    2I_5(\lambda) = 5^{-5/3} \tilde{\gamma}_S^{-4/3} \tilde{\kappa}(5,2)^{-8/3} \lambda^{2/3}.
    \]
    
    In order to show \eqref{eq:cross-exp-moment}, we will first show that $\chi(\mathcal{S}_n, \tilde{\mathcal{S}}_n)$ converges weakly to the limit
        \[  
        \frac{1}{\sqrt{n}}\chi(\mathcal{S}_n , \tilde{\mathcal{S}}_n) \xrightarrow{d} 5^{5/2} \tilde{\gamma}_S^2 \int_0^1 \int_0^1 G(B_t - \tilde{B}_s) \mathrm{d}t \mathrm{d}s
        \]
    and then show that each of the moments in \eqref{eq:cross-exp-moment} may be replaced with those of its continuous limit. Then \eqref{eq:cross-exp-moment} follows from the results of Bass, Chen, and Rosen.
    \begin{theorem}[\cite{BassChenRosen2009}, Theorem 1.2 restated]
    \label{thm:brownian-LDP}
        \[  
        \lim_{\lambda \to \infty} \frac{1}{\lambda^{2/3}} \log \mathbb{P} \left\{ \int_0^1 \int_0^1 G(B_t - \tilde{B}_s) \mathrm{d} t \mathrm{d} s \ge \lambda \right\} =  \tilde{\kappa}(5,2)^{-8/3}.
        \]
    \end{theorem}
    This is also reminiscent of a similar theorem for mutual intersection local times of Brownian motions in $\mathbb{Z}^3$ (\cite{Chen2010}, Theorem 3.3.2; see reference for relevant definitions):
        \[  
        \lim_{ \lambda \to \infty} \lambda^{-2/3} \log \mathbb{P} \left\{ \alpha([0,1]^2) \ge \lambda \right\} = \kappa(3,2)^{-8/3}.
        \]
    In its original form, Theorem~\ref{thm:brownian-LDP} was written as an optimization problem. To get to the form stated above, we need the following lemma. This will be used once more in the proof of the lower bound.
    \begin{lemma}
        \label{lem:gagliardo-nirenberg-constant}
        \begin{equation}
        \label{eq:optimization}
        \sup_{g \in \mathcal{F}} \left\{ \theta \left[ \int_{\mathbb{R}^5} g^2(x) G(x - \tilde{x}) g^2(\tilde{x}) \mathrm{d}x \mathrm{d} \tilde{x} \right]^{1/2} - \frac{1}{10} \int_{\mathbb{R}^5} |\nabla g(x)|^2 \mathrm{d} x \right\} = \frac{27 \cdot 5^3}{32} \theta^4 \tilde{\kappa}(5,2)^8
        \end{equation}
        where $\mathcal{F} = \{g : \| g \|_2 =1, \int_{\mathbb{R}^5}|\nabla g(x)|^2 \text{d}x< \infty \}$.
    \end{lemma}
    \begin{proof}
         Recall that $\tilde{\kappa}(5,2)$ is the optimal constant such that the following generalized Gagliardo--Nirenberg inequality holds.
        \begin{equation} \label{eq:GenGagNir}
    \left[\int_{\mathbb{R}^5} g^2(x)G(x-\tilde{x}) g^2(\tilde{x}) \text{d}x \text{d}\tilde{x} \right]^{1/4} \le \tilde{\kappa}(5,2) \left(\int_{\mathbb{R}^5} g^2(x) \text{d}x\right)^{1/8} \left(\int_{\mathbb{R}^5} |\nabla g(x)|^2 \text{d}x \right)^{3/8}
    \end{equation}
    If $g_{opt}$ is a function that attains the optimal value, then we see that the scaling transformation $g_{opt}(x) \to t^{5/2} g_{opt}(tx)$ will also attain the optimal value in the Gagliardo--Nirenberg value with a modified value for $\int_{\mathbb{R}^5} |\nabla g(x)|^2 \text{d}x$ without changing $\int_{\mathbb{R}^5} g^2(x) \text{d}x.$
    
     If we then substitute the inequality in equation \eqref{eq:GenGagNir} into the left-hand side of \eqref{eq:optimization} and bound it by,
     \begin{equation}
    \theta \tilde{\kappa}^2(5,2) \left[ \int_{\mathbb{R}^5}|\nabla g(x)|^2 \text{d}x \right]^{3/4} - \frac{1}{10} \int_{\mathbb{R}^5}|\nabla g(x)|^2 \text{d}x,
     \end{equation}
     where the supremum is taken over all functions $g$ whose $L^2$ norm is 1 and the square integral of the gradient is finite. Now, recall that by our scaling transformation, we can set $\int_{\mathbb{R}^5}|\nabla g(x)|^2 \text{d}x$ to attain any value we want while still satisfying the Gagliardo--Nirenberg inequality. We optimize over the value of $\int_{\mathbb{R}^5}|\nabla g(x)|^2 \text{d}x$ and see that the maximum is given by $\frac{27\cdot 5^3}{32} \theta^4 \tilde{\kappa}(5,2)^8$.
    \end{proof}
    
    In the next two sections, we will show upper and lower bounds of \eqref{eq:cross-exp-moment}, thus proving Theorem~\ref{thm:cross-term-LDP}. The main obstacle in handling $\chi(\mathcal{S}_n, \tilde{\mathcal{S}}_n)$ are the probability terms $\mathbb{P}(\tau_{\mathcal{S}_n} = \infty)$ and $\mathbb{P}(\tau_{\tilde{\mathcal{S}}_n} = \infty)$ which correlate the entire sample path. This makes it difficult to compute high moments or to use tools such as Markov operators. To combat this, we consider a localized modification. If we divide the walk $\mathcal{S}_n$ into $b_n$ subparts, $S[0, n/b_n], S[n/b_n, 2n/b_n], \dots S[n - n/b_n, n]$, then we have the upper bound,
    \begin{equation} \label{eq:defchibn}
    \chi_{b_n}(\mathcal{S}_n, \tilde{\mathcal{S}}_n) := \sum_{i, j = 1}^{b_n} \chi ( S[(i-1)n/b_n, in/b_n], \tilde{S}[(j-1)n/b_n, jn/b_n]) \ge \chi(\mathcal{S}_n, \tilde{\mathcal{S}}_n).
    \end{equation}
     Further, by the moment bounds of Lemma~\ref{lem:aux-error-moment}, we can deduce that their difference is negligible in the moderate deviation regime.
    
    \begin{prop} \label{prop:chiischibn}
    Let $1 \ll b_n \ll n^{1/3}(\log n)^{-8/3}$ and consider $\chi(\mathcal{S}_n, \tilde{\mathcal{S}}_n)$ as well as $\chi_{b_n}$. We have that,
    \begin{equation}
    \lim_{n \to \infty}\frac{1}{b_n} \log \mathbb{P}(\chi(\mathcal{S}_n, \tilde{\mathcal{S}}_n) \ge \lambda \sqrt{n b_n^3}) =\lim_{n \to \infty} \frac{1}{b_n} \log \mathbb{P}(\chi_{b_n} \ge \lambda \sqrt{n b_n^3})\end{equation}
    
    \end{prop}
    
    \begin{proof}
    
    For convenience, we shall assume $b_n = 2^L$ for some integer $L$. The general case can be done with a similar method. We first show
    \begin{equation}
    \label{eq:aux-chi-decomposition}
    0 \le \chi_{2^L} - \chi(\mathcal{S}_n, \tilde{\mathcal{S}}_n) \le C \sum_{l=1}^L (X_l + Y_l + Z_l + W_l).
    \end{equation}
    Recalling that $\mathcal{S}_n = \mathcal{S}_n^{(1, 1)} \cup \mathcal{S}_n^{(1, 2)}$ and $\tilde{\mathcal{S}}_n = \tilde{\mathcal{S}}_n^{(1, 1)} \cup \tilde{\mathcal{S}}_n^{(1, 2)}$, we can write
    \begin{align*}
        \chi_2 - \chi(\mathcal{S}_n, \tilde{\mathcal{S}}_n) &= \{\chi(\mathcal{S}_n^{(1, 1)}, \tilde{\mathcal{S}}_n^{(1, 1)}) + \chi(\mathcal{S}_n^{(1, 2)}, \tilde{\mathcal{S}}_n^{(1, 1)}) - \chi(\mathcal{S}_n, \tilde{\mathcal{S}}_n^{(1, 1)}) \} \\ 
        &\quad + \{\chi(\mathcal{S}_n^{(1, 1)}, \tilde{\mathcal{S}}_n^{(1, 2)}) + \chi(\mathcal{S}_n^{(1, 1)}, \tilde{\mathcal{S}}_n^{(1, 2)}) - \chi(\mathcal{S}_n, \tilde{\mathcal{S}}_n^{(1, 2)}) \} \\
        &\quad + \{\chi(\mathcal{S}_n, \tilde{\mathcal{S}}_n^{(1, 1)}) + \chi(\mathcal{S}_n, \tilde{\mathcal{S}}_n^{(1, 2)}) - \chi(\mathcal{S}_n, \tilde{\mathcal{S}}_n) \}.
    \end{align*}
    By Lemma~\ref{lem:aux}, the sum of the first two terms are bounded by $C(X_1 + Y_1)$, and the second term is bounded by $C(Z_1 + W_1)$. Similarly,
    \[  
    0 \le \chi_{2^l} - \chi_{2^{l-1}} \le C(X_l + Y_l + Z_l + W_l).
    \]
    Summing over $l = 1, \dots , L$ gives \eqref{eq:aux-chi-decomposition}. Further,
    \begin{align*}
    &\frac{1}{b_n} \log \mathbb{P}\left(\sum_{l= 1}^{L} X_l\ge \epsilon \sqrt{nb_n^3}\right) \le \frac{1}{b_n} \log L + \frac{1}{b_n } \log \mathbb{P}\left(X_l \ge \epsilon L^{-1} \sqrt{n b_n^3}\right) \\
    &\le \frac{\log L}{b_n}+ \frac{1}{b_n} \log \frac{\mathbb{E}[X_l^m]}{(\epsilon L^{-1} \sqrt{nb_n^3})^m} \\
    &\le \frac{\log L}{b_n} + \frac{1}{b_n} \log \frac{C^m 2^l \exp[2^l/m] (m!)^3 (\log n)^{3m}  (\log b_n)^m }{\epsilon^m n^{m/2} b_n^{3m/2}}
    \end{align*}
    by a union bound and Lemma~\ref{lem:aux-error-moment}. Recall that $b_n = 2^L$ and choose $m = n^{1/12}b_n^{3/4}(\log n)^{-2/3}$. This ensures
    \[  
    \frac{2^l}{m} \le \frac{b_n}{m} = \frac{b_n^{1/4}}{n^{1/12}(\log n)^{-2/3}} = \left( \frac{b_n}{n^{1/3} (\log n)^{-8/3}} \right)^{1/4} \ll 1
    \]
    and
    \[
    \frac{C m^3 (\log n)^4}{n^{1/2} b_n^{3/2}} = \frac{Cb_n^{3/4}(\log n)^{2}}{n^{1/4}} = C \left( \frac{b_n}{n^{1/3}(\log n)^{-8/3}}\right)^{3/4} \ll 1.
    \]
    Therefore,
    \begin{align*}
        \frac{1}{b_n} \log \frac{C^m 2^l \exp[2^l/m] (m!)^3 (\log n)^{3m}(\log b_n)^m}{\epsilon^m n^{m/2} b_n^{3m/2}} &\le \frac{1}{b_n} \log b_n \left(\frac{C m^{3} (\log n)^{4}}{n^{1/2} b_n^{3/2}} \right)^m \\
        &= \frac{m}{b_n} \left\{ \log b_n + \log \left(\frac{C m^{3} (\log n)^{4}}{n^{1/2} b_n^{3/2}} \right) \right\} \\
        &= - \infty.
    \end{align*}
    In other words,
    \[
    \lim_{n \to \infty} \frac{1}{b_n} \log \mathbb{P} \left( |\chi_{b_n} - \chi(\mathcal{S}_n, \tilde{\mathcal{S}}_n )| \ge \epsilon \sqrt{n b_n^3} \right) = - \infty
    \]
    so we are done.
    \end{proof}
    
    The goal of the next few subsections is to derive appropriate upper and lower bounds for the moments of $\chi_{b_n}$. This will allow us to derive a large deviation principle for $\chi_{b_n}$, and by Proposition \ref{prop:chiischibn}, we would derive the appropriate large deviation principle for $\chi(\mathcal{S}_n, \tilde{\mathcal{S}}_n)$.
    
    \subsection{Moment Upper Bounds for $\chi(\mathcal{S}_n, \tilde{\mathcal{S}}_n)$}
    
    In what proceeds, we will derive upper bounds for the moments of $\chi_{b_n}$. In order to do this, we first show that $\chi_{b_n}$ converges in distribution to a continuous limit. Then we prove that the moments of $\chi_{b_n}$ in \eqref{eq:cross-exp-moment} can be replaced with its limit to obtain an upper bound.
    
    \begin{lemma}
    \label{lem:weak-conv}
        As $n \to \infty$,
        \[  
        \frac{1}{\sqrt{n}}\chi(\mathcal{S}_n , \tilde{\mathcal{S}}_n) \xrightarrow{d} 5^{5/2} \tilde{\gamma}_S^2 \int_0^1 \int_0^1 G(B_t - \tilde{B}_s) dt ds.
        \]
    \end{lemma}
    
    \begin{proof}
        The proof is almost identical to Proposition 6.1 of \cite{AsselahSchapiraSousi2019}, so we only present a sketch here. By Carleman's condition and the moment estimates of Lemma~\ref{lem:cross-term-moment}, it suffices to show that all finite moments converge. Furthermore, by Proposition~\ref{prop:chiischibn}, we may replace $\chi(\mathcal{S}_n, \tilde{\mathcal{S}}_n)$ with $\chi_{2^L}$ where $1 \ll 2^L \ll n^{1/4}$. We define the event
        \[  
        A_n^{(L, j)}(\mathcal{S}_n, i) = \{ S_i \notin S[2^{-L}nj, i-1]\} \cap \{ (S_i + \mathcal{S}') \cap \mathcal{S}_n^{(L, j)} = \emptyset \}, \quad i \in [n]^{(L, j)}
        \]    
        From the viewpoint of $S_i$, $\mathcal{S}_n^{(L, j)}$ is a two-sided random walk. As such, it is reasonable that $\lim_{n \to \infty} \mathbb{P}(A_n^{(L, j)}(\mathcal{S}_n, i)) = \tilde{\gamma}_S$ for most values of $i$ and $j$ (that is to say, outside of a negligible proportion). With some more work, we can also show that this convergence holds even after conditioning on the position of $S_i$. That is, under some conditions on $i, j$, and $x$,
        \[  
        \lim_{n \to \infty} \mathbb{P}(A_n^{(L, j)}(\mathcal{S}_n, i) | S_i = x) = \tilde{\gamma}_S.
        \]
        Using this definition, we can rewrite $\chi(\mathcal{S}_n, \tilde{\mathcal{S}}_n)$ as
        \begin{align*}
        \chi_{2^L} &= \sum_{i, j = 1}^{2^L} \sum_{x \in \mathcal{S}_n^{(L, j)}} \sum_{y \in \tilde{S}_n^{(L, j)}} \mathbb{P}^x(\tau_{\mathcal{S}_n} = \infty) G_D(x - y) \mathbb{P}^y (\tau_{\tilde{S}_n} = \infty) \\
        &= \sum_{i, j = 1}^{2^L} \sum_{i_1 \in [n]^{(L, i_1)}}\sum_{j_1 \in [n]^{(L, j)}} \mathbb{P}(A_n^{(L, i)}(\mathcal{S}_n, i_1)) G_D(S_{i_1} - \tilde{S}_{j_1}) \mathbb{P}(A_n^{(L, j)}(\tilde{\mathcal{S}}_n, j_1)).
        \end{align*}
        Taking expectations, we get
        \[
            \mathbb{E} \chi_{2^L} =  \sum_{i, j = 1}^{2^L} \sum_{i' \in [n]^{(L, i)}}\sum_{j' \in [n]^{(L, j)}} \mathbb{E}\left[ \mathbb{P}(A_n^{(L, i)}(\mathcal{S}_n, i')|S_{i'}) G_D(S_{i'} - \tilde{S}_{j'}) \mathbb{P}(A_n^{(L, j)}(\tilde{\mathcal{S}}_n, j')|\tilde{S}_{j'}) \right] \\.
        \]
        By the arguments given above, we may replace the probability terms with $\tilde{\gamma}_S$ to get
        \begin{align*} 
        \mathbb{E} \left[\frac{1}{\sqrt{n}} \chi_{2^L} \right]  &\sim \frac{1}{\sqrt{n}}\sum_{i, j = 1}^{2^L} \sum_{i' \in [n]^{(L, i)}}\sum_{j' \in [n]^{(L, j)}}  \mathbb{E} [\tilde{\gamma}_S^2 G_D(S_{i'} - \tilde{S}_{j'})] \\
        &= \tilde{\gamma}_S^2 \mathbb{E} \left[ \frac{1}{\sqrt{n}} \sum_{i, j=1}^n G_D(S_i - \tilde{S}_j)\right].
        \end{align*}
        From here, convergence to the Brownian integral is a easy consequence of Donsker's invariance principle and the fact that $G_D(x) = 5G(x) + O(|x|^{-5})$. However, there is one problem in that $G(x)$ is singular at $x = 0$. This can be resolved by a smooth approximation, following methods similar to either Lemma~\ref{lem:smooth-error} of this paper or \cite{AsselahSchapiraSousi2019}, Lemma 6.8.
        
        The proof for higher moments is similar. First write
        \[  
        |\chi_{2^L}|^m= \sum \prod_{k=1}^m \mathbb{P}(A_n^{(L, i_k)} (\mathcal{S}_n, i_k')) G_D(S_{i_k'} - \tilde{S}_{j_k'}) \mathbb{P}(A_n^{(L, i_k)} (\tilde{\mathcal{S}}_n, j_k')),
        \]
        where the summation is over all $1 \le i_k, j_k \le 2^L$ and $i_k' \in [n]^{(L, i_k)}, j_k' \in [n]^{(L, j_k)}$ ($1 \le k \le m$). By the moment bounds of Lemma~\ref{lem:cross-term-moment}, we may only consider the cases where $i_k, j_k$ are sufficiently far apart. In these cases, we may similarly replace $\mathbb{P}(A_n^{(L, i_k)} (\mathcal{S}_n, i_k')), \mathbb{P}(A_n^{(L, j_k)} (\tilde{\mathcal{S}}_n, j_k'))$ with $\tilde{\gamma}_S$ to get
        \[  
        \mathbb{E} \left| \frac{1}{\sqrt{n}} \chi_{2^L} \right|^m \sim \tilde{\gamma}_S^{2m} \mathbb{E} \left[ \left| \frac{1}{\sqrt{n}} \sum_{i, j=1}^n G_D(S_i - \tilde{S}_j)\right|^m \right].
        \]
    \end{proof}
    
    \begin{prop}\label{prop:chibnuprbnd}
    Let $b_n = o(n^{1/3} (\log n)^{-8/3})$ and recall the definition of $\chi_{b_n}$ from equation \eqref{eq:defchibn}. We have the moment upper bound,
    \begin{equation}
    \begin{aligned}
    &\limsup_{n \to \infty}\frac{1}{b_n} \log \sum_{m=0}^{\infty} \frac{ \theta^m \left(\frac{b_n}{n}\right)^{\frac{m}{4}}}{m!} \left\{ \mathbb{E}[ \chi_{b_n}^m]  \right\}^{1/2} \le 
    \\& \lim_{t \to \infty} \frac{1}{t}\log \sum_{m=0}^{\infty} \frac{\theta^m}{m!} t^{\frac{m}{4}}\mathbb{E}\left[\left(5^{5/2} \tilde{\gamma}_S^2 \int_0^1 \int_0^1 G(B_t - \tilde{B}_s) dt ds\right)^m \right]^{1/2}.
    \end{aligned}
    \end{equation}
    
    \end{prop}
    
    \begin{proof}
    
    Let $t$ be an integer. We will introduce appropriate notation so that we can apply the subadditivity arguments of \cite{Chen2010}, Chapter 6. In his notation, we will set the values $n_1,n_2,\ldots, n_{\frac{b_n}{t}}$ as $n_i  = i \frac{tn}{b_n}$. We can also define,
    \begin{equation}
    \begin{aligned}
    &\mathcal{F}^x_k(S) = \sum_{i=0}^{tk-1} \sum_{z \in S\left[i \frac{n}{b_n}, (i+1) \frac{n}{b_n} \right]} \mathbb{P}^z(\tau_{S\left[i \frac{n}{b_n}, (i+1) \frac{n}{b_n} \right]} = \infty) \tilde{G}_D(x-z),\\
    & \mathcal{G}^x_k(S) =\sum_{i=0}^{t-1} \sum_{z \in S\left[i \frac{n}{b_n},(i+1) \frac{n}{b_n} \right]} \mathbb{P}^z(\tau_{S\left[i\frac{n}{b_n},(i+1)  \frac{n}{b_n} \right]} = \infty) \tilde{G}_D(x-z).
    \end{aligned}
    \end{equation}
    
    It is clear to see that $\mathcal{F}$ is adapted to $S[0, i \frac{tn}{b_n}]$ and that $\mathcal{G}^{x+c}_k(S+c) = \mathcal{G}^x_k(S)$. Furthermore, we have the subadditivity relation $\mathcal{F}^x_{k+1}(S) = \mathcal{F}^x_k(S) + \mathcal{G}^x_{k+1}(\theta^{n_k} S)$, where $\theta$ is the time shift operator. Thus, we can apply Theorem 6.1.1 of \cite{Chen2010} to assert that for any $u$, we must have that,
    \begin{equation}
    \begin{aligned}
    &\sum_{m=0}^{\infty} \frac{u^m}{ m!} \left\{ \mathbb{E}[ \chi_{b_n}^m]  \right\}^{1/2} \\&\le \prod_{i=0}^{\frac{b_n}{t} -1} \sum_{m=0}^{\infty}  \frac{u^m}{m!} \Bigg\{ \mathbb{E}\Bigg[ \sum_{j_1,j_2=it }^{(i+1)t-1} \sum_{z_1 \in S\left[ j_1 \frac{n}{b_n}, (j_1+1) \frac{n}{b_n} \right]} \sum_{z_2 \in \tilde{S}\left[ j_2 \frac{n}{b_n}, (j_2+1) \frac{n}{b_n}  \right]} \mathbb{P}^{z_1}(\tau_{S\left[ j_1 \frac{n}{b_n}, (j_1+1) \frac{n}{b_n} \right]} = \infty)\\
    &\hspace{2cm} \times \mathbb{P}^{z_2}(\tau_{\tilde{S}\left[ j_2 \frac{n}{b_n}, (j_2+1) \frac{n}{b_n}  \right]} = \infty) G_D(z_1 - z_2)\Bigg]\Bigg\}^{1/2}.
    \end{aligned}
    \end{equation}
    
    Thus, we have,
    \begin{equation}
    \begin{aligned}
    &\frac{1}{b_n} \log \sum_{m=0}^{\infty} \frac{ \theta^m \left(\frac{b_n}{n}\right)^{\frac{m}{4}}}{m!} \left\{ \mathbb{E}[ \chi_{b_n}^m]  \right\}^{1/2}\\
    & \le \frac{1}{t}\log\sum_{m=0}^{\infty}  \frac{\theta^m \left(\frac{b_n}{n}\right)^{\frac{m}{4}}}{m!} \Bigg\{ \mathbb{E}\Bigg[ \sum_{j_1,j_2=0 }^{t-1} \sum_{z_1 \in S\left[ j_1 \frac{n}{b_n}, (j_1+1) \frac{n}{b_n} \right]} \sum_{z_2 \in \tilde{S}\left[ j_2 \frac{n}{b_n}, (j_2+1) \frac{n}{b_n}  \right]} \mathbb{P}^{z_1}(\tau_{S\left[ j_1 \frac{n}{b_n}, (j_1+1) \frac{n}{b_n} \right]} = \infty)\\
    &\hspace{2cm} \times \mathbb{P}^{z_2}(\tau_{\tilde{S}\left[ j_2 \frac{n}{b_n}, (j_2+1) \frac{n}{b_n}  \right]} = \infty) G_D(z_1 - z_2)\Bigg]\Bigg\}^{1/2}.
    \end{aligned}
    \end{equation}
    
    For the latter quantity, we can apply weak convergence from Lemma \ref{lem:weak-conv} to assert that, 
    \begin{equation}
    \begin{aligned}
    &\lim_{n \to \infty}\frac{1}{t}\log\sum_{m=0}^{\infty}  \frac{\theta^m \left(\frac{b_n}{n}\right)^{\frac{m}{4}}}{m!} \Bigg\{ \mathbb{E}\Bigg[ \sum_{j_1,j_2=0 }^{t-1} \sum_{z_1 \in S\left[ j_1 \frac{n}{b_n}, (j_1+1) \frac{n}{b_n} \right]} \sum_{z_2 \in \tilde{S}\left[ j_2 \frac{n}{b_n}, (j_2+1) \frac{n}{b_n}  \right]} \mathbb{P}^{z_1}(\tau_{S\left[ j_1 \frac{n}{b_n}, (j_1+1) \frac{n}{b_n} \right]} = \infty)\\
    &\hspace{2cm} \times \mathbb{P}^{z_2}(\tau_{\tilde{S}\left[ j_2 \frac{n}{b_n}, (j_2+1) \frac{n}{b_n}  \right]} = \infty) G_D(z_1 - z_2)\Bigg]\Bigg\}^{1/2}\\
    &= \frac{1}{t} \log \sum_{m=0}^{\infty} \frac{\theta^m}{m!} t^{\frac{m}{4}}\mathbb{E}\left[\left(5^{5/2} \tilde{\gamma}_S^2 \int_0^1 \int_0^1 G(B_t - \tilde{B}_s) dt ds\right)^m \right]^{1/2}
    \end{aligned}
    \end{equation}
    
    \end{proof}
    
    \subsection{Moment Lower Bounds for Large Deviations}
    
    Using a Feynman--Kac decomposition, we can derive appropriate lower bounds on $\chi_{b_n}$.
    \begin{prop}
    \label{prop:lower-bound}
    Let $b_n=o( n^{1/3} (\log n)^{-8/3})$Recall the definition of $\chi_{b_n}$ from equation \eqref{eq:defchibn}. We have the moment lower bound,
    \begin{equation}
    \begin{aligned}
    &\limsup_{n \to \infty}\frac{1}{b_n} \log \sum_{m=0}^{\infty} \frac{ \theta^m \left(\frac{b_n}{n}\right)^{\frac{m}{4}}}{m!} \left\{ \mathbb{E}[ \chi_{b_n}^m]  \right\}^{1/2} \ge 
    \\& \sup_{f \in L^2(\mathbb{R}^5),g \in \mathcal{F}}\left\{\int_{\mathbb{R}^5} \sqrt{5} \theta \tilde{\gamma}_S  f*\tilde{G}(x) g^2(x) \text{d}x - \frac{1}{10} \int_{\mathbb{R}^5} |\nabla g(x)|^2 \right\},
    \end{aligned}
    \end{equation}
    where $\mathcal{F}$ is the collection of all function such that $\{\int_{\mathbb{R}^5} g(x)^2 \text{d}x=1, \int_{\mathbb{R}^5}|\nabla g(x)|^2 \text{d}x< \infty \}$.
    
    \end{prop}
    
    \begin{proof}
    
    Let $f$ be any function such that $\int_{\mathbb{R}^5} f^2(x) \text{d}x = 1$. We can consider the discretization,
    \begin{equation}
    f_{b_n}(x):= \frac{1}{\sum_{z \in \mathbb{Z}^5} f^2\left( z \sqrt{\frac{b_n}{n}} \right)} f\left( x \sqrt{\frac{b_n}{n}} \right).
    \end{equation}
    One can observe that $\sum_{x \in \mathbb{Z}^5} f^2_{b_n}(x) =1$ and, in addition, we also have that the normalizing factor,
    \begin{equation}
    \sum_{x \in \mathbb{Z}^5} f^2 \left( x \sqrt{\frac{b_n}{n}}\right) \asymp \sqrt{\frac{n^5}{b_n^5}}.
    \end{equation}
    
    Define 
    $$
    \mathcal{F}^x(S):= \sum_{i=0}^{b_n} \sum_{z \in S[i \frac{n}{b_n}, (i+1) \frac{n}{b_n}]} \mathbb{P}^z(\tau_{ S[i \frac{n}{b_n}, (i+1) \frac{n}{b_n}]} = \infty) \tilde{G}_D(x-z).
    $$
    We have that $\chi_{b_n} = \sum_{x \in \mathbb{Z}^5}\mathcal{F}^x(S) \mathcal{F}^x(\tilde{S})$. Notice that we can write,
    $$
    \mathbb{E}[\chi_{b_n}^m] = \sum_{x_1,\ldots,x_m} \mathbb{E}\left[\prod_{i=1}^m \mathcal{F}^{x_i}(S)\right]^2.
    $$
    As such, we can apply the Cauchy--Schwartz inequality to say that,
    \begin{equation}
    \begin{aligned}
    &(\mathbb{E}[\chi^m_{b_n}])^{1/2} = \left( \sum_{x_1,\ldots,x_m} \mathbb{E}\left[\prod_{i=1}^m \mathcal{F}^{x_i}(S)\right]^2\right)^{1/2} \left( \sum_{x_1,\ldots,x_m} \prod_{i=1}^m f_{b_n}(x_i)^2 \right)^{1/2}\\
    & \ge \mathbb{E}\left[\sum_{x} \mathcal{F}^x(S) f_{b_n}(x) \right]^m = \mathbb{E} \left( \sum_{i=0}^{b_n} \sum_{z \in S\left[i\frac{n}{b_n},(i+1) \frac{n}{b_n} \right]}\mathbb{P}^z(\tau_{S\left[i \frac{n}{b_n}, (i+1) \frac{n}{b_n} \right]} = \infty) \tilde{f}_{b_n}(z) \right)^m.
    \end{aligned}
    \end{equation}
    Here,
    $$
    \tilde{f}_{b_n}(z) := \sum_{x \in \mathbb{Z}^5} f_{b_n}(x) \tilde{G}_D(z-x).
    $$
    As such, we see that,
    \begin{equation}\label{eq:Exponential}
    \sum_{m=0}^{\infty} \frac{u^m}{m!} \left(\mathbb{E}[\chi_{b_n}^m]\right)^{1/2} \ge \mathbb{E}\exp\left[ u \sum_{i=0}^{b_n} \sum_{z \in S\left[i\frac{n}{b_n},(i+1) \frac{n}{b_n} \right]}\mathbb{P}^z(\tau_{S\left[i \frac{n}{b_n}, (i+1) \frac{n}{b_n} \right]} = \infty) \tilde{f}_{b_n}(z)  \right].
    \end{equation}
    
    We can introduce the Markov operator $\mathcal{M}^u_{\tilde{f}_{b_n}}$ that acts on functions  $g$ as follows:
    \begin{equation}
    \mathcal{M}^u_{\tilde{f}_{b_n}}g(x):= g\left(x+ S\left(\frac{n}{b_n}\right) \right) \mathbb{E}\exp \left[ u \sum_{z\in S\left[0, \frac{n}{b_n} \right]} \mathbb{P}^z(\tau_{S[0,  \frac{n}{b_n}]} = \infty)  \tilde{f}_{b_n}(z)\right].
    \end{equation}
    
    Roughly speaking, the right-hand side of equation \eqref{eq:Exponential} can be written as $\langle 1, (\mathcal{M}^{u}_{b_n})^{b_n} \delta_0 \rangle.$ (We remark here that repeated applications of the operator $\mathcal{M}^u_{b_n} $ introduce summation over the later times of the random walk inside the exponential.) Since we cannot directly evaluate the inner product on the constant function $1$ or the function $\delta_0$, we do have to make slight adjustments to our analysis.
    
    First, notice that, by Holder's inequality, we have
    \begin{equation}
    \begin{aligned}
    &\log \mathbb{E}\exp\left[ u \sum_{i=0}^{b_n} \sum_{z \in S\left[i\frac{n}{b_n},(i+1) \frac{n}{b_n} \right]}\mathbb{P}^z(\tau_{S\left[i \frac{n}{b_n}, (i+1) \frac{n}{b_n} \right]} = \infty) \tilde{f}_{b_n}(z)  \right] \\
    & \ge (1+\epsilon)\log \mathbb{E}\exp \left[ \frac{u}{1+\epsilon} \sum_{i=1}^{b_n}\sum_{z \in S\left[i\frac{n}{b_n},(i+1) \frac{n}{b_n} \right]}\mathbb{P}^z(\tau_{S\left[i \frac{n}{b_n}, (i+1) \frac{n}{b_n} \right]} = \infty) \tilde{f}_{b_n}(z) \right] \\&- \epsilon \log \mathbb{E} \exp \left[ - \frac{1+\epsilon}{\epsilon} u\sum_{z \in S[0, \frac{n}{b_n}]} \mathbb{P}^z(\tau_{S[0, \frac{n}{b_n}]} = \infty) \tilde{f}_{b_n}(z)\right]
    \end{aligned}
    \end{equation}
    
    We now set $u = \frac{\theta b_n^{1/4}}{n^{1/4}}$, to see that,
    by the analysis of the upper bound of $\chi_{b_n}$ from Proposition \ref{prop:chibnuprbnd}, one can bound,
    \begin{equation}
    \lim_{n \to \infty} \frac{1}{b_n} \log \mathbb{E} \exp \left[- \frac{1+\epsilon}{\epsilon}  \frac{\theta b_n^{1/4}}{n^{1/4}} \sum_{z \in S[0, \frac{n}{b_n}]} \mathbb{P}^z(\tau_{S\left[0, \frac{n}{b_n}\right]} = \infty) \tilde{f}_{b_n}(z)\right]  = 0.
    \end{equation}
    
    We first take $n \to \infty$ and then $\epsilon \to 0$ in order to show,
    \begin{equation}
    \begin{aligned}
    &\liminf_{n \to \infty}  \frac{1}{b_n}\log \mathbb{E}\exp\left[ \frac{\theta b_n^{1/4}}{n^{1/4}} \sum_{i=0}^{b_n} \sum_{z \in S\left[i\frac{n}{b_n},(i+1) \frac{n}{b_n} \right]}\mathbb{P}^z(\tau_{S\left[i \frac{n}{b_n}, (i+1) \frac{n}{b_n} \right]}) \tilde{f}_{b_n}(z) \right]\\
    &\ge \limsup_{\epsilon \to 0} \limsup_{n \to \infty}\frac{(1+\epsilon)}{b_n}\log \mathbb{E}\exp \left[ \frac{\theta b_n^{1/4}}{(1+\epsilon)n^{1/4}} \sum_{i=1}^{b_n}\sum_{z \in S\left[i\frac{n}{b_n},(i+1) \frac{n}{b_n} \right]}\mathbb{P}^z(\tau_{S\left[i \frac{n}{b_n}, (i+1) \frac{n}{b_n} \right]} = \infty) \tilde{f}_{b_n}(z) \right] .
    \end{aligned}
    \end{equation}
    
    Now, notice that for any function $\xi$ that is bounded from above by some constant $C$, we have that 
    \begin{equation}
    \begin{aligned}
    \mathbb{E}\exp \left[ \frac{\theta b_n^{1/4}}{n^{1/4}} \sum_{i=1}^{b_n}\sum_{z \in S\left[i\frac{n}{b_n},(i+1) \frac{n}{b_n} \right]}\mathbb{P}^z(\tau_{S\left[i \frac{n}{b_n}, (i+1) \frac{n}{b_n} \right]} = \infty) \tilde{f}_{b_n}(z) \right] &\ge C^{-2} \Bigg\langle \xi, \left(\mathcal{M}^{\frac{\theta b_n^{1/4}}{n^{1/4}}}_{\tilde{f}_{b_n}}\right)^{b_n -1} \xi\Bigg\rangle\\
    & \ge C^{-2} \bigg\langle \xi,\mathcal{M}^{\frac{\theta b_n^{1/4}}{n^{1/4}}}_{\tilde{f}_{b_n}} \xi \bigg\rangle^{b_n-1} .
    \end{aligned}
    \end{equation}
    Let $g$ be any smooth, bounded function with $\int_{\mathbb{R}^5} g^2(x) \text{d}x =1$, and we define
    $$
    g_{b_n}(x):= \frac{1}{\sum_{z \in \mathbb{Z}^5} g^2(z)} g\left(x \sqrt{\frac{b_n}{n}} \right). 
    $$
    With this choice of $\xi(x)$,we have that,
    \begin{equation}
    \lim_{n \to \infty} \langle g_{b_n}, \mathcal{M}^{\frac{\theta b_n^{1/4}}{n^{1/4}}}_{\tilde{f}_{b_n}} g_{b_n} \rangle =
    \int_{\mathbb{R}^5} g(x)\exp\left[\theta \int_0^1 \sqrt{5} \tilde{\gamma}_S\tilde{f}\left(x + B\left(\frac{t}{5} \right)\right) dt \right] g\left(x + B\left(\frac{1}{5}\right)\right),
    \end{equation}
    by the invariance principle. By the Feynman--Kac formula as in \cite{Chen2010}, equation 4.1.25, we see that if we take the $\log$, then the right-hand side above will be greater than,
    \begin{equation} \label{eq:optimizationprob}
     \int_{\mathbb{R}^5} \sqrt{5} \theta\tilde{\gamma}_S  \tilde{f}(x) g^2(x) \text{d}x - \frac{1}{10} \int_{\mathbb{R}^5} |\nabla g(x)|^2. 
    \end{equation}
     To get an optimal lower bound, one can take the supremum over all $f$ and $g$ satisfying $\int_{\mathbb{R}^5}f^2(x) \text{d}x=1, \int_{\mathbb{R}^5} g^2(x) \text{d}x=1, \int_{\mathbb{R}^5}|\nabla g(x)|^2 \text{d}x < \infty.$
    
    Observe that we have for any function $f$ such that $\int_{\mathbb{R}^5}f^2(x) \text{d}x.$
    \begin{equation}
    \begin{aligned}
    &\int_{\mathbb{R}^5} \tilde{f}(x)g^2(x) \text{d}x = \int_{\mathbb{R}^5} f(y) \left[ \int_{\mathbb{R}^5}\tilde{G}(x-y) g^2(x) \text{d}x \right] \text{d} y \\
    &\le \left(\int_{\mathbb{R}^5} f^2(y) \text{d}y \right)^{1/2} \left(\int_{\mathbb{R}^5} \left[\int_{\mathbb{R}^5}g^2(x) \tilde{G}(y-x)\text{d}x \right]^2 \text{d} y \right)^{1/2}\\
    & \le \left(\int_{\mathbb{R}^5}g^2(x) \tilde{G}(y-x) \tilde{G}(y - \tilde{x}) g^2(\tilde{x}) \text{d}x \text{d}y \text{d} \tilde{x} \right)^{1/2} = \left(\int_{\mathbb{R}^5} g^2(x) G(x-\tilde{x})g^2(\tilde{x})\text{d}x \text{d}\tilde{x} \right)^{1/2}
    \end{aligned}
    \end{equation}
    One can also obtain equality in this expression, by setting $f = \int_{\mathbb{R}^5}\tilde{G}(x-y) g^2(x) \text{d}x.$
    
    As such, we see that the optimization problem in equation \eqref{eq:optimizationprob}, we have that the supremum can be expressed as the supremum of 
    \begin{equation} \label{eq:modopt}
    \sqrt{5}\theta \tilde{\gamma}_S \left[\int_{\mathbb{R}^5} g^2(x)G(x-\tilde{x}) g^2(\tilde{x}) \text{d}x \text{d}\tilde{x} \right]^{1/2} - \frac{1}{10} \int_{\mathbb{R}^5}|\nabla g(x)|^2 \text{d}x,
    \end{equation}
    over all functions $g$ such that $\int_{\mathbb{R}^5}g^2(x) \text{d}x=1$ and $\int_{\mathbb{R}^5} |\nabla g(x)|^2 \text{d}x< \infty$.
    \end{proof}
    
    \subsection{Proof of Theorem~\ref{thm:non-gaussian}}\label{subsec:main-theorem}
    
    Here we prove Theorem~\ref{thm:non-gaussian}. The proof follows the same strategy as \cite{Chen2010}, Theorem 8.4.2. Because of the singularity around zero, we must first pass to a smoothed version of the cross-term defined as follows. Recall that $\tilde{G}(x) = \int_0^{\infty} \frac{1}{\sqrt{\pi t}} p_t(x) dt$ is the convolutional square root of the (continuous) Green's function, and $\tilde{G}^{\epsilon} = p_{\epsilon} \ast \tilde{G}$. Define
    \begin{gather*}
    \label{eq:smooth-def}
    L^{n, \epsilon}(A, x) :=  \left( \frac{b_n}{n} \right)^2 \sum_{y \in A} \mathbb{P}^y(\tau_A = \infty) \tilde{G}^{\epsilon} \left( \sqrt{\frac{b_n}{n}} (y -  x)\right) \\
     \chi^{n, \epsilon} (A, B) := 5 \sum_{x \in \mathbb{Z}^5} L^{n, \epsilon}(A, x) L^{n, \epsilon} (B, x) .
    \end{gather*}
    When obvious, we may omit $n$ and simply write $\chi^{\epsilon}(\mathcal{S}_n, \tilde{\mathcal{S}}_n) = \chi^{n, \epsilon}(\mathcal{S}_n, \tilde{\mathcal{S}}_n)$ and $\chi^{\epsilon}(\mathcal{S}_n, \mathcal{S}_n) = \chi^{n, \epsilon}(\mathcal{S}_n, \mathcal{S}_n)$. Following a similar method as those for the cross-term, we get the following corollary.
    \begin{cor}
        \label{cor:smooth-cap}
        \begin{multline}
        \label{eq:smooth-moment}
            \lim_{n \to \infty} \frac{1}{b_n} \log \mathbb{E} \exp \left\{ \theta \left( \frac{b_n}{n}\right)^{1/4} |\chi^{\epsilon}(\mathcal{S}_n, \mathcal{S}_n)|^{1/2} \right\} \\
            = \sup_{g} \left\{ \sqrt{5} \tilde{\gamma}_S \theta \left[ \int_{\mathbb{R}^5} |(\tilde{G}^{\epsilon} \ast g^2)(x)|^2 \mathrm dx \right]^{1/2} - \frac{1}{10} \int_{\mathbb{R}^5} |\nabla g(x)|^2 \mathrm dx \right\}
        \end{multline}
        Moreover, the right-hand side converges to
        \[  
        \sup_{g \in \mathcal{F}} \left\{ \sqrt{5} \tilde{\gamma}_S \theta \left[ \int_{\mathbb{R}^5} |(\tilde{G}\ast g^2)(x)|^2 \mathrm dx \right]^{1/2} - \frac{1}{10} \int_{\mathbb{R}^5} |\nabla g(x)|^2 \mathrm dx \right\} = \frac{27 \cdot 5^5}{32}\theta^4 \tilde{\gamma}_S^4 \tilde{\kappa}^8(5, 2)
        \]
        from below as $\epsilon \to 0$. This implies
        \begin{equation}
        \label{eq:smooth-LDP}
            \lim_{\epsilon \to 0} \lim_{n \to \infty} \frac{1}{b_n} \mathbb{P} \left\{ \chi^{\epsilon}(\mathcal{S}_n, \mathcal{S}_n) \ge \lambda \sqrt{n b_n^3} \right\} = -I_5(\lambda).
        \end{equation}
    \end{cor}
    \begin{proof}
        The proof of \eqref{eq:smooth-moment} is almost identical to that of Propositions~\ref{prop:chibnuprbnd} and ~\ref{prop:lower-bound}. The second claim is clear from Jensen's inequality and the fact that
        \[  
        \lim_{\epsilon \to 0} \int_{\mathbb{R}^5} |(\tilde{G}^{\epsilon} \ast g^2)(x)|^2 \text{d} x = \int_{\mathbb{R}^5} |(\tilde{G} \ast g^2)(x)|^2 \text{d} x.
        \]
        From the above, \eqref{eq:smooth-LDP} is a direct consequence of Theorem~\ref{thm:gartner-ellis}.
    \end{proof}
    
    Now repeat the decomposition of \eqref{eq:decomposition-L}, except with $\chi^{\epsilon}(\mathcal{S}_n, \mathcal{S}_n)$ instead of $\cp(\mathcal{S}_n)$.
    \begin{equation}
        \label{eq:decompsition-epsilon}
        \chi^{\epsilon}(\mathcal{S}_n, \mathcal{S}_n) = \sum_{j=1}^{2^L} \chi^{n, \epsilon}(\mathcal{S}_n^{(L,j)}, \mathcal{S}_n^{(L, j)}) + \sum_{l=1}^L \Lambda_l^{\epsilon} - \epsilon_L^{\epsilon},
    \end{equation}
    where $\Lambda_l^{\epsilon}$ and $\epsilon_L^{\epsilon}$ stand for
    \begin{gather*}
    \Lambda_l^{\epsilon} := 2\sum_{j=1}^{2^{l-1}} \chi^{n, \epsilon} (\mathcal{S}_n^{(l, 2j-1)}, \mathcal{S}_n^{(l, 2j)}), \quad \epsilon_L^{\epsilon} = \sum_{l=1}^L \sum_{j=1}^{2^{l-1}}  \epsilon^{n, \epsilon} (\mathcal{S}_n^{(l, 2j-1)}, \mathcal{S}_n^{(l, 2j)}), \\
    \epsilon^{\epsilon}(A, B) = \chi^{n, \epsilon}(A, A) + \chi^{n, \epsilon}(B,B) + 2\chi^{n, \epsilon}(A, B) - \chi^{n, \epsilon}(A \cup B, A\cup B).
    \end{gather*}
    This reveals a resemblance between $\chi^{\epsilon}(\mathcal{S}_n, \mathcal{S}_n)$ and $\cp(\mathcal{S}_n)$. Namely, if we sum the decompositions of \eqref{eq:decomposition-L} and \eqref{eq:decompsition-epsilon}, we get
    \begin{equation}
    \label{eq:decompostion-sum}
        \cp(\mathcal{S}_n) + \chi^{\epsilon}(\mathcal{S}_n, \mathcal{S}_n) 
        = \sum_{j=1}^{2^L} \left[ \cp(\mathcal{S}_n^{(L,j)}) + \chi^{\epsilon}(\mathcal{S}_n^{(L,j)}, \mathcal{S}_n^{(L,j)}) \right] + \sum_{l=1}^L (\Lambda_l^{\epsilon} - \Lambda_l) + \epsilon_L - \epsilon_L^{\epsilon}.
    \end{equation}
    In what follows, we show that the deviations of each of the above terms are small when $L$ and $\epsilon$ are large and small, respectively. In this way, the lower tail deviations of $\cp(\mathcal{S}_n)$ can be related to the upper tail deviations of $\chi^{\epsilon}(\mathcal{S}_n, \mathcal{S}_n)$ given by Corollary~\ref{cor:smooth-cap}. In what follows, we prove inequalities concerning each of the terms in \eqref{eq:decompostion-sum}. Assuming such lemmas, Theorem~\ref{thm:non-gaussian} follows from a routine argument taking $\epsilon$ to $0$ and $L$ to infinity.
    
    \begin{cor}
    \label{cor:L-sum-LDP}
    For fixed $L$,
    \begin{equation}
    \label{eq:cap-sum-LDP}
    \limsup_{n \to \infty} \frac{1}{b_n} \log \mathbb{P} \left\{ \left| \sum_{j=1}^{2^L} (\cp(\mathcal{S}_n^{(L, j)}) - \mathbb{E} \cp (\mathcal{S}_n^{(L, j)})) \right| \ge \lambda \sqrt{nb_n^3} \right\} \le - C 2^{L/3} \lambda^{2/3}.
    \end{equation}
    \begin{equation}
    \label{eq:smooth-sum-LDP}
    \limsup_{n \to \infty} \log \mathbb{P} \left\{ \left| \sum_{j=1}^{2^L} \chi^{n,\epsilon}(\mathcal{S}_n^{(L, j)}, \mathcal{S}_n^{(L, j)}) \right| \ge \lambda \sqrt{nb_n^3} \right\} \le - C 2^{L/3} \lambda^{2/3}.
    \end{equation}
    In particular, the constant does not depend on the choice of $\epsilon$.
    \end{cor}
    
    \begin{proof}
    \eqref{eq:cap-sum-LDP} follows from Lemma~\ref{lem:cap-C-LDP} and Lemma~\ref{lem:sum-LDP} below. Similarly, \eqref{eq:smooth-sum-LDP} will follow from Lemma~\ref{lem:sum-LDP} and Corollary~\ref{cor:smooth-cap}. One place of concern is changing from $\chi^{n, \epsilon}$ to $\chi^{2^{-L}n, \epsilon}$. However, since $p_t$ and $\tilde{G}$ are both scale-invariant, this can be done by changing $\epsilon$ by a constant multiple depending on $L$.    
    \end{proof}
    
    \begin{lemma}[\cite{Chen2010}, Theorem 1.2.2]
    \label{lem:sum-LDP}
        Let $Z_1(n), \dots, Z_l(n)$ be independent non-negative random variables with $l \ge 2$ fixed. If there exist constants $C_1 > 0$ and $0 < a \le 1$ such that
        \[  
        \limsup_{n \to \infty} \frac{1}{b_n} \log \mathbb{P}(Z_j (n) \ge \lambda) \le -C \lambda^a \quad \forall \lambda > 0,
        \]
        for $j = 1, \dots , l$, then
        \[  
        \limsup_{n \to \infty} \frac{1}{b_n} \log \mathbb{P}( Z_1 (n) + \dots + Z_l(n) \ge \lambda) \le -C \lambda^a \quad \forall \lambda > 0.
        \]
    \end{lemma}
    
    We turn our attention to the second term, $\sum_{l=1}^L (\Lambda_l^{\epsilon} - \Lambda_l)$. Recall that
    \[  
    \Lambda_l^{\epsilon} - \Lambda_l = 2\sum_{j=1}^{2^{l-1}} \chi^{n,\epsilon}(\mathcal{S}_n^{(l, 2j-1)}, \mathcal{S}_n^{(l, 2j)}) - \chi(\mathcal{S}_n^{(l, 2j-1)}, \mathcal{S}_n^{(l, 2j)}).
    \]
    Since each pair $(\mathcal{S}_n^{(l, 2j-1)}, \mathcal{S}_n^{(l, 2j)})$ may be viewed as two independent random walks of length $2^{n-l}$ starting at $S_{2^{-l}n(2j-1)}$, it suffices to understand the deviations of $\chi^{\epsilon}(\mathcal{S}_n, \tilde{\mathcal{S}}_n) - \chi(\mathcal{S}_n, \tilde{\mathcal{S}}_n)$.
    
    \begin{lemma}
        \label{lem:smooth-cross-term}
        For any $\theta > 0$ and $1 \ll b_n \ll n^{1/3}(\log n)^{-2}$,
        \begin{equation}
            \lim_{\epsilon \to 0^+} \limsup_{n \to \infty} \frac{1}{b_n} \log \mathbb{E} \exp \left\{ \theta \left(\frac{b_n}{n}\right)^{1/4} |\chi(\mathcal{S}_n, \tilde{\mathcal{S}}_n) - \chi^{\epsilon}(\mathcal{S}_n, \tilde{\mathcal{S}}_n)|^{1/2} \right\} = 0
        \end{equation}
    \end{lemma}
    
    \begin{proof}
        The proof is a combination of sub-additivity and weak convergence, similar in spirit to the proof of Proposition~\ref{prop:chibnuprbnd}. For more details, see Lemma 8.2.2 of \cite{Chen2010}, where the show a similar result for the smoothed local times of random walks, or Lemma 8.5.4 of \cite{Chen2010} for the smoothed intersection of the range. For a more direct approach, we can also use the fact that
        \[  
        |G(x) - G^{\epsilon}(x)| = O(\sqrt{\epsilon} |x|^{-5})
        \]
        and explicitly compute the moments in a manner similar to Proposition~\ref{prop:cross-term-moment}.
    \end{proof}
    
    Now we show that $\epsilon_L$ and $\epsilon_L^{\epsilon}$ are negligible.
    \begin{lemma}
    \label{lem:smooth-error}
    For any fixed $L$ and $1 \ll b_n \ll n^{1/3}(\log n)^{-2}$,
        \begin{equation}
        \label{eq:epsilon-L-LDP}
        \lim_{n \to \infty} \frac{1}{b_n} \log \mathbb{P}\left(\epsilon_L \ge \lambda \sqrt{n b_n^3} \right) = - \infty.
        \end{equation}
        \begin{equation}
        \label{eq:smooth-epsilon-L-LDP}
        \lim_{n \to \infty} \frac{1}{b_n} \log \mathbb{P}\left(\epsilon_L^{\epsilon} \ge \lambda \sqrt{n b_n^3} \right) = - \infty.
        \end{equation}
    \end{lemma}
    
    \begin{proof}
        Since $L$ is fixed, \eqref{eq:epsilon-L-LDP} is a simple consequence of Lemma~\ref{lem:error-LDP} and a union bound. Similarly, to show \eqref{eq:smooth-epsilon-L-LDP}, we simply show that $\epsilon^{\epsilon} (\mathcal{S}_n, \tilde{\mathcal{S}}_n)$ is negligible. By methods similar to Lemma~\ref{lem:error-bound} it suffices to compute the moments of
        \begin{equation}  
        \label{eq:smooth-error}
        \sum_{x \in \mathcal{S}_n} \sum_{y \in \tilde{S}_n} \sum_{z \in \mathcal{S}_n \cup \tilde{\mathcal{S}}_n} G_D(x - y) G_D^{\epsilon} (x-z),
        \end{equation}
        where
        \[
        G_D^{\epsilon}(x-y) := 5 \left( \frac{b_n}{n} \right)^{3/2} \sum_{z \in \mathbb{Z}^5} \tilde{G}^{\epsilon}\left( \sqrt{\frac{b_n}{n}} (x-z)\right) \tilde{G}^{\epsilon}\left( \sqrt{\frac{b_n}{n}} (y-z)\right).
        \]
        Since $G_D^{\epsilon}$ is bounded above by $C_{\epsilon}(b_n / n)^{3/2}$ for some $C_{\epsilon}$ depending on $\epsilon$, the above is bounded by a constant (depending on $\epsilon$) multiple of
        \[  
        n^{-1/2} b_n^{3/2} \sum_{i=0}^n \sum_{j=0}^n |S_i - \tilde{S}_j|_+^{-3}.
        \]
        By Lemma~\ref{prop:cross-term-moment}, the $m$-th moment of the above is bounded by $C_{\epsilon}^m (m!)^{3/2} b_n^{3m/2}$. An argument similar to Lemma~\ref{lem:error-LDP} completes the proof.
    \end{proof}
    
    \begin{proof}[Proof of Theorem~\ref{thm:non-gaussian}]
        By the above,
        \[
        \lim_{\epsilon \to 0} \lim_{n \to \infty} \frac{1}{b_n} \log \mathbb{P} \left\{ | \cp(\mathcal{S}_n) + \chi^{\epsilon}(\mathcal{S}_n, \mathcal{S}_n) - \mathbb{E} \cp(\mathcal{S}_n)| \ge \delta \sqrt{nb_n^3} \right\} \le -C2^{L/2}
        \]
        for any $\delta > 0$. Since this holds for any $L$, we have
        \[
        \lim_{\epsilon \to 0} \lim_{n \to \infty} \frac{1}{b_n} \log \mathbb{P} \left\{ | \cp(\mathcal{S}_n) + \chi^{\epsilon}(\mathcal{S}_n, \mathcal{S}_n) - \mathbb{E} \cp(\mathcal{S}_n)| \ge \delta \sqrt{nb_n^3} \right\} = - \infty.
        \]
        In other words,
        \[
            \lim_{n \to \infty} \frac{1}{b_n} \mathbb{P}\left\{ \cp(\mathcal{S}_n) - \mathbb{E} \cp(\mathcal{S}_n) \le \lambda \sqrt{nb_n^3} \right\} 
            = \lim_{\epsilon \to 0} \lim_{n \to \infty} \frac{1}{b_n} \mathbb{P}\left\{ \chi^{\epsilon} (\mathcal{S}_n, \mathcal{S}_n) \ge \lambda \sqrt{nb_n^3} \right\} = - I_5(\lambda)
        \]
        by Corollary~\ref{cor:smooth-cap}.
    \end{proof}
    
    \section{Application to $d = 4$}\label{sec:4d}
    
    In this section, we explain how our techniques can be used to improve the results of the first author and Okada \cite{AdhikariOkada2023} for the random walk on $\mathbb{Z}^4$, where they showed Theorem~\ref{thm:dim-4} in the range $1 \ll b_n =O( \log \log n)$. The key improvement is the following lemma.
    \begin{lemma}
        \label{prop:error-term-moment-4d}
        Let $S$ and $\tilde{S}$ be independent random walks on $\mathbb{Z}^5$ starting from the origin. For any positive integers $n$ and $m$,
        \begin{equation}
        \label{eq:error-term-moment-4d}
        \mathbb{E} \left[ \sum_{1 \le i_1 , \dots , i_m \le n} \sum_{1 \le j_1 , \dots , j_{2m} \le n} \prod_{k = 1}^{m} |S_{i_k} - \tilde{S}_{j_{2k-1}}|_+^2 |S_{i_k} - \tilde{S}_{j_{2k}}|_+^2\right] \le C^m (m!)^2 n^{m}.
        \end{equation}
    \end{lemma}
    
    \begin{proof}
    The proof is identical to that of Proposition~\ref{prop:error-term-moment}, except that all of the edge weights need to be scaled by $2/3$. This means that the summation changes to
    \begin{align*}
    \mathbb{E} &\left[ \sum_{1 \le i_1 , \dots , i_m \le n}  \sum_{1 \le j_1 , \dots , j_{2m} \le n} \prod_{k = 1}^{m} |S_{i_k} - \tilde{S}_{j_{2k-1}}|_+^2 |S_{i_k} - \tilde{S}_{j_{2k}}|_+^2 \right] \\
    &= C^{m} (m!)(2m)! \left( \prod_{k=1}^{2m} \sum_{j_k = j_{k-1}}^n |j_k - j_{k-1}|_+^{-2/3} \right) \left( \prod_{k=1}^{m} \sum_{i_k = i_{k-1}}^n |i_k - i_{k-1}|_+^{-2/3} \right) \\
    &\le C^m (m!)^3 n^m \left( \int\limits_{0 \le x_1 \le \dots \le x_{2m} \le 1}\frac{dx_{2m} \dots dx_1}{\prod\limits_{k=1}^{2p} |x_k - x_{k-1}|^{2/3}} \right) \left( \int\limits_{0 \le x_1 \le \dots \le x_{m} \le 1} \frac{dx_{m} \dots dx_1}{\prod\limits_{k=1}^{m} |x_k - x_{k-1}|^{2/3}} \right)\\
    &\le C^m (m!)^2 n^m.
    \end{align*}
    The asymptotics for the integral can be found in \cite{Kono1977}.

    \end{proof}
    
    As a consequence of this estimate, we can prove the following result for the moments of the error term $MT_n$ that appear in \cite{AdhikariOkada2023}[Equation 4.7]. 
    
    \begin{lemma} \label{lem:MTn}
    For $m  \le \log n$, there exists some constant $C$ (not depending on m) such that,
    \begin{equation}
    \mathbb{E}[MT_n^m] \le C^m (m!)^2 \frac{n^m}{(\log n)^{3m}}.
    \end{equation}
    
    \end{lemma}
    
    We start with a corollary of this lemma before proving this,
    \begin{theorem} \label{thm:FinalEstMtn}
    If $b_n$ is any sequence such that $b_n = o(\log n)$, then we have that,
    \begin{equation}
    \frac{1}{b_n} \log \mathbb{P}\left(MT_n \ge \epsilon b_n \frac{n}{(\log n)^2} \right) = - \infty.
    \end{equation}
    \end{theorem}
    \begin{proof}
    By Markov's Theorem, we have that,
    \begin{equation}
    \begin{aligned}
    &\frac{1}{b_n} \log \mathbb{P}\left(MT_n \ge \epsilon b_n \frac{n}{(\log n)^2} \right) \le \frac{1}{b_n} \log \left[ \frac{\mathbb{E}[MT_n^m](\log n)^{2m}}{\epsilon^m b_n^m n^m} \right] \le \frac{1}{b_n} \log \left[ \frac{C^m (m!)^2}{\epsilon^m b_n^m (\log n)^m} \right]\\
    & \le \frac{1}{b_n} \log \left[ \frac{C^m m^{2m}}{\epsilon^m b_n^m (\log n)^m} \right] = \frac{m}{b_n} \log \left[ \frac{C m^2}{\epsilon b_n \log n} \right]
    \end{aligned}
    \end{equation}
    If we choose a moment $m$ such that $m \gg b_n$, but $m^2 \ll b_n \log n$, then as we take $n \to \infty$ the right hand side above will go to 0 as $n \to \infty$.  If $b_n$ is some sequence that is $o(\log n)$, then we could choose $m = (\log n)^{1/3} b_n^{2/3}$. Then, both of the conditions above are satisfied.
    
    \end{proof}
    
    Now, we can turn to the proof of Lemma \ref{lem:MTn}.
    \begin{proof}
    
    We introduce the two parameters $\gamma_1 < \gamma_2 <1$ whose value will be specified later.
    We first remark that the term $MT_n$ from \cite{AdhikariOkada2023}[Equation 4.7] can be bounded from above by the following term,
    \begin{equation}
    \begin{aligned}
    B_n:=&\sum_{i_1,i_2=1}^n \sum_{\tilde{i}=1}^n \mathbb{P}(S[i_1 - n^{\gamma_1}, i_1 + n^{\gamma_1} ] \cap R'_{S_{i_1}} = \emptyset) \mathbb{P}(S[i_2 - n^{\gamma_1}, i_2 + n^{\gamma_1} ] \cap R'_{S_{i_2}} = \emptyset)  \\
    &G(S_{i_1}- \tilde{S}_{\tilde{i}}) G(S_{i_2}- \tilde{S}_{\tilde{i}}) \mathbb{P}(\tilde{S}[\tilde{i} - n^{\gamma_1}, \tilde{i} + n^{\gamma_1} ] \cap R'_{\tilde{S}_{\tilde{i}}} = \emptyset) .
    \end{aligned}
    \end{equation}
    
    The main difference in the upper bound above is that the probability term $\mathbb{P}(S[i_1 - n^{\gamma_1}, i_1 + n^{\gamma_1} ] \cap R'_{S_{i_1}} = \emptyset) $ is a non-intersection event that only correlates a smaller region of the random walk. This makes it more likely that when we take a small number of moments up to order $\log n$, we do not have to worry about this probability term creating correlations. Instead, we could just replace it with its expectation, which would be of order $\frac{1}{\log n^{\gamma}}$. Note that if we took more than $n$ moments, this heuristic would be impossible to apply. Thus, it is important in the course of the proof that we will take only $\log n$ moments.
    
    We will compute $\mathbb{E}[BT_n^m]$ via an induction in the parameter $m$. The $m$th moment of $BT_n$ will look like
    $$
    \begin{aligned}
    &BT_n^m = \sum_{i_1^1,i_2^1,\ldots, i_1^m, i_2^m}^n \sum_{\tilde{i}^1,\ldots,\tilde{i}^m=1}^n \prod_{j=1}^m \mathbb{P}(S[i^m_1 - n^{\gamma_1}, i^m_1 + n^{\gamma_1} ] \cap R'_{S_{i^m_1}} = \emptyset)\\ &\mathbb{P}(S[i^m_2 - n^{\gamma_1}, i^m_2 + n^{\gamma_1} ] \cap R'_{S_{i^m_2}} = \emptyset) 
     G(S_{i^m_1}- \tilde{S}_{\tilde{i}^m}) G(S_{i^m_2}- \tilde{S}_{\tilde{i}^m}) \\&\mathbb{P}(\tilde{S}[\tilde{i}^m - n^{\gamma_1}, \tilde{i}^m + n^{\gamma_1} ] \cap R'_{\tilde{S}_{\tilde{i}^m}} = \emptyset) .
    \end{aligned}
    $$
    
    \textit{Part 1: All indices are sufficiently far apart }
    
    If it were true that all of the terms $|\tilde{i}^j 
     - \tilde{i}^k| \ge n^{\gamma_2}$ $\forall j,k$ and $|i^k_l - i^j_m| \ge n^{\gamma_2}$ $\forall k,j$ and $\forall l,m \in \{1,2\}$, then we could readily use Lemma \ref{prop:error-term-moment-4d} and the method of Claim 5.8 of \cite{AdhikariOkada2023}. Just to summarize the method of Claim 5.8 of \cite{AdhikariOkada2023}, the main idea is that one could explicitly write out $\mathbb{E}[BT_n^m]$ when all of the distances between the $\tilde{i}$'s and $i$'s are relatively large by conditioning on the values of the random walk at the points $S_{\tilde{i}^j}, S_{\tilde{i}^j \pm n^{\gamma_1}}$ and $S_{i^j_k}, S_{i^j_k \pm n^{\gamma_1}}$ and writing out explicitly the appropriate transition kernel between the points. We need to introduce the `boundary' points $S_{\tilde{i}^j \pm n^{\gamma_1}}$ in order to isolate the probability term  $\mathbb{P}(\tilde{S}[\tilde{i}^m - n^{\gamma_1}, \tilde{i}^m + n^{\gamma_1} ] \cap R'_{\tilde{S}_{\tilde{i}^m}} = \emptyset)$. 
     
     Let $\tilde{i}^k$ and $\tilde{i}^j$ be elements in the ordering of the $\tilde{i}^1,\ldots, \tilde{i}^m$ such that $\tilde{i}^k$ and $\tilde{i}^j$ are immediately next to each other. Let $p_t(x)$ also denote the transition kernel of $t$ steps of the random walk. Then, the transition probability between boundary points like $p_{\tilde{i}^k - \tilde{i}^j - 2 n^{\gamma_1}}(\tilde{S}_{\tilde{i}^j + n^{\gamma_1}} - \tilde{S}_{\tilde{i}^k - n^{\gamma_1}})$  can be replaced by $p_{\tilde{i}^k - \tilde{i}^j}(\tilde{S}_{\tilde{i}^j} - \tilde{S}_{\tilde{i}^k})$. The only remaining terms that depend on the points $\tilde{S}_{\tilde{i}^j \pm n^{\gamma_1}}$ will be the product
     $$
     \mathbb{P}(\tilde{S}[\tilde{i}^j - n^{\gamma_1}, \tilde{i}^j + n^{\gamma_1} ] \cap R'_{\tilde{S}_{\tilde{i}^j}} = \emptyset) p_{n^{\gamma_1}}(\tilde{S}_{\tilde{i}^j} - \tilde{S}_{\tilde{i}^j + n^{\gamma_1}})  p_{n^{\gamma_1}}(\tilde{S}_{\tilde{i}^j} - \tilde{S}_{\tilde{i}^j - n^{\gamma_1}}).
     $$
    
    For this term one can first take the expectation over all of the terms $\tilde{S}_{\tilde{i}^j \pm n^{\gamma_1}}$. The expectation of the term above would be of order $\frac{1}{\log n^{\gamma}}$. Now, after taking the expectation over all of the boundary points, the remaining moment is just the $m$th moment appearing in Lemma \ref{prop:error-term-moment-4d} multiplied by a term of the form $\frac{C^k}{(\log n)^{3k}}$ for some constant $C$.
    
    \textit{Part 2: Some indices are close together}
    
    It remains to try to understand what happens when there are elements $\tilde{i}^k$ and $\tilde{i}^j$ (or elements $i^k_{l_1}$ and $i^j_{l_2}$) that are within distance $n^{\gamma_2}$ of each other. We will first analyze the case in which there are two elements $\tilde{i}^k$ and $\tilde{i}^j$ that are within distance $n^{\gamma_2}$ of each other. Note, that there are $m^2$ ways to pick the indices $j$ and $k$ for the pair $\tilde{i}^j$ and $\tilde{i}^k$. Without loss of generality, we may assume that this pair is $\tilde{i}^1$ and $\tilde{i}^2$. In what proceeds, we will try to sum over the all of the quantities involving the indices $\tilde{i}^1, i^1_1, i^1_2$ (and keeping the other indices fixed at first). 
    
    Let us first take the average of $S_{i^1_1}$ after first fixing all other indices and the value of $S$ at these indices. This index $i^1_1$ must lie in between $i^{k_1}_{l_1}< i^1_1 < i^{k_2}_{l_2}$ for some indices $i^{k_1}_{l_1}$ and $i^{k_2}_{l_2}$. We remark that one of these indices could possibly be $i^{1}_2$.  We can bound all the probability terms involving the indices $i^1_1$,$i^1_2$ and $\tilde{i}^1$ by 1 and then, by Lemma A.15 of \cite{AdhikariOkada2023}, we get that
    $$
    \mathbb{E}\left[\sum_{i^1_l \in (i^{k_1}_{l_1}, i^{k_2}_{l_2})}G(S_{i^1_1} - \tilde{S}_{\tilde{i}^1})\bigg| \tilde{S}_{i_1}, S_{i^{k_1}_{l_1}}, S_{i^{k_2}_{l_2}}\right] \le O(\log n)
    $$
    
    Now, all of the $i$ indices excluding $i^1_1$ will subdivide $1\ldots n$  into $2m$ distinct intervals. Thus, the multiplicative factor coming from summing over the index $i^1_1$ and averaging the random walk interval involving $S_{i^1_1}$ when all other indices are fixed will be given by $O(m \log n)$. After this step, we can use similar reasoning to sum up over the index $i^1_2$ and averaging over the random walk interval involving $S_{i^1_2}$ would give us another factor of $O(m\log n)$. Finally, we can sum up over the index $\tilde{i}^1$, of which there will be $O(n^{\gamma_2})$ choices. Thus, in total, we see that the error term accrued when there exists at least one pair of indices just that $|\tilde{i}^j - \tilde{i}^k| \le 2 n^{\gamma_2}$ is bounded by,
    $
    O(m^3 n^{\gamma_2} (\log n)^2) \mathbb{E}[BT^{m-1}_n].
    $
    
    Now, consider what happens when there is some pair of indices $i^{j_1}_{k_1}$ and $i^{j_2}_{k_2}$ that are within distance $n^{\gamma_2}$ of each other. There are two cases the consider: the first is that the indices are $i^{j}_1$ and $i^{j}_2$ for the same superscript index $j$, and the other is that the superscript indices satisfy $j_1 \ne j_2$. The analysis when the superscript indices do not match is largely similar to what has been done previously. We remark that there are $O(m^2)$ ways to choose a pair $i^{j_1}_{k_1}$ and $i^{j_2}_{k_2}$. Without loss of generality, we can assume these indices are $i^1_1$ and $i^2_1$. At this point, we want to sum up over the values $i^1_1,i^1_2,\tilde{i}^1$ as well as average over the appropriate segment of the random walk involving these points; we also bound all probability terms involving $i^1_1,i^1_2$ and $\tilde{i}^1$ by 1. The first step is to sum up over $i^1_2$ and average over the segment involving $i^1_2$ while fixing all other indices. There will be $2m$ ways to choose the segment involving $i^1_2$ (say we have $i^{j_1}_{k_1}<i^1_2< i^{j_2}_{k_2}$. Then, we have that
    $$
    \mathbb{E}\left[\sum_{i^{j_1}_{k_1}<i^1_2 < i^{j_2}_{k_2}}G(\tilde{S}_{\tilde{i}^1} - S_{i^1_2})\bigg|S_{i^{j_1}_{k_1}}, S_{i^{j_2}_{k_2}}, \tilde{S}_{i^{\tilde{1}}}\right] \le O(\log n).
    $$
    After this point, we can perform the sum over $\tilde{i}^1$ and average over the walk containing $\tilde{S}_{\tilde{i}^1}$; there are $m$ subintervals for the choice of the index $\tilde{i}^1$ (say $\tilde{i}^{j_1} < \tilde{i}^1 < \tilde{i}^{j_2}$). Then, we again have
    $$
    \mathbb{E}\left[\sum_{\tilde{i}^{j_1}<\tilde{i}^1 < \tilde{i}^{j_2}}G(\tilde{S}_{\tilde{i}^1} - S_{i^1_1})\bigg|\tilde{S}_{\tilde{i}^{j_1}, \tilde{S}_{\tilde{i}^{j_2}}, \tilde{S}_{i^{\tilde{1}}}}\right] \le O(\log n).
    $$
    Finally, we can sum up over the remaining choices of $i^1_1$, (of which there will be $O(n^{\gamma_2})$ of them since $i^1_1$ will be within distance $n^{\gamma_2}$ of $i^2_1$. 
    
    Thus, the total contribution from those indices such that $|i^{j_1}_{k_1} - i^{j_2}_{k_2}| \le n^{\gamma_2}$ is bounded by $
    O(m^3 n^{\gamma_2} (\log n)^2) \mathbb{E}[BT^{m-1}_n].
    $
    
    The final case that we have to consider is then $i^j_1$ and $i^j_2$ are within distance $n^{\gamma_2}$ of each other. We may assume that the two indices are $i^1_1$ and $i^1_2$. This implies that $|S_{i^1_1} -S_{i^1_2}| \le n^{\gamma_2}$ and thus, $G(S_{i^1_2}- \tilde{S}_{\tilde{i}^1})  \le n^{6 \gamma_2} G(S_{i^1_1}- \tilde{S}_{\tilde{i}^2})$. Let us introduce the quantity
    \begin{equation}\label{eq:defZn}
    Z_n:= \sum_{i=1}^n \sum_{j=1}^n G(S_i -\tilde{S}_j)^2.
    \end{equation}
    Under the condition that $i^j_1 - i^j_2$ are within distance $n^{\gamma_2}$, we can bound the moment under consideration by
    $E[Z_n BT^{m-1}_n]\le  \mathbb{E}[Z_n^m]^{1/m} \mathbb{E}[BT_n^m]^{(m-1)/m}$. Now, one can show by methods similar to Proposition~\ref{prop:cross-term-moment} (multiply every edge weight by $4/3$) that $\mathbb{E}[Z_n^m] \le C^m (m!)^2 (\log n)^{2m}$ for some constant $C$.
    
    \textit{Part 3: Combining all estimates}
    
    If we combine all of our estimates, ultimately we can derive the recursive bound,
    \begin{equation} \label{eq:recurbnd}
    \begin{aligned}
    \mathbb{E}[BT_n^m] &\le C^m (m!)^2 \frac{n^m}{(\log n)^{3m}} +
    Cm^3 n^{\gamma_2}(\log n)^2 \mathbb{E}[BT_n^{m-1}] \\ &+ m n^{6 \gamma_2} (C^m (m!)^2 (\log n)^{2m})^{1/m} (\mathbb{E}[BT^m_n])^{(m-1)/m}.
    \end{aligned}
    \end{equation}
    
    Now, we will show by induction that there exists some universal constant $B$ such that $\mathbb{E}[BT^m_n] \le B^m(m!)^2 \frac{n^m}{(\log n)^{3m}}$ and for all $m \le (\log n)^2$.
    
    Assuming this bound holds for $k \le m-1$, to now show this for $k=m$, we first observe that $ x - m n^{6 \gamma_2} (C^m(m!)^2 (\log n)^{2m})^{1/m} x^{(m-1)/m}$ has derivative $1 - (m-1) n^{6 \gamma_2} (C^m (m!)^2 (\log n)^{2m})^{1/m} x^{-1/m}$. If $x \ge B^m (m!)^2 \frac{n^m}{(\log n)^{3m}},$ then we see that the derivative is positive provided $6 \gamma_2 <1$, $C<B$ and $m \le (\log n)^2$ for all sufficiently large $n$. Thus the difference between the two sides of equation \eqref{eq:recurbnd} is smallest when $\mathbb{E}[BT^m_n] = B^m (m!)^2 \frac{n^m}{(\log n)^{3m}}.$ However, one can directly see that when $6 \gamma_2 <1$, $m \le (\log n)^2$ and $B>C$, then, for all $n$ sufficiently large, the right hand side of equation \eqref{eq:recurbnd} is much smaller than $B^m (m!)^2 \frac{n^m}{(\log n)^{3m}}$. Thus, we must have that $\mathbb{E}[BT^m_n] \le B^m (m!)^2 \frac{n^m}{(\log n)^{3m}}$ and this completes the induction.

    \end{proof}
    
    Now, we can describe a sketch of how Lemma \ref{lem:MTn} improves the result of \cite{AdhikariOkada2023} to derive results all the way up to the critical threshold of $b_n \ll \log n$.
    
    \begin{proof}[Proof of Theorem \ref{thm:dim-4} ]
    For the full details of the proof, we will refer the reader to the paper \cite{AdhikariOkada2023}; we will only mention here how our results improve the analysis of the aforementioned paper.  In section 5 of \cite{AdhikariOkada2023}, the key term that was used to determine the large deviation statistics was the term $TL'_n$.  The analysis of $TL'_n$ was based upon using a subadditivity argument combined with a weak convergence argument for the upper bound, and a Feynman--Kac formula argument for the lower bound. The argument analyzing $TL'_n$ is robust and works all the way to $b_n \ll \log n$. Thus, we only have to check the steps between the reduction between the original cross term $\chi$ and the final analyzed term $TL'_n$.
    
    A proof similar to Lemma \ref{lem:aux-error-moment} demonstrates that the error incurred upon the reduction from $TL_n$ to $TL'_n$ is similar to the error incurred from the reduction of $\chi$ to $TL_n$. This is exactly the term $MT_n$ from Lemma 4.1 of \cite{AdhikariOkada2023}. Lemma 4.1 of \cite{AdhikariOkada2023} was only able to compute finite moments of this error term $MT_n$ efficiently. Thus, when it came to applying Markov's inequality in Corollary 4.2 to control this error term, it was not possible to go above $b_n = O(\log \log n)$. The improved error term from Theorem \ref{thm:FinalEstMtn} will allow us to go to $b_n \ll \frac{\log n}{\log \log n}$. 
    
    \end{proof}
    
    \section*{Acknowledgments}
    We thank Amir Dembo and Izumi Okada for many helpful discussions. We thank Persi Diaconis for suggesting the reference \cite{BenderRichmond1984}. We thank Yanxin Zhou for pointing out some typos in a previous version of this paper. Research started under the support of NSF grant DMS 2102842 (A.A) Research partly funded by NSF grant DMS-2348142 (J.P.) This work was supported by a grant from the Simons Foundation International [SFI-MPS-SDF-00014916, J.P.]

    \end{document}